\documentclass[letter,11pt]{amsart}
\usepackage[latin1]{inputenc}   
\usepackage{amssymb,amsmath,amsthm,mathrsfs,mathdots}
\usepackage{graphicx,mathpazo}
\usepackage[all]{xy}
\usepackage[usenames,dvipsnames]{color}
\usepackage{enumerate}
\usepackage{stmaryrd}

\usepackage{adjustbox}

\usepackage{youngtab,ytableau}
\ytableausetup{centertableaux,boxsize=0.25em}

\usepackage[colorlinks=true,linkcolor=Cerulean,pagebackref,citecolor=Mahogany]{hyperref}

\definecolor{farbe}{RGB}{121,60,0}
\definecolor{Farbe}{RGB}{139,0,139}

 \usepackage{tikz}
\usetikzlibrary{arrows,decorations.pathmorphing,decorations.pathreplacing,positioning,shapes.geometric,shapes.misc,decorations.markings,decorations.fractals,calc,patterns,}
\usepackage{tikz-cd}

\tikzset{>=stealth',
     cvertex/.style={circle,draw=black,inner sep=1pt,outer sep=3pt},
     vertex/.style={circle,fill=black,inner sep=1pt,outer sep=3pt},
     star/.style={circle,fill=yellow,inner sep=0.75pt,outer sep=0.75pt},
     tvertex/.style={inner sep=1pt,font=\criptsize},
     gap/.style={inner sep=0.5pt,fill=white}}

\newcommand{\PP}{\ensuremath{\mathbb{P}}}
\newcommand{\RR}{\ensuremath{\mathbb{R}}}
\newcommand{\ZZ}{\ensuremath{\mathbb{Z}}}
\newcommand{\CC}{\ensuremath{\mathbb{C}}}

\newcommand{\QQ}{\ensuremath{\mathbb{Q}}}

\newcommand{\Proj}{\ensuremath{\mathbb{P}}}

\newcommand{\D}{\ensuremath{\partial}}
\newcommand{\F}{\ensuremath{\mathcal{F}}} %mathcal R

\DeclareMathOperator{\add}{add}
\DeclareMathOperator{\Bl}{Bl}

\DeclareMathOperator{\Ext}{Ext}
\DeclareMathOperator{\Hom}{Hom}
\DeclareMathOperator{\mmod}{mod}

\DeclareMathOperator{\Spec}{Spec}

\newcommand{\reg}{{\rm reg}}

\newcommand{\lb}{\llbracket}
\newcommand{\rb}{\rrbracket}

\newcommand{\mc}[1]{\ensuremath{\mathcal{#1}}}   %mathcal
\newcommand{\ms}[1]{\ensuremath{\mathscr{#1}}}   %mathscr
\newcommand{\mf}[1]{\ensuremath{\mathfrak{#1}}}   %mathfrak

\newcommand{\col}[1]{{\color{lightgray} #1}}
\newcommand{\rot}[1]{{\color{red} #1}}

\newcommand{\cF}{\mathcal{F}}

\newcounter{CountAlpha}

\theoremstyle{theorem}

\newtheorem{MainThm}[CountAlpha]{Theorem}  % fÃÂÃÂÃÂÃÂÃÂÃÂÃÂÃÂÃÂÃÂÃÂÃÂÃÂÃÂÃÂÃÂr Thm A, etc.

\newtheorem{Thm}{Theorem}[section]         %f"ur Satz1, 2 ,3, etc.
\newtheorem{lemma}[Thm]{Lemma}
\newtheorem{cor}[Thm]{Corollary}
\newtheorem{corollary}[Thm]{Corollary}

\newtheorem{proposition}[Thm]{Proposition} 

\newtheorem{Qu}[Thm]{Question}
\newtheorem{lem-def}[Thm]{Lemma-Definition}

\theoremstyle{definition}
\newtheorem{defi}[Thm]{Definition} %neu

\newtheorem{example}[Thm]{Example}
\newtheorem{Bem}[Thm]{Remark}

\newtheorem{Not}[Thm]{Notation}
\newtheorem{DefCon}[Thm]{Definition/Construction}

\numberwithin{equation}{subsection}

\title[Friezes and resolutions]{Frieze patterns and combinatorics of curve singularities}

\author{Eleonore Faber}
\address{
Institut f\"ur Mathematik und Wissenschaftliches Rechnen,
Universit\"at Graz,
Heinrichstr.~36,
A-8010 Graz, Austria and School of Mathematics, University of Leeds, LS2 9JT Leeds, UK
}
\email{e.m.faber@leeds.ac.uk, eleonore.faber@uni-graz.at}

\author{Bernd Schober}
\address{
	None. (Hamburg, Germany).
}
\email{schober.math@gmail.com}

\date{\today}

\makeatletter
\@namedef{subjclassname@2020}{\textup{2020} Mathematics Subject Classification}
\makeatother



\begin{document}

\maketitle

\begin{abstract}
We study the connection between Conway--Coxeter frieze patterns and the data of the minimal resolution of a complex curve singularity: using Popescu-Pampu's notion of the lotus of a singularity, we describe a bijection between the dual resolution graphs of Newton non-degenerate plane curve singularities and Conway--Coxeter friezes. We use representation theoretic reduction methods to interpret some of the entries of the frieze coming from the partial resolutions of the corresponding curve singularity. Finally, we translate the notion of mutation, coming from cluster combinatorics, to resolutions of plane complex curves.
\end{abstract}

%\tableofcontents

\section{Introduction}

This paper is concerned with uncovering relations between two at first sight very different topics: frieze patterns with positive integer entries on the one hand, and the resolution of singularities of complex plane curves on the other hand. 

\emph{Friezes} are arrays of numbers consisting of a finite number of infinite rows and are usually written in an offset fashion 

\adjustbox{scale=0.85,center}{
\begin{tikzcd}[row sep=0.35em, column sep=-0.35em]
 &\ldots && 0 && 0 && 0 && 0 && \ldots &\\
   && 1 && 1 && 1 && 1 && 1 && \\
  &\ldots && p_{-2,0} && p_{-1,1} && p_{0,2} && p_{1,3} && \ldots& \\
   && p_{-3,0} && p_{-2,1} && p_{-1,2} && p_{0,3} && p_{1,4} && \ldots \\
     &\ldots && p_{-3,1} && p_{-2,2} && p_{-1,3} && p_{0,4} && p_{1,5} & & \ldots \\
     &   & \ldots && \ldots && \ldots && \ldots && \ldots && \ldots \\
   & \ldots && \ldots && \ldots && p_{-2,w-1} && p_{-1,w} && p_{0,w+1} && \ldots \\
   && 1 && 1 && 1 && 1 && 1 && \ldots \\
   &\ldots && 0 && 0 && 0 && 0 && \ldots &
\end{tikzcd}
}

such that the first and last row consist of $0$s, the second row and the penultimate row consist of $1$s and any four entries arranged in the form
$$
\begin{tabular}{ccc}
& $b$ & \\
$a$ & & $d$ \\
& $c$ &
\end{tabular}
$$
satisfy the condition $ad-bc=1$. Here $w$ is called the \emph{width} of the frieze. Such friezes were first considered by Coxeter and then studied by Conway and Coxeter in the 1970s \cite{Coxeter,CoCo1,CoCo2}. In particular, such friezes are always determined by the entries of the first nontrivial row (with entries $p_{i-1,i+1}$), the so-called \emph{quiddity row}. The entries of the quiddity row and consequently all rows of the frieze are periodic with period $w+3$ in the horizontal direction. Thus, the \emph{quiddity sequence} $\{p_{i-1,i+1}\}_{i=1}^{w+3}$ determines the frieze. Further, Conway and Coxeter showed a bijection of friezes with positive integer entries of width $w$ and triangulated polygons with $w+3$ vertices: the quiddity sequence is simply given as associating each vertex $i$ of the triangulated polygon the number $p_{i-1,i+1}$ of triangles incident to $i$ (we recall this in Section \ref{Sec:continuedfractions-friezes}). We will call friezes with positive integer entries \emph{Conway-Coxeter friezes} (or: \emph{CC-friezes} for short). \\
In the 2000s these combinatorial objects gained more interest following the introduction of cluster algebras and cluster categories. Cluster algebras were discovered by Fomin and Zelevinsky \cite{FZ2002} in the context of Lusztig's dual canonical basis and total positivity and their categorification is a very active topic of research in representation theory, see e.g.~\cite{BMRRT06, GLS-ClusterChar, JKS16}. A cluster algebra is constructed from a set of generators (so-called \emph{cluster variables} forming the initial cluster) where more cluster variables are determined recursively through a process called \emph{mutation}, which can be described in terms of matrices or quivers (the latter under some mild conditions on the cluster algebra). \\
 In particular, it was shown  that CC-friezes can be obtained by specializing all cluster variables of a given cluster of a type $A_n$ cluster algebra 
 to 1s \cite{CalderoChapoton}.  More recently, friezes have received considerable attention from the point of view of (Grassmannian) cluster categories, see e.g.~\cite{BaurIsfahan, BFGST2}, and moreover, a mutation rule for friezes has been established in \cite{BFGST18}. See also \cite{Crossroads}, which surveys several research directions involving friezes.

On the other hand, \emph{resolutions of  plane curve singularities} over $\CC$ are a classical topic in algebraic geometry, and their study dates back to Newton, for some approaches see e.g.~\cite{Cut2004, KollarResolutionBook}. The data of a resolution of a curve singularity can be encoded in a weighted graph, the so-called \emph{dual resolution graph}, whose vertices correspond to the exceptional divisors of the resolution and the weights are given by their self-intersection numbers. There are several other singularity invariants encoding data of the resolution, i.e., the Enriques diagram and the Eggers--Wall tree. Garc\'ia Barroso, Gonz\'alez P\'erez and Popescu-Pampu introduced in \cite{Lotus} the notion of a lotus of a curve singularity, which is a certain simplicial complex, with the purpose of being able to read off all of these invariants. That notion evolved from a previous notion of lotus introduced by Popescu-Pampu in \cite{PPP-cerfvolant}, in order to have a common geometric interpretation of Enriques diagrams and dual graphs associated to constellations of infinitely near points of a smooth point of a surface. A lotus can also be defined abstractly and in the present paper we will be interested in \emph{Newton lotuses} $\Lambda(\mc{E})$, which are determined by a finite set of rational numbers $\mc{E}$, see Def.~\ref{Def:lotus}. To any plane curve singularity $C$ a lotus can be associated, and a Newton lotus yields the dual resolution graph of $C$ precisely when the curve is \emph{Newton non-degenerate}, a condition on the Newton polygon of $C$, see Def.~\ref{Def:nnd}.
In particular, the dual resolution graph is part of the boundary of the lotus of the curve singularity and the negative of the weights of the vertices are given as the number of triangles in the lotus incident to the given vertex. 

This is the surprising first connection to friezes, since the quiddity sequence of a Conway--Coxeter frieze is determined in the same way! Now one is inclined to ask if there are more connections between friezes and resolutions of curve singularities: the present paper first makes the correspondence between resolutions and friezes precise and then investigates whether one can see other cluster theoretic phenomena in resolutions, such as mutation, and in particular whether the other entries of the frieze have an interpretation in terms of singularity invariants. 

Our main results are 
\begin{MainThm}[cf.~Thm.~\ref{Thm:Frieze_Coord_Lotus} for a detailed version]
	\label{MainThm:Frieze-Lotus}
Let $ \cF $ be a CC-frieze of width $ w $ with entries $ p_{i,j} $ indexed as in the frieze above. 
	Let $ P $ be the corresponding $ (w+3) $-gon with triangulation $\mc{T}$ (see Thm.~\ref{Thm:CoCo}). 
	For every element $ p_{k-1,k+1} $ of the quiddity sequence of $ P $, 
	there exists a unique embedding of $ P  $ 
	as a Newton lotus of the form $  \Lambda  = \Lambda(\mc{E})$
	into the universal lotus $ \Lambda(e_1,e_2) $ of $ \ZZ^2 $ relative to the standard basis $ (e_1, e_2) $
	such that the quiddity sequence of the resulting triangulated polygon is 
		$ ( p_{k-1,k+1}, p_{k,k+2}, \ldots, p_{k+w+1,k+w+3} )$ 
		starting from the vertex $ (0,1) $.
	\\
	In particular, the vertices of the embedded polygon are determined by the two diagonals from top left to bottom right containing $ p_{k-1,k+1} $ and $ p_{k,k+2} $ respectively.
\end{MainThm}

As a consequence we show that for every resolution graph of a Newton non-degenerate curve there exists a lotus such that the lateral boundary of the lotus is the resolution graph of the curve (Cor.~\ref{Cor:graph-boundary}).  This gives us a 1-1 correspondence between CC-friezes and dual resolution graphs of Newton non-degenerate curves. 
Further, we can enumerate the pairwise different resolution graphs of type $ A_n $ taking into account the self-intersection numbers of the exceptional divisors. 
More precisely, their number is equal to
$\left \lceil  \frac{C_n  }{2} \right \rceil=\left \lceil  \frac{1}{2(n+1)} {2n \choose n}  \right \rceil $ 
(Cor.~\ref{Cor:numberresgraphs}).  \\
Note here, that Popescu-Pampu also proved an embedding result for the lotus using so-called membranes in \cite{PPP-cerfvolant}, we comment on the connection to our Thm.~\ref{MainThm:Frieze-Lotus} in Remark~\ref{Bem:cerfvolant}.
Further, we want to point out another curious connection to Farey graphs: A lotus can be identified with a normalized $ m $-gon in a Farey graph \cite[Def.~2.1.5]{Farey}
	by identifying a vertex $ (a,b) $ of the lotus with the fraction $ \frac{a}{b} $.
	The identification can be deduced from \cite[Remarque~5.7]{PPP-cerfvolant} using that a Farey series \cite[Def.~2.1.2]{Farey} is a special case of a normalized $ m $-gon in a Farey graph.
	In \cite[Prop.~2.2.1]{Farey} a one-to-one correspondence analogous to our Thm.~\ref{MainThm:Frieze-Lotus} is discussed.
	Furthermore, as explained in \cite[Section~2.3]{Farey}, Coxeter's formula \cite[(5.6)]{Coxeter} provides an interpretation of entries of a frieze in terms of Farey distances,
	where the Farey distance of two 
vertices $ v_i, v_j $ of the lotus is the determinant of the $ 2 \times 2 $ matrix given by them. \\

From the construction of the frieze associated to a lotus, $w+1$ elements of the quiddity sequence are given as negatives of the self-intersection numbers of the exceptional divisors in the minimal resolution of a curve $C$.
In order to interpret other entries of the frieze associated to  $C$ in terms of resolution invariants, we make use of representation theoretic reduction techniques (see \cite{IY08}): one may reduce the frieze, i.e., cut the corresponding lotus/triangulated polygon into two smaller pieces, where one of them is a lotus of a curve singularity $C'$ yielding a partial resolution of the original $C$. This procedure can be iterated and gives an interpretation of the frieze entries corresponding to diagonals in the triangulated polygon as negatives of self-intersection numbers in a partial resolution of $C$:

\begin{MainThm}[Thm.~\ref{Thm:Partial_reduction}]
	\label{MainThm:Partial_reduction}
	Consider the curve $C=V(f)$, where $f$ is assumed to be Newton non-degenerate, and its minimal resolution $\pi$.  
	Let $\pi'$ be a partial resolution of $C$ (cf.~Def.~\ref{Def:partial_resol}) and denote by $\cF(\pi')$ the corresponding frieze. Then $\cF(\pi')$ is obtained as a reduction of the frieze of $\pi$. In particular, if the dual resolution graph $\Gamma(f)$ is of type $A_{m-2}$ with self-intersection numbers $\{-a_i\}_{i=2}^{m-1}$, then  the dual graph of the exceptional curves appearing in $\pi'$ is of type $A_k$ for some $k \leq m-2$ and the self-intersection numbers $\{-b_j\}_{j=1}^{k}$ correspond to negatives of entries in the frieze of $\pi$.
\end{MainThm}

Finally, we determine a mutation rule for lotuses and also explain what this means for the resolution process: 
\begin{MainThm}[Thm.~\ref{Thm:Mutation}]
	\label{MainThm:Mutation}
	Let $\mc{T}$ be a triangulation of a polygon $ P $ and
	let $ \Lambda = \Lambda (P)$ be an embedding into the universal lotus with respect to a chosen basis $ (e_1, e_2) $.
	Fix an inner diagonal $ [a, b] $ and denote by $ \mu_{[a,b]} (\mc{T}) $ the mutation of the diagonal $  [a, b]  $ in $ \mc{T} $.
	There exists a well-defined lotus $ \mu_{[a,b]} (\Lambda ) $  associated to $ \mu_{[a,b]} (\mc{T}) $
	which can be explicitly described, see  Notation~\ref{Not:Not_mut} and Thm.~\ref{Thm:Mutation}.
\end{MainThm}

We believe that there are many more connections between friezes and singularities to be explored: for example, one can associate to any \emph{continued fraction} a triangulated polygon (see e.g., \cite{FareyBoat} for a nice exposition) and hence a CC-frieze. We make the connection to lotuses and friezes precise in Example \ref{Ex:lotus-continuedfraction} and Remark \ref{Rmk:Riemenschneider}. \\
Continued fractions determine cyclic quotient surfaces and were studied in great detail, see e.g. \cite{Riemenschneider74, Stevens, BrohmeDiss}. 
In these papers versal deformations of cyclic quotient surfaces were studied as well as their minimal resolution.

The paper is organized as follows: in Section \ref{Sec:continuedfractions-friezes} we recall the basic properties of Conway--Coxeter friezes, triangulated polygons, continued fractions and how to connect them, in particular, Kidoh's lemma, which relates dual continued fractions to different parts of the quiddity sequences of a frieze. Section \ref{Sec:lotus} recalls the necessary notions from the theory of complex curve singularities and toric geometry to define the lotus of a curve singularity. We also comment on how to relate a lotus to a continued fraction and its corresponding triangulated polygon (cf.~Example \ref{Ex:lotus-continuedfraction}). Our main Thm.~\ref{MainThm:Frieze-Lotus}  and Corollary about the number of dual resolution graphs of type $A_n$ are stated and proven in Section \ref{Sec:friezes-graphs}. In order to interpret the entries of a frieze in terms of the resolution of the corresponding curve singularity we change gears and venture in a more categorical territory: in Section~\ref{Sub:categoryA} we give a short introduction to cluster categories of type $A$ and the notion of reduction of a frieze in Section \ref{Sub:reduction}, where we show how the quiddity sequence changes under reduction of a frieze. Thm.~\ref{Thm:Partial_reduction} connecting partial resolutions graphs of a curve singularity and some entries of the corresponding frieze is then shown. Finally, in Section \ref{Sec:mutation} we show how to interpret mutation of friezes/triangulated polygons in terms of lotuses and associated curve singularities. We end with some questions and potential generalizations of lotuses.

{\em Acknowledgements:}
We want to express our gratitude to the Mathematisches Forschungs\-institut Oberwolfach for the inspiring environment and perfect working conditions during a stay as Oberwolfach Research Fellows in 2022. We also thank Patrick Popescu-Pampu for interesting discussions that stirred our interest in the topic of lotuses, as well as Ian Short for helpful comments. 

This work was supported by the Engineering and Physical Sciences Research Council [grant number EP/W007509].
This material is based upon work supported by the National Science Foundation under Grant No.~DMS-1928930 and by the Alfred P. Sloan Foundation under grant G-2021-16778, while E.F.~was in residence at the Simons Laufer Mathematical Sciences Institute (formerly MSRI) in Berkeley, California, during the Spring 2024 semester.

\section{Useful facts about continued fractions and friezes} \label{Sec:continuedfractions-friezes}

We begin by recalling the basic notions of our article, (Hirzebruch--Jung) continued fractions and 
Conway--Coxeter friezes. 
Along this, we recall well-known results such as the connection of each of these notions to triangulated polygons.
For more detailed references, we refer to \cite{Perron,Patrick-continued-fractions} about continued fractions, \cite{Crossroads} and references therein for friezes, and \cite{FareyBoat, CanakciSchiffler} about connections of continued fractions and cluster algebras.

\subsection{Continued fractions and triangulations of polygons} Let $\lambda=\frac{n}{q} \in \QQ_+ $, with $n>q$ and $(n,q)=1$ without loss of generality. Then $\lambda$ can be written as
$$\lambda=b_1 - \frac{1}{b_2- \frac{1}{\ldots - \frac{1}{b_r}}} \ , b_i \geq 2 \ .$$
This expansion is called \emph{(Hirzebruch--Jung) continued fraction}, and we will denote it by $\lambda=\lb b_1, \ldots, b_r \rb$. Note that any $\lambda \in \QQ_{>1}$ has a unique Hirzebruch--Jung continued fraction expansion (see \cite[Section 2]{Patrick-continued-fractions}, or \cite{Perron}
for a proof for positive continued fractions). If we have $ n \leq q $, the Hirzebruch-Jung continued fraction for $\frac{n}{q} \leq 1$ is defined analogously with the difference that we have $b_i \geq 1$.

\begin{Bem} In the literature Hirzebruch--Jung continued fractions are sometimes called negative continued fractions. Often only positive continued fraction expansions  are considered (and sometimes dubbed \emph{Euclidean continued fractions}, see \cite[Section 2]{Patrick-continued-fractions}). But here we will deal exclusively with Hirzebruch--Jung continued fractions, and since there will not be any danger of confusion, we will just speak of $\lambda=\lb b_1, \ldots, b_r \rb$ as the continued fraction expansion of $\lambda$.
\end{Bem}

\begin{defi} \label{def:polygon}
A \emph{polygon} consists of a finite set 
 $V$ of $m \geq 3$ vertices with a cyclic order. 
We may treat $P$ realized as convex $m$-gon in the Euclidean plane. If $a \neq b$ are vertices of $P$, then there is a diagonal, which we denote by $[a,b]$. 
A \emph{triangulation} $\mc{T}$ of $P$ is a maximal set of pairwise non-crossing diagonals between non-neighboring vertices (we sometimes call these \emph{inner diagonals}). 
\end{defi}

 Consider a polygon $P$ with $m$ vertices. There are $\frac{1}{m-1} { 2m-4 \choose m-2 }$ different triangulations of $P$ (see e.g. \cite[Exercise 6.19]{Stanley2}), i.e., the $(m-2)$-nd Catalan number, and each triangulation consists of $(m-2)$-triangles and determines $m-3$ inner diagonals. 
\\
Fix a triangulation of $P$. We number the vertices of $P$ as 
$ v_1, \ldots, v_m $
and define the \emph{quiddity sequence of the polygon} 
$\{\alpha_i\}_{i=1}^{m}$
via: $\alpha_i$ is the number of triangles incident to $v_i$. We will sometimes consider the indices modulo $m$, i.e., $\alpha_{m}=\alpha_0$ etc. If $\alpha_i=1$, then we call the corresponding vertex $v_i$ an \emph{ear} in the triangulation. Further note that $m-2\geq \alpha_i \geq 1$ and each triangulation has at least two ears 
(for a proof see e.g. \cite[Proof of Lemma 11]{HenrySoizic} or \cite{CoCo2}). 
In the following we will quickly explain the bijection between triangulations with exactly two ears and continued fractions. For more on this topic and matrices related to continued fractions, see \cite{FareyBoat}.

\begin{lem-def} \label{Lem:quiddity}
Let $\lambda=\lb b_1, \ldots, b_r \rb > 1  $ be a rational number and $m:=\sum_{i=1}^rb_i-r+3$. 
There exists a triangulation of an $m$-gon with exactly two ears such that
the sequence 
$\{b_i\}_{i=1}^r $  is part of the corresponding quiddity sequence $\{\alpha_i\}_{i=0}^{m-1}$:
\[ 
	\alpha_0=1 \ , \ \ 
	\alpha_i=b_i \ , \ \ \mbox{ for } i \in \{ 1, \ldots, r \} \ , \ \  \alpha_{r+1}=1 \ , 
\] 
and $ \alpha_i > 1 $ for $ i \in \{ r+2, \ldots, m-1 \} $. 
The remaining $\alpha_i$, for $ i \in \{ r+2, \ldots, m-1 \} $, can be uniquely determined (cf.~Prop.~\ref{Prop:Kidoh-duality}).
\\
We call this triangulation the \emph{triangulation associated to $\lambda$}.
\end{lem-def}

\begin{proof} The formula for $m$ can easily be deduced from the interpretation of the $b_i$ as part of a quiddity sequence,  see e.g.~\cite{FareyBoat}. The remaining $\alpha_i$ can be computed by drawing the triangulation of the polygon, cf.~Prop.~\ref{Prop:Kidoh-duality} and Fig.~\ref{Fig:Kidoh-duality}.
\end{proof}

Note that with Lemma-Definition~\ref{Lem:quiddity} we have characterized triangulations of an $m$-gon with exactly two ears as triangulations associated to continued fractions.

Note that there is a duality of the Hirzebruch--Jung-expansion for $\lambda=\frac{n}{q}$ and $\frac{\lambda}{\lambda-1}=\frac{n}{n-q}$  due to Kidoh. 
We will explicitly show  how to read off the quiddity sequence of the triangulation of the $m$-gon corresponding to $\frac{n}{q}$ using Kidoh's Lemma \cite[Prop.~1.2]{Kidoh}:

\begin{proposition}[Duality and quiddity] \label{Prop:Kidoh-duality}
Let $\frac{n}{q}=\lb b_1, \ldots, b_r \rb$, $\frac{n}{n-q}=\lb b'_1, \ldots, b'_s \rb$ 
be such that $n > q >0$. 
Set $m:=\sum_{i=1}^r b_i-r+3$. Then $s=m-r-2$ and there are positive integers $c_i$ and $d_i$, 
$ i \in \{ 1, \ldots, \kappa \} $
, such that:
\begin{align}
\label{Eq:nq} \frac{n}{q}  &= \lb d_1 +1 \ , \underbrace{2, \ldots, 2}_{c_1-1} \ ,d_2+2, \ldots, d_{\kappa-1}+2 \ , \underbrace{2, \ldots, 2}_{c_{\kappa-1}-1} \ ,d_\kappa+2\ , \underbrace{2, \ldots, 2}_{c_\kappa-1} \rb \\
\label{Eq:n-q} \frac{n}{n-q} & = \lb \underbrace{2, \ldots, 2}_{d_1-1} \ , c_1+2 \ , \underbrace{2, \ldots, 2}_{d_2-1} \ , c_2 +2, \ldots, c_{\kappa-1}+2 \ , \underbrace{2, \ldots, 2}_{d_{\kappa}-1} \ , c_{\kappa}+1 \rb
\end{align}
Further, the quiddity sequence of the $m$-gon with triangulation coming from $\frac{n}{q}$ is
\[
	\alpha_0=1, \quad \alpha_i=b_i, \ \mbox{ for } i \in \{ 1, \ldots, r \}, \quad \alpha_{r+1}=1, \quad \alpha_i=b'_{s+r+2-i}, \ 
	\mbox{ for } i \in \{r+2, \ldots, m-1\} \ . 
\]
\end{proposition}

\begin{proof}
The continued fraction expansions of $\frac{n}{q}$ and $\frac{n}{n-q}$ are calculated in \cite[Prop.~1.2]{Kidoh}. The formula for $s$ follows directly from Lemma-Definition~\ref{Lem:quiddity} (and can best seen by a sketch, see Fig.~\ref{Fig:Kidoh-duality}). 
Assume now that $\frac{n}{q}$ is given by expression \eqref{Eq:nq}. Then we have $b_1=d_1+1, b_2=2, \ldots$ and so on. 
We can see this sequence on the corresponding triangulated $m$-gon in Fig.~\ref{Fig:Kidoh-duality}. 
\begin{figure}[h!] 
\begin{tikzpicture}[scale=0.5, baseline=(current  bounding  box.center) ]
%untere Punkte
\coordinate (u0) at (1,2);
\coordinate (u1) at (2,0);
\coordinate (u2) at (4,0);
\coordinate (u3) at (6,0);
\coordinate (u4) at (8,0);
\coordinate (u5) at (10,0);
\coordinate (u6) at (13,0);
\coordinate (u7) at (17,0);
\coordinate (u8) at (20,0);
\coordinate (u9) at (22,0);
\coordinate (u10) at (24,0);
\coordinate (u11) at (26,2);

%obere Punkte
\coordinate (o1) at (3,4);
\coordinate (o2) at (6,4);
\coordinate (o3) at (8,4);
\coordinate (o4) at (11,4);
\coordinate (o5) at (18,4);
\coordinate (o6) at (21,4);
\coordinate (o7) at (23,4);
\coordinate (o8) at (25,4);

\draw (u0) node[left] {$1$};
\draw (u1) node[below] {$2$};
\draw (u2) node[below] {$2$};
\draw (u3) node[below] {$c_1+2$};
\draw (u4) node[below] {$2$};
\draw (u5) node[below] {$2$};
\draw (u6) node[below] {$c_2+2$};
\draw (u7) node[below] {$c_{\kappa-1}+2$};
\draw (u8) node[below] {$2$};
\draw (u9) node[below] {$2$};
\draw (u10) node[below] {$c_{\kappa}+1$};
\draw (u11) node[right] {$1$};

\draw (o1) node[above] {$d_1+1$};
\draw (o2) node[above] {$2$};
\draw (o3) node[above] {$2$};
\draw (o4) node[above] {$d_2+2$};
\draw (o5) node[above] {$d_{\kappa}+2$};
\draw (o6) node[above] {$2$};
\draw (o7) node[above] {$2$};
\draw (o8) node[above] {$2$};

%klammern
\draw [decoration={brace, raise=3ex}, decorate] (o2) -- (o3) node[above=3ex,pos=0.1ex] {$c_1-1$};
\draw [decoration={brace, raise=3ex}, decorate] (o6) -- (o8) node[above=3ex,pos=0.1ex] {$c_{\kappa}-1$};

\draw [decoration={brace, mirror, raise=3ex}, decorate] (u1) -- (u2) node[below=3ex,pos=0.1ex] {$d_1-1$};
\draw [decoration={brace, mirror, raise=3ex}, decorate] (u4) -- (u5) node[below=3ex,pos=0.1ex] {$d_2-1$};
\draw [decoration={brace, mirror, raise=3ex}, decorate] (u8) -- (u9) node[below=3ex,pos=0.1ex] {$d_{\kappa}-1$};

%\draw (10,4)--(4,0); 
\filldraw[draw=black, fill=blue!20, densely dotted] 
     (u1)--(o1)--(u2)--(u1);
\draw [fill=blue!40] (u0)--(u1)--(o1)--(u0);
\draw [fill=blue!40] (o1)--(u2)--(u3)--(o1);
\filldraw[draw=black, fill=blue!20, densely dotted] 
     (u4)--(o4)--(u5)--(u4);
\draw [fill=blue!40] (u3)--(u4)--(o4)--(u3);
\draw [fill=blue!40] (u5)--(u6)--(o4)--(u5);
\filldraw[draw=black, fill=blue!20, densely dotted] 
     (u8)--(o5)--(u9)--(u8);
\draw [fill=blue!40] (u7)--(u8)--(o5)--(u7);
\draw [fill=blue!40] (u9)--(u10)--(o5)--(u9);
\draw [fill=blue!40] (u5)--(u6)--(o4)--(u5);
\draw [fill=blue!40] (u10)--(u11)--(o8)--(u10);

\filldraw[draw=black, fill=orange!20, densely dotted] 
     (o2)--(o3)--(u3)--(o2);
\draw [fill=orange!40] (o1)--(u3)--(o2)--(o1);
\draw [fill=orange!40] (o3)--(u3)--(o4)--(o3);
\draw [fill=orange!40] (o1)--(u3)--(o2)--(o1);

\filldraw[draw=black, fill=orange!20, densely dotted] 
     (o6)--(u10)--(o7)--(o6);
\draw [fill=orange!40] (o5)--(u10)--(o6)--(o5);
\draw [fill=orange!40] (o8)--(u10)--(o7)--(o8);

\draw (15,1.5) node[above] {$\cdots$};

\filldraw[draw=black, fill=orange!30, dotted] 
     (o4)--(u6)--(u7)--(o5)--(o4);
\end{tikzpicture}
\caption{Triangulation associated to a continued fraction $\frac{n}{q}$ and $\frac{n}{n-q}$.}
\label{Fig:Kidoh-duality}
\end{figure}
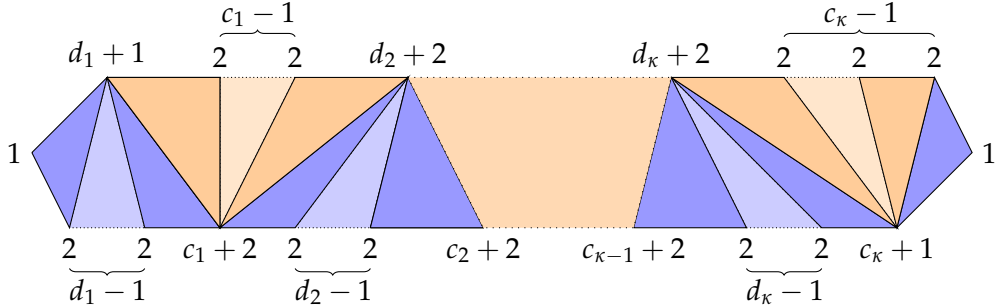
\end{proof}

Observe in Fig.~\ref{Fig:Kidoh-duality} that the ``upper" part of the triangulation provides the continued fraction for $ \frac{n}{q} $,
	while the ```lower" part corresponds to the continued fraction of $ \frac{n}{n-q} $.

\begin{example} \label{Ex:fountain}
Let $\frac{n}{q}=\lb n \rb$ with $n \geq 2$, i.e., $q=1$. We calculate the continued fraction expansion of $\frac{n}{n-q}=\frac{n}{n-1}$ with Prop.~\ref{Prop:Kidoh-duality}. We have $r=1$, $b_1=n$, and $m=n-1+3=n+2$. Further, the length of the continued fraction expansion of $\frac{n}{n-1}$ is $s=m-r-2=(n+2)-1-2=n-1$. We first calculate the integers $c_i$ and $d_i$ from 
Prop.~\ref{Prop:Kidoh-duality} from $\lb n \rb$, that is, $ \kappa =1$, and $d_1=n-1$, and $c_1=1$. Then the continued fraction expansion for $\frac{n}{n-1}$ is
$$\frac{n}{n-1}=\lb \underbrace{2, \ldots, 2}_{(n-1)-1} \ , 1+1 \rb=\lb \underbrace{2, \ldots, 2}_{n-1} \rb \ .$$
The triangulation of an $(n+2)$-gon corresponding to $\lb n \rb$ is illustrated for $ n = 6 $ in Fig.~\ref{Fig:fontain}. \\
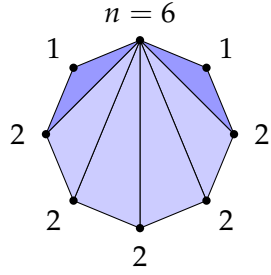
\begin{figure}[h!]  
    \centering
    \begin{tikzpicture}[scale=2.5,cap=round,>=latex]

 %decorations
    \node  at (90:0.65) {$n=6$};
        \node   at (45:0.65) {$1$};
    \node  at (0:0.65) {$2$};
        \node at (315:0.65) {$2$};
    \node at (270:0.65) {$2$};
        \node at (225:0.65) {$2$};
            \node at (180:0.65) {$2$};
         \node at (135:0.65) {$1$};
         
    \coordinate  (p1) at (90:0.5) {}; 
        \coordinate  (p2) at (45:0.5){};
   \coordinate (p3) at (0:0.5){};
    \coordinate (p4) at (315:0.5){};
   \coordinate (p5) at (270:0.5){}; 
    \coordinate (p6) at (225:0.5){}; 
   \coordinate (p7) at (180:0.5){};  
   \coordinate (p8) at (135:0.5){};

    %the triangles                                    
 \draw[fill=blue!40] (p1)--(p2)--(p3)--(p1); 
  \draw[fill=blue!20] (p1)--(p3)--(p4)--(p1);
    \draw[fill=blue!20] (p1)--(p4)--(p5)--(p1);                           
  \draw[fill=blue!20] (p1)--(p5)--(p6)--(p1); 
  \draw[fill=blue!20] (p1)--(p6)--(p7)--(p1);  
  \draw[fill=blue!40] (p1)--(p7)--(p8)--(p1);                             
    %marked points
  
	\draw (90:0.5) node[fill=black,circle,inner sep=0.039cm] {} circle (0.01cm);	        
     \draw (45:0.5) node[fill=black,circle,inner sep=0.039cm] {} circle (0.01cm);
     \draw (0:0.5) node[fill=black,circle,inner sep=0.039cm] {} circle (0.01cm);
     \draw (315:0.5) node[fill=black,circle,inner sep=0.039cm] {} circle (0.01cm);
     \draw (270:0.5) node[fill=black,circle,inner sep=0.039cm] {} circle (0.01cm);
     \draw (225:0.5) node[fill=black,circle,inner sep=0.039cm] {} circle (0.01cm);
     \draw (180:0.5) node[fill=black,circle,inner sep=0.039cm] {} circle (0.01cm);
     \draw (135:0.5) node[fill=black,circle,inner sep=0.039cm] {} circle (0.01cm);

    \end{tikzpicture}
    \caption{The fan triangulation associated to the continued fraction $n$ and its dual $\frac{n}{n-1}$.}
\label{Fig:fontain}
\end{figure}
\end{example}

\begin{example} \label{Ex:running}
Let $\frac{n}{q}=\frac{11}{8}=\lb 2,2,3,2 \rb$. Then $\frac{n}{n-q}=\frac{11}{3}=\lb 4,3 \rb$. We have $r=4$, $m=8$, and $s=2$. For the integers $c_i$ and $d_i$ from Prop.~\ref{Prop:Kidoh-duality}, one has $c_1=2, d_1=1, c_2=2, d_2=1$. Note that $ \kappa=2$. 
The corresponding triangulation of the $8$-gon is shown in Fig.~\ref{Fig:running}.
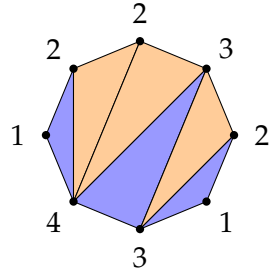
\begin{figure}[h!]
    \centering
    \begin{tikzpicture}[scale=2.5,cap=round,>=latex]

 %points
    \node at (90:0.65) {$2$};
        \node at (45:0.65) {$3$};
    \node at (0:0.65) {$2$};
        \node at (315:0.65) {$1$};
    \node at (270:0.65) {$3$};
        \node at (225:0.65) {$4$};
            \node at (180:0.65) {$1$};
         \node at (135:0.65) {$2$};
         
    \coordinate  (p1) at (90:0.5) {}; 
        \coordinate  (p2) at (45:0.5){};
   \coordinate (p3) at (0:0.5){};
    \coordinate (p4) at (315:0.5){};
   \coordinate (p5) at (270:0.5){}; 
    \coordinate (p6) at (225:0.5){}; 
   \coordinate (p7) at (180:0.5){};  
   \coordinate (p8) at (135:0.5){};      
    
%triangles                                        
   \draw[fill=blue!40] (p7)--(p8)--(p6)--(p7);       
      \draw[fill=orange!40] (p6)--(p8)--(p1)--(p6);                                 
      \draw[fill=orange!40] (p6)--(p1)--(p2)--(p6);
   \draw[fill=blue!40] (p6)--(p2)--(p5)--(p6); 
      \draw[fill=orange!40] (p5)--(p2)--(p3)--(p5);
   \draw[fill=blue!40] (p5)--(p3)--(p4)--(p5);                                               
    %marked points
	\draw (90:0.5) node[fill=black,circle,inner sep=0.039cm] {} circle (0.01cm);	        
     \draw (45:0.5) node[fill=black,circle,inner sep=0.039cm] {} circle (0.01cm);
     \draw (0:0.5) node[fill=black,circle,inner sep=0.039cm] {} circle (0.01cm);
     \draw (315:0.5) node[fill=black,circle,inner sep=0.039cm] {} circle (0.01cm);
     \draw (270:0.5) node[fill=black,circle,inner sep=0.039cm] {} circle (0.01cm);
     \draw (225:0.5) node[fill=black,circle,inner sep=0.039cm] {} circle (0.01cm);
     \draw (180:0.5) node[fill=black,circle,inner sep=0.039cm] {} circle (0.01cm);
     \draw (135:0.5) node[fill=black,circle,inner sep=0.039cm] {} circle (0.01cm);
    \end{tikzpicture}
         \caption{Triangulation of $8$-gon corresponding to $\frac{11}{8}$.}            \label{Fig:running}   
\end{figure}{}
\end{example}

\subsection{Friezes and continuants}

Here we recall the notions of Conway--Coxeter friezes and continuant polynomials.

\begin{defi} A \emph{continuant} of order $n$ is the determinant of a tri-diagonal matrix of the form 
$$\begin{pmatrix} y_1 & b_1 & 0 & \cdots & \cdots & \cdots \\
c_1 & y_2 & b_2 & 0 & \cdots & \cdots  \\
0 & c_2 & y_3 & b_3 & 0 & \cdots \\
\vdots & \ddots &  \ddots & \ddots &  \ddots & \ddots \\
0 & \cdots & 0  &  c_{n-2} & y_{n-1} & b_{n-1} \\
0 & \cdots & \cdots  & 0 & c_{n-1} & y_n \end{pmatrix} \ . $$
We will consider the special case where $b_i=c_i=1$ for all 
$ i \in \{ 1, \ldots , n-1 \} $, and denote this continuant by $P_n(y_1, \ldots, y_n)$. We set $P_0:=1$. 
\end{defi}

\begin{example} \label{Ex:continuants}
The continuants for small $n$ are $P_0=1$, $P_1(y_1)=y_1$, $P_2(y_1,y_2)=y_1y_2-1$, and 
$P_3(y_1, y_2, y_3) = y_1y_2y_3-y_1-y_3 $.
\end{example}

\begin{defi} A (closed) {\em frieze} is a grid of numbers (elements in a commutative ring with $1$) with a finite number of infinite rows, where the top and bottom rows are bi-infinite repetition of $0$s and the second to top and the second to bottom row are bi-infinite repetitions of $1$s:

\begin{equation} \label{Eqn:frieze}
\adjustbox{scale=0.85,center}{
\begin{tikzcd}[row sep=0.35em, column sep=-0.35em]
 &\ldots && 0 && 0 && 0 && 0 && \ldots &\\
   && 1 && 1 && 1 && 1 && 1 && \\
  &\ldots && (-2,0) && (-1,1) && (0,2) && (1,3) && \ldots& \\
   && (-3,0) && (-2,1) && (-1,2) && (0,3) && (1,4) && \ldots \\
     &\ldots && (-3,1) && (-2,2) && (-1,3) && (0,4) && (1,5) & & \ldots \\
     &   & \ldots && \ldots && \ldots && \ldots && \ldots && \ldots \\
   & \ldots && \ldots && \ldots && (-2,w-1) && (-1,w) && (0,w+1) && \ldots \\
   && 1 && 1 && 1 && 1 && 1 && \ldots \\
   &\ldots && 0 && 0 && 0 && 0 && \ldots &
\end{tikzcd}
}
\end{equation}

satisfying the \emph{frieze rule}: any four adjacent entries arranged in a \emph{diamond}
\[ \xymatrix@=0.1em{ & b & \\ a && d \\ & c &} \] satisfy the equation
\begin{equation} 
	\label{eq:friezerule}
	ad - bc= 1 \ . 
\end{equation}
The $w$ rows between are sometimes called the \emph{nontrivial rows} and 
their number $w$ is called the \emph{width} of the frieze. 
With the indexing of \eqref{Eqn:frieze}, we have
\[ 
	(i,i+1)=(i,i+w+2)=1
	\ \ \mbox{ and } \ \ 
	(i,i)=(i,i+w+3)=0 
	\ , \ \ 
	\mbox{ for any } i \in \ZZ \ .
\] 
	The first nontrivial row of the frieze with entries $(i-1,i+1)=:a_i$ is called the \emph{quiddity row} and its entries $\{ a_i \}_{i \in \ZZ}$ the \emph{quiddity sequence of the frieze}. 
\end{defi}

Coxeter showed that friezes of width $w \geq 0$ are periodic in the horizontal direction with period $w+3$ \cite[Section 6]{Coxeter}. In particular, the quiddity sequence of the frieze is determined by $ \{ a_i \}_{i=1}^{w+3} $.
Hence, we often call the latter finite sequence also the quiddity sequence of the frieze.
Furthermore, in \cite[(6.6)]{Coxeter} it is seen that each entry of the frieze is a continuant:
\begin{equation} \label{eq:frieze_entry_cont}
(i,j)=P_{j-i-1}(a_{i+1}, \ldots, a_{j-1}) \ .
\end{equation}

In the following we will consider friezes of finite width $w$ with positive integer entries, that is, all entries in the nontrivial rows are in $\ZZ_+$. Such friezes are called \emph{finite integral friezes} or \emph{Conway--Coxeter friezes}. 
We will often denote them by CC-friezes. 
The beautiful theorem of Conway and Coxeter below relates these friezes  to triangulations of polygons.

\begin{defi}[Frieze of a triangulation] \label{Def:friezetriang}
Let $P$ be an $m$-gon with vertices 
$v_1, \ldots, v_{m}$ and a triangulation $\mc{T}$. 
The \emph{frieze of $\mc{T}$} is the frieze $\mathcal{F}(\mc{T})$ of width $w=m-3$ with quiddity 
$ \{ a_i\}_{i=1}^m $, where $ a_i $ is the number of triangles of $\mc{T}$ incident to $v_i$. We denote the frieze that is associated to the triangulated polygon for the continued fraction $\lambda$ by $\mathcal{F}(\lambda)$. 
\end{defi}

\begin{Thm}[Conway--Coxeter \cite{CoCo1, CoCo2}]
	\label{Thm:CoCo}
There is a bijection between triangulated polygons with $m$ vertices and Conway--Coxeter friezes of width $m-3$: 
\begin{enumerate}
\item Let $\mathcal{F}$ be a Conway--Coxeter frieze of width $m-3$ with the entries labeled $(i,j)$ as in \eqref{Eqn:frieze}. Then the entries $(i-1,i+1)=:a_i$ for 
$ i \in \{ 1, \ldots, m \} $
of $\mathcal{F}$ give the quiddity sequence of a triangulated $m$-gon $P$.
\item Let $P$ be an $m$-gon with triangulation $\mc{T}$. Then the frieze corresponding to $\mc{T}$ is the one defined in Def.~\ref{Def:friezetriang}, that is, $\mathcal{F}(\mc{T})$ is the frieze with quiddity sequence 
$ \{ a_i \}_{i=1}^m $.
The remaining entries in the frieze can be calculated by the frieze rule and the horizontal periodicity.
\end{enumerate}
Note here that sometimes it is more convenient to label the vertices $v_0, \ldots, v_{m-1}$, in this case the quiddity sequence will be denoted by 
	$ \{ a_i\}_{i=0}^{m-1} $.
\end{Thm}

\begin{example}
Let $P$ be the $8$-gon with triangulation given by $\lambda=\lb 2,2,2,2,2 \rb$ as in Example \ref{Ex:fountain}. Then the frieze $\F(\lambda)$ is shown in Fig.~\ref{Fig:FriezeFountain}:
\begin{figure}[h!]
\begin{tikzcd}[row sep=0.35em, column sep=-0.35em]
\col{1}&&\col{1} &&\col{1} && 1 && 1   && 1   && 1  && 1   && 1   && 1 &&1 && \col{1} && \col{1} && \col{1} && \col{1}\\
&\col{2}&& \col{1}&& \col{6} && 1     && 2     && 2     && 2     && 2    && 2    && 1 && 6 && \col{1} && \col{2} && \col{2}  && \col{2}\\
&& \col{1}&&\col{5}&&\col{5}&& 1     && 3     && 3     && 3     && 3    && 1     && 5 &&5 &&\col{1} &&\col{3} &&\col{3}  &&\col{3} \\
&&& \col{4} &&\col{4}&&\col{4}&& 1     && 4     && 4     && 4      && 1     && 4     && 4 &&4 && \col{1}&& \col{4} && \col{4} && \col{4}\\
& &&& \col{3} &&\col{3}&&\col{3}&& 1     && 5   && 5     && 1      && 3   && 3     && 3 &&3 && \col{1}&& \col{5} && \col{5} && \col{1}\\
& & &&& \col{2} &&\col{2}&&\col{2}&& 1     && 6   && 1     && 2      && 2   && 2     &&2 &&2 && \col{1}&& \col{6} && \col{1} && \col{2}\\
&& &&&& \col{1} &&\col{1}&&\col{1} && 1   && 1  && 1   && 1   && 1   && 1  && 1  && 1 &&\col{1} && \col{1} && \col{1} && \col{1}
\end{tikzcd}
\caption{The frieze obtained from $\lambda=\frac{6}{5}=\lb2,2,2,2,2\rb$.}
\label{Fig:FriezeFountain}
\end{figure}
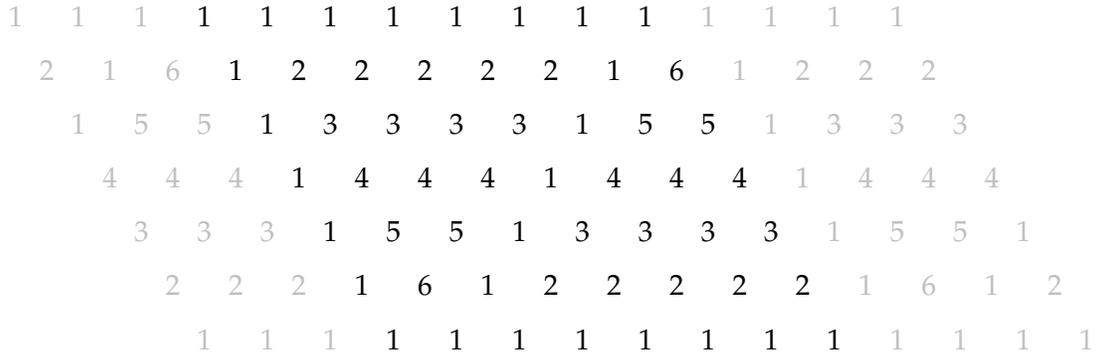
 
For the triangulated polygon of Example \ref{Ex:running} corresponding to $\lambda=\frac{11}{8}=\lb 2,2,3,2 \rb$ we obtain the frieze $\F(\lambda)$ shown in Fig.~\ref{Fig:FriezeRunning}:
\begin{figure}[h!]
\begin{tikzcd}[row sep=0.35em, column sep=-0.35em]
\col{1}&&\col{1} &&\col{1} && 1 && 1   && 1   && 1  && 1   && 1   && 1 &&1 && \col{1} && \col{1} && \col{1} && \col{1}\\
&\col{1}&& \col{3}&& \col{4} && 1     && 2     && 2     && 3     && 2    && 1    && 3 && 4 && \col{1} && \col{2} && \col{2}  && \col{3}\\
&& \col{2}&&\col{11}&&\col{3}&& 1     && 3     && 5     && 5     && 1    && 2     && 11 &&3 &&\col{1} &&\col{3} &&\col{5}  &&\col{5} \\
&&& \col{7} &&\col{8}&&\col{2}&& 1     && 7     && 8     && 2      && 1     && 7     && 8 &&2 && \col{1}&& \col{7} && \col{8} && \col{4}\\
& &&& \col{5} &&\col{5}&&\col{1}&& 2     && 11   && 3     && 1      && 3   && 5     && 5 &&1 && \col{2}&& \col{11} && \col{3} && \col{1}\\
& & &&& \col{3} &&\col{2}&&\col{1}&& 3     && 4   && 1     && 2      && 2   && 3     &&2 &&1 && \col{3}&& \col{4} && \col{1} && \col{2}\\
&& &&&& \col{1} &&\col{1}&&\col{1} && 1   && 1  && 1   && 1   && 1   && 1  && 1  && 1 &&\col{1} && \col{1} && \col{1} && \col{1}
\end{tikzcd}
\caption{The frieze obtained from $\lambda=\frac{11}{8}=\lb2,2,3,2\rb$.}
\label{Fig:FriezeRunning}
\end{figure}
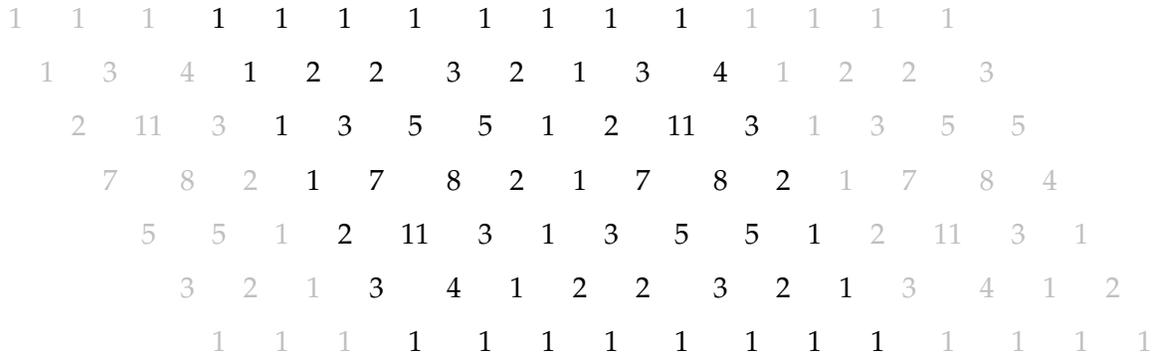

\end{example}

\begin{lemma}
\label{Lem:nq_conti}
For $ \frac{n}{q} = \lb b_1,\ldots, b_r \rb $ with $ \gcd(n,q) = 1 $ and $ q < n $ the following holds:
\[
	\frac{n}{q}	= \frac{P_r(b_1,\ldots, b_r)}{P_{r-1}(b_2,\ldots,b_r)} \ . 
\]
\end{lemma}

The lemma follows from an induction on $ r $ using that $\lb b_1, \ldots, b_r,b_{r+1} \rb= b_1 - \frac{1}{\lb b_2, \ldots, b_{r+1} \rb} $
and that the recursion $ P_n(y_1, \ldots, y_n) = y_n P_{n-1}(y_1, \ldots, y_{n-1}) - P_{n-2} (y_1, \ldots, y_{n-2} ) $ \cite[Number~547]{Muir} holds.

Using the mentioned recursion, one can deduce that 
$ P_n (y_1, \ldots, y_n) $ and $ P_{n-1} (y_1, \ldots, y_{n-1}) $ have no common divisor, for $ n \geq 1 $,  via an induction on $ n $.
This together with \eqref{eq:frieze_entry_cont} and
Lemma~\ref{Lem:nq_conti} implies:

\begin{cor}
	\label{Cor:n_q_in_frieze}
Let $\gcd(n,q)=1$, $ n > q$, and $\frac{n}{q}=\lb b_1,\ldots, b_r \rb $. Then $n$ is the $(0,r+1)$ entry and $q$ is the $(1,r+1)$ entry in the corresponding frieze $\mathcal{F}(\lambda)$.
\end{cor}

\section{Resolution of singularities and lotuses} \label{Sec:lotus}

Here we recall basic notions about complex plane curve singularities and their resolutions of singularities. In particular, we will be interested in the \emph{lotus} of a plane curve singularity. We follow \cite{Lotus}, where more details and references about this material can be found.

Let $S$ be a smooth complex analytic surface, and consider the germ $(S,s)$ at the point $s \in S$.  A \emph{curve $C$ in a complex surface} $S$ is an effective Cartier divisor of $S$, i.e., a complex subspace of $S$ that is locally at $s$ defined by the vanishing of a non-zero holomorphic function. Choosing a holomorphic coordinate system $(x,y)$ at $s$, we have $\mc{O}_{S,s} \cong \CC\{x,y\}$ the ring of convergent power series with maximal ideal $\mf{m}_{S,s}=(x,y)$. If $C$ passes through $s \in S$, the germ $(C,s)$ is thus locally defined by $f \in \mf{m}_{S,s}$ and $\mc{O}_{C,s}\cong \CC\{x,y\}/(f)$. 
Since $(S,s)$ is locally isomorphic to $(\CC^2,0)$, we call $C$ a \emph{plane curve singularity}. Locally at $s$ the curve $C$ is given as $V(f)=\{f(x,y)=0\}$. Recall that in this setting the singular locus of $(C,s)$ is defined by the vanishing locus of the Jacobian ideal $J_C := (f, \frac{\partial f}{\partial x}, \frac{\partial f}{\partial y} ) $, see e.g. \cite{Cut2004}. In the following we will only consider reduced curve singularities, that is, $C$ is a reduced complex analytic space, so that locally at $s$ the convergent power series $f \in \CC\{x,y\}$ is reduced. \\
Note that one can also consider $f$ as a formal power series (as done in \cite{Lotus}), i.e., define a plane curve singularity in $\widehat{\mathcal{O}}_{S,s}\cong \CC\lb x,y\rb$. \\

Let $(C,s)$ and $(D,s)$ be two curve singularities in $(S,s)$ defined by functions $f, g \in \mc{O}_{S,s}$ respectively. The \emph{intersection number} of $C$ and $D$ at $s$ is denoted by $C \cdot D$ and defined by
$$C \cdot D:= \dim_{\CC} (\mc{O}_{S,s}/(f,g)) \ . $$
More generally, the intersection number of two divisors $C$ and $D$ on any smooth complex surface may be calculated whenever at least one of them has compact support. 
Moreover, the \emph{self-intersection number} $C \cdot C$ or $C^2$ at $s$ can also be defined, see \cite[Thm.~2.3]{Laufer}.
Note that the self-intersection number may be negative. 

\subsection{Blowups and resolution graphs} In the following we will define the resolution graph of a plane curve singularity (sometimes called weighted dual graph). This material is well-known and covered in the literature, see e.g., \cite[Section 2.4]{Lotus}, \cite[Chapter 5]{dJP}, \cite[Section 3.6]{Wall}. We will only introduce the notions needed without any proofs. \\
First recall the \emph{blowup of a point $s$ in $\CC^2$}:  choose coordinates $(x,y)$ at $s$ (so that $s$ can be thought of as the origin) and denote by $\PP^1(\CC)$ the projective line with coordinates $ (X:Y) $ (lines through $s$). Define $\Bl_s(\CC^2)$ as the set of points in $\CC^2 \times \Proj^1(\CC)$ satisfying the equation 
\begin{equation} 
	\label{eq:BU}
	x Y - y X = 0 \ .
\end{equation}  
The projection of $\CC^2 \times \Proj^1(\CC)$ onto $\CC^2$ yields the \emph{blowup map} 
$\pi: \Bl_s(\CC^2) \xrightarrow{} \CC^2. $
Any point $s' = (\alpha, \beta) \neq s = (0,0) $ determines a unique $(X:Y) = (\alpha : \beta)$ and thus $\pi^{-1}(s')$ is a point, whereas $\pi^{-1}(s) =: E$ is a curve isomorphic to $\PP^1(\CC)$, called the \emph{exceptional divisor} of the blowup. The same construction can be carried out for a point $s$ on a smooth complex surface $S$ (introducing local coordinates $(x,y)$ in a neighbourhood of $s$). One can check, see e.g. \cite[Section 2.4]{Lotus}, \cite[Section 3.2]{Wall}, that this yields a well-defined blowup morphism that is independent of the chosen coordinates, denoted by $\pi: \widetilde S \xrightarrow{} S$, where $\widetilde S$ is the blown-up surface. 
Since blowups are isomorphisms away from the corresponding centers, 
 it makes sense to choose as the center of a blowup a finite set of points $ \{ s_1, \ldots, s_\tau \} $.
\\
For a complex surface $ S $ and a finite set of points $Z:=\{s_1, \ldots, s_{\tau} \} \subseteq S $ a {\em sequence of blowups with centers lying above $Z$} is defined as a proper, birational morphism $ \pi \colon \widetilde{S} \to S $ factoring as 
$  \pi = \varphi_{\ell} \circ \cdots \circ \varphi_2 \circ \varphi_1, $ 
for some $ {\ell} > 0 $, where 
\begin{enumerate}[(1)]
	\item $ \varphi_1 \colon S^{(1)} := \Bl_{Z} (S) \to S $ is the blowup with center the given points $Z$,
	\item for $ i > 0 $, 
	$ \varphi_{i+1} \colon S^{(i+1)} \to S^{(i)} $ is the blowup of the complex surface $ S^{(i)} $ with center a finite set of points $ Z^{(i)} \subseteq E^{(i)} $,
	where $ E^{(i)} := \pi_i^{-1} (Z) \subset S^{(i)} $ is called the {\em exceptional locus} of $ \pi_i := \varphi_i \circ \cdots \circ \varphi_1 $,
	and 
	\item $ \widetilde{S} := S^{({\ell})} $.   
\end{enumerate}
Notice that the restriction $\pi_i|_{S^{(i)}\backslash E^{(i)}} \colon S^{(i)}\backslash E^{(i)} \to S \backslash Z $ is an analytic isomorphism, and $S^{(i)}$ is smooth for every $ i \in \{ 1, \ldots, \ell \} $.
\\
Since we are working with germs, we always assume that $Z=\{s\}$ is a single point.  If $C \subseteq S$ is a curve passing through $s$, then its preimage $\pi_i^{-1}(C)$ is called the \emph{total transform} of $C$ in $S^{(i)}$ and the closure of $\pi_i^{-1}(C)\backslash E^{(i)}$ in $S^{(i)}$ is called its \emph{strict transform}. 
  \\
 
 Let $C$ be a curve in a smooth complex surface $S$. Then an \emph{embedded resolution of singularities } of $C$ is a sequence of blowups 
  $\pi \colon \widetilde S \to S $ as defined above, such that
 \begin{enumerate}[(1)]
 \item $\widetilde S$ is smooth,
 \item the total transform of $C$ is a normal crossing divisor, and 
 \item the strict transform of $C$ is smooth.
 \end{enumerate}
One can show that there always exists an embedded resolution of singularities by a sequence of blowups of points, see e.g. \cite[Thm.~5.4.2]{dJP}, \cite[Thm.~3.4.4]{Wall}. Moreover, this resolution can be chosen \emph{minimal}, that is, any other embedded resolution of $C$ factors through it. For a constructive proof see \cite[p.~48f]{Wall}.

\begin{defi}
Let $C$ be a curve in a smooth complex surface $S$ and let $\pi \colon \widetilde S \xrightarrow{} S$ be an embedded resolution of singularities of $C$. Its \emph{(weighted dual) resolution graph} is a simple finite connected graph whose vertices are labeled by the irreducible components of the exceptional locus of $ \pi $ and two vertices are connected by an edge if their associated curves intersect in $\widetilde S$.
	The weight of each vertex corresponding to a component $E_j$ of the exceptional locus is the self-intersection number $E_j^2$ on $\widetilde S$.
	\\
	Sometimes the components corresponding to the strict transform of $C$ in $ \widetilde{S} $ are taken into account as additional vertices, in which case they are drawn with an arrowhead without a weight. 
	\\
When $\pi: \widetilde S \xrightarrow{} S$ is the minimal embedded resolution of singularities of $C$, then we will write $\Gamma(C)$ for its weighted dual resolution graph.
\end{defi}

\begin{example}
	\label{Ex:Resol_cusp}
	The cusp in  $S=\CC^2$ is defined by $f=y^2-x^3$. 
	One can show that a minimal embedded resolution is obtained with three blowups: 
	first we blow up the origin. 
		For explicit computations, one usually uses that the projective line $ \PP^1 (\CC) $ is covered by two affine charts.
		Using the notation of \eqref{eq:BU}, there is 
		the $ X $-chart, where the projective coordinate $ X $ is invertible, as well as the $ Y $-chart, where $ Y $ is invertible.
		In the $ X $-chart, the equation \eqref{eq:BU} can be rewritten as $ y = x \frac{Y}{X} $. 
	By setting $ y' := \frac{Y}{X} $, we see that the $ X $-chart is the affine chart with coordinates $ (x, y') $ and we have to apply the substitution $ y = x y' $ in order to determine the total transform of $ f $.
	(The situation in the $ Y $-chart is completely analogous.)
	Hence, in the $X$-chart of $S^{(1)}$,
	 the total transform is $f_1 := x^2(y'^2-x)$  and the exceptional divisor is $E_1=V(x)$. 
	In the  $ Y $-chart the strict transform is smooth and does not meet $E_1$.
	In order to lighten the notation, we apply the usual abuse of notation and write $ y $ instead of $ y' $ for the new coordinate in the $ X $-chart, for example.\\
	The center for the second blowup is the origin of the $ X$-chart, 
	which yields the total transform $f_2 :=x^2y^3(y-x)$ in the $Y$-chart of the second blowup . Here the exceptional divisors are locally given by $E_1=V(x)$, $E_2=V(y)$ and the strict transform is smooth. However, the total transform is not a normal crossing divisor, hence we need to blowup the origin of the given chart. 
	Observe that no further blowups are required in the $ Y $-chart of the second blowup.
	\\
In the $ Y $-chart of the third blowup the total transform is $f_{3,Y} =x^2y^6(1-x)$ with exceptional divisors $E_1=V(x)$ and $E_3=V(y)$, whereas in the respective $ X $-chart we obtain the total transform $f_{3,X} :=x^6y^3(y-1)$ with exceptional divisors seen $E_2=V(y)$ and $E_3=V(x)$. 
	The total transform has normal crossings and thus we have obtained an embedded resolution of singularities. 
	\\
	Using the formula \cite[Prop.~2.37]{Lotus} for computing the self-intersection numbers of the $E_i$, we obtain the dual resolution graph $\Gamma(V(y^2-x^3))$, see Fig.~\ref{fig:GammaofCusp}.
    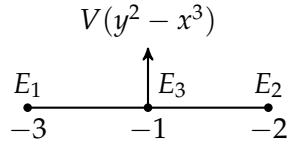
\begin{figure}[h!]
    \begin{center}
\begin{tikzpicture}[scale=0.8]
\draw [->, color=black, thick] (4,0)--(4,1);
\draw [-, color=black, thick](2,0) -- (6,0);

\node [below] at (2,0) {$-3$};
\node [below] at (4,0) {$-1$};
\node [below] at (6,0) {$-2$};

\node [above] at (2,0) {$E_{1}$};
\node [above] at (4.4,0) {$E_3$};
\node [above] at (6,0) {$E_{2}$};

\node[draw,circle, inner sep=1.pt,fill=black] at (2,0){};
\node[draw,circle, inner sep=1.pt,fill=black] at (4,0){};
\node[draw,circle, inner sep=1.pt,fill=black] at (6,0){};

\node [above] at (4,1) {$V(y^2 - x^3)$};

\end{tikzpicture}
\end{center}
 \caption{The dual resolution graph of the cusp $C=V(y^2-x^3)$.}
\label{fig:GammaofCusp}
   \end{figure}
\end{example}

\begin{example} 
	Analogously to the previous example, one determines  the dual resolution graph of the plane curve given by
		$ f=x^{11}- y^8 $,  seen in Fig.~\ref{fig:GammaofRunning}.
    \begin{figure}[h!]
    \begin{center}
\begin{tikzpicture}[scale=0.8]
\draw [->, color=black, thick] (4,0)--(4,1);
\draw [-, color=black, thick](0,0) -- (10,0);

\node [below] at (0,0) {$-4$};
\node [below] at (2,0) {$-3$};
\node [below] at (4,0) {$-1$};
\node [below] at (6,0) {$-2$};
\node [below] at (8,0) {$-3$};
\node [below] at (10,0) {$-2$};

\node [above] at (0,0) {$E_{1}$};
\node [above] at (2,0) {$E_{4}$};
\node [above] at (4.4,0) {$E_6$};
\node [above] at (6,0) {$E_{5}$};
\node [above] at (8,0) {$E_{3}$};
\node [above] at (10,0) {$E_{2}$};

\node[draw,circle, inner sep=1.pt,fill=black] at (0,0){};
\node[draw,circle, inner sep=1.pt,fill=black] at (2,0){};
\node[draw,circle, inner sep=1.pt,fill=black] at (4,0){};
\node[draw,circle, inner sep=1.pt,fill=black] at (6,0){};
\node[draw,circle, inner sep=1.pt,fill=black] at (8,0){};
\node[draw,circle, inner sep=1.pt,fill=black] at (10,0){};

\node [above] at (4,1) {$V(x^{11} - y^8)$};

\end{tikzpicture}
\end{center}
 \caption{The dual resolution graph of $C=V(x^{11}-y^8)$.}
\label{fig:GammaofRunning}
   \end{figure}
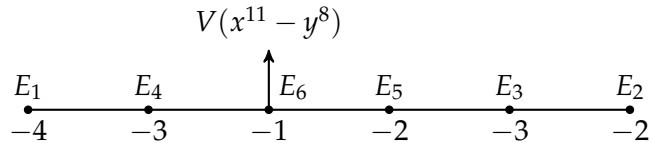
\end{example}

\subsection{Fans, toric geometry and lotus of a Newton fan}

A lotus is a simplicial complex that encodes the dual resolution graph as well as other graphs related to the resolution, such as the Eggers--Wall tree. Lotuses were introduced by Popescu-Pampu in \cite[Section 5]{PPP-cerfvolant}, see \cite{Lotus} for an extensive introduction.

In the following we consider reduced curves $C$ on a smooth complex surface $S$. In this paper we are mainly interested in simple lotuses coming from Newton fans.

For the basics of toric geometry see \cite{CLS, Fulton}. Let $N$ be a lattice of rank $2$ with chosen basis $e_1, e_2$. As usual, we denote by $N_{\RR}=N \otimes_{\ZZ}\RR$ the real vector space generated by $N$, that is, $N_{\RR}\cong \RR^2$. We denote the cone generated by $e_1, e_2$ by $\sigma_0=\langle e_1, e_2 \rangle_{\RR_{\geq 0}}$. The \emph{slope} of an element $w=c e_1 +d e_2 \in N_{\RR} \backslash \{ 0 \}$ is $\frac{d}{c}$. For $\lambda \in \QQ_{\geq 0} \cup \{ \infty \}$ one denotes by $p(\lambda)$ the unique primitive element of $N$ contained in $\sigma_0$ with slope $\lambda$. \\
A \emph{fan} $\Sigma$ of the lattice $N$ is a finite set of strictly convex (rational) cones $\sigma$ that is closed under the operation of taking faces and such that if $\sigma_1, \sigma_2$ are contained in $\Sigma$, then $\sigma_1 \cap \sigma_2$ is a face of each $\sigma_i$, for $i \in \{ 1, 2 \}$.

\begin{defi}
Let $N=\langle e_1, e_2 \rangle$ be a rank $2$ lattice with associated cone $\sigma_0 \cong \RR_{\geq 0}^{2} $. 
Let $\mc{E}=\{ \lambda_1, \ldots, \lambda_{\rho} \} \subseteq \QQ_{\geq 0}$ be a finite set. 
We define the \emph{fan of $ \mc{E} $}, $\Sigma(\mc{E})=\Sigma(\lambda_1, \ldots, \lambda_{\rho})$, 
to be the fan subdividing $\sigma_0$ by the rays $p(\lambda_i)$, where $ i \in \{ 1, \ldots, \rho \} $.
\end{defi}

\begin{example}
	\label{Ex:3durch2}
Let $\mc{E}=\{ \frac{3}{2} \}$, then $p(\frac{3}{2})=2e_1+3e_2$ and the fan $\Sigma(\frac{3}{2})$ consists of the two two-dimensional cones $\sigma_1=\langle e_1, 2e_1+3e_2\rangle_{\RR_{\geq 0}} $ and $\sigma_2=\langle 2e_1+3e_2, e_2 \rangle_{\RR_{\geq 0}}$, the three one-dimensional cones $\sigma_{11}= \langle e_1\rangle_{\RR_{\geq 0}}$, $\sigma_{12}=\sigma_{21}=\langle 2e_1+3e_2\rangle_{\RR_{\geq 0}}$, $\sigma_{22}=\langle e_2\rangle_{\RR_{\geq 0}}$ and the cone $\{0\}$.
In Fig.~\ref{Fig:fan_x3_y2}, we visualize the fan $ \Sigma(\frac{3}{2}) $.

\begin{figure}[h!]
	\begin{center}
		\begin{tikzpicture}[scale=0.8]

			\filldraw [fill=yellow!40,draw=yellow!40] (0,0) -- (2.22,3.33) -- (0,3.33) -- cycle;
			
			\filldraw [fill=red!60,draw=red!60] (0,0) -- (2.22,3.33) -- (2.22,0) -- cycle;

			\draw [-, black!50, thick] (0,0) -- (2.35,0);
			\draw [->, very thick] (0,0) -- (1,0);
			\draw [-, black!50, thick] (0,0) -- (0,3.35);
			\draw [->, very thick] (0,0) -- (0,1);
			\draw [-, black!50, thick] (0,0) -- (2.22,3.33);
			\draw [->, very thick] (0,0) -- (2,3);

			\draw [dotted] (0,0) grid (2.2,3.2);
			
			\node [below] at (2,0) {$ \sigma_{11} $}; 
			\node [left] at (0,3) {$ \sigma_{22} $}; 
			\node at (2.65,3.25) {$ \sigma_{12} $};
			\node  at (1.55,1.3) {$ \sigma_{1} $};
			\node at (0.7,2.25) {$ \sigma_{2} $};

		\end{tikzpicture}
	\end{center}
	\caption{The fan $ \Sigma(\frac{3}{2}) $.}  
	\label{Fig:fan_x3_y2} 
\end{figure}
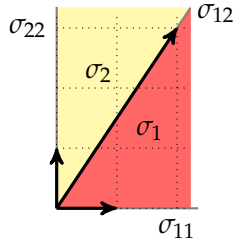     
\end{example}

We will often choose suitable coordinates at a given point of a smooth surface coming from blowups. For this we say that a \emph{cross} on a smooth surface germ $(S,s)$ is a pair $(L,L')$ of transversal smooth branches on $(S,s)$. A local coordinate system $(x,y)$ on the surface germ \emph{defines the cross} $(L,L')$ if $L=V(x)$ and $L'=V(y)$. \\

Let $f \in \mc{O}_{S,s} \cong \CC\{x,y\}$ be a nonzero element, written $f=\sum a_{ij} x^i y^j$. 
The \emph{support of $f$} is $\mathrm{supp}(f)=\{ (i,j) \in \ZZ_{\geq 0}^2: a_{ij} \neq 0 \}$.
The \emph{Newton polyhedron} $\mc{N}(f)$ is the following convex subset of $\RR_{\geq 0}^2$:
$$\mc{N}(f):= \mathrm{conv}(\mathrm{supp}(f) + \ZZ_{\geq 0}^2 ) \subseteq \RR_{\geq 0}^2 \ .$$
Here $ + $ denotes the Minkowski sum, $ A + B := \{ a+ b \mid a \in A, b \in B\} $ for $ A, B $ subsets of an additive group.
The \emph{boundary} of $\mc{N}(f)$ is denoted by $\partial \mc{N}(f)$. Note that if we choose coordinates relative to a cross $(L,L')$, then we can define the weight lattices $N:=N_{L,L'} \simeq \ZZ^2$, $M:=M_{L,L'} = N^\vee \simeq \ZZ^2$, and the cones $\sigma_0=\langle e_1, e_2 \rangle_{\RR_{\geq 0}} \simeq \RR_{\geq 0}^2 $ and $\sigma_0^\vee \simeq \RR_{\geq 0}^2 $. Then $\mathrm{supp}(f) \subseteq \sigma_0^\vee \cap M \simeq \ZZ_{\geq 0}^2$ and
$$\mc{N}(f)=\mathrm{conv}(\mathrm{supp}(f) + (  \sigma_0^\vee \cap M )  ) \subseteq \sigma_0^\vee   \ .$$
The \emph{Newton fan} $\Sigma(f)$ of $f$ is the fan in $N$ obtained by subdividing the cone $\sigma_0$ using the rays orthogonal to the compact edges of the Newton polyhedron $\mc{N}(f) \subseteq \sigma_0^\vee$ of $f$. More generally, a \emph{Newton fan} $\Sigma$ in $N$ is any fan subdividing the regular cone $\sigma_0$.

\begin{example}
	\label{Ex:cusp}
The Newton polyhedron $\mc{N}(x^3-y^2)$ has only one compact edge $\overline{(3,0),(0,2)}$. 
The orthogonal ray is generated by the primitive vector $w=2e_1+3e_2$. Thus the Newton fan is $\Sigma(f)=\Sigma(\frac{3}{2})$
of Example~\ref{Ex:3durch2}.
\end{example}
 
Suppose that $(L,L')$ defines a cross on $(S,s)$ and $C$ is a curve singularity in $(S,s)$. One can show \cite[Prop.~4.13]{Lotus} that the Newton polyhedron $\mc{N}(f)$ and the Newton fan $\Sigma(f)$ do not depend on the choice of the defining functions $x,y,f$ of the curve germs $L,L', C$.

Now we have gathered all the necessary notation to define $\Lambda(\Sigma)$, the lotus of a Newton fan $\Sigma$ (see Def.~\ref{Def:lotus}). 

\begin{DefCon} 
Let $N$ be a lattice of rank $2$ with a chosen basis $(e_1,e_2)$. The \emph{petal associated with the basis $(e_1,e_2)$} (or \emph{base petal}) is the convex and compact triangle $\delta(e_1,e_2) \subseteq N_{\RR}$ with vertices $e_1, e_2, e_1+e_2$. The line segment $[e_1,e_2]$ is called its \emph{base} (oriented from $e_1$ to $e_2$). The points $e_1$ and $e_2$ are called the \emph{basic vertices} of the petal. The segments $[e_i, e_1+e_2]$ for $ i\in \{ 1, 2 \} $ are called its \emph{lateral edges}. 
We illustrate these notions in Fig.~\ref{Fig:base_petal}(a).

\begin{figure}[h!]
	\begin{center}
		\begin{tikzpicture}[scale=1.4]
			
			\draw [fill=orange!40] (1,0) -- (0,1) -- (1,1) -- cycle;
			
			\draw [->, thick, blue] (1,0)--(0.5, 0.5);
			\draw [-, thick, blue] (0.5, 0.5)--(0,1);

			\draw [->, thick, Farbe] (1,0)--(1, 0.5);
			\draw [-, thick, Farbe] (1, 0.5)--(1,1);

			\draw [->, thick, farbe] (1,1)--(0.5, 1);
			\draw [-, thick, farbe] (0.5, 1)--(0,1);
			
			\draw [dotted] (0,0) grid (2,2.2);
			\draw [->] (0,0) -- (2.35,0);
			\draw [->] (0,0) -- (0,2.35);
			\node [below] at (1,0) {$ e_1 $}; 
			\node [left] at (0,1) {$ e_2 $}; 
			\node at (-0.15,-0.15) {$ 0 $}; 
			\node at (1.4,1.15) {$ e_1 + e_2 $};

			 \node at (-0.6,2.15) {(a)}; 
		\end{tikzpicture}
		\hspace{50pt}
			\begin{tikzpicture}[scale=1.4]
			
			\draw [fill=orange!40] (1,0) -- (0,1) -- (1,1) -- cycle;
			\draw [fill=yellow!40] (1,0) -- (1,1) -- (2,1) -- cycle;
			\draw [fill=red!40] (0,1) -- (1,1) -- (1,2) -- cycle;
			
			\draw [->, thick, blue] (1,0)--(0.5, 0.5);
			\draw [-, thick, blue] (0.5, 0.5)--(0,1);
			
			\draw [->] (1,0)--(1, 0.5);
			\draw [->] (1,1)--(0.5, 1);
			\draw [->] (1,1)--(1, 1.5);
			\draw [->] (1,2)--(0.5, 1.5);
			\draw [->] (1,0)--(1.5, 0.5);
			\draw [->] (2,1)--(1.5, 1);

			\draw [dotted] (0,0) grid (2.2,2.2);
			\draw [->] (0,0) -- (2.35,0);
			\draw [->] (0,0) -- (0,2.35);
			\node [below] at (1,0) {$ e_1 $}; 
			\node [left] at (0,1) {$ e_2 $}; 
			\node at (-0.15,-0.15) {$ 0 $};

			\node at (-0.6,2.15) {(b)};
		\end{tikzpicture}
	\end{center}
	\caption{(a) Base petal $ \delta(e_1,e_2) $.
	(b) Lotus consisting of the base petal $ \delta(e_1,e_2) $ (in orange)
	as well as the iteratively constructed petals $ \delta(e_1+e_2,e_2) $ (in red on top of $ \delta(e_1,e_2) $) and $ \delta(e_1,e_1+e_2) $ (in yellow to the right of $ \delta(e_1,e_2) $).}
	\label{Fig:base_petal} 
	\label{Fig:base_petal_plus}
\end{figure}
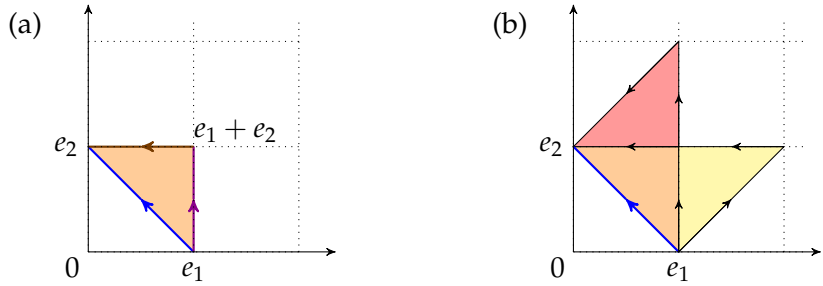     

Now one  constructs more petals iteratively: from the basis $(e_1, e_1 +e_2)$ the petal $\delta(e_1, e_1+e_2)$ may be constructed and from $(e_1+e_2,e_2)$ the petal $\delta(e_1+e_2,e_2)$ and so on. In the $n$-th step thus $2^n$ petals are added to the ones already constructed. The base of any petal $\delta$, except for the base petal $\delta(e_1,e_2)$, has a common edge with exactly one of the petals constructed in the previous step. This petal is called the \emph{parent} of $\delta$. Note that we consider bases ordered, that is, they respect the orientation of $N_{\RR}$ (sometimes these are called \emph{positive bases}). 
Fig.~\ref{Fig:base_petal_plus}(b) pictures the first two petals constructed from the base petal $ \delta (e_1, e_2) $.
\end{DefCon} 

This construction yields an infinite simplicial complex in $\sigma_0$, called the \emph{universal lotus} $\Lambda(e_1, e_2)$ of $N$ relative to the basis $(e_1, e_2)$. 
A partial view on the universal lotus is given in%
\footnote{This is a slight adaptation of the original figure in~\cite[Fig.~1.26, p.~82]{Lotus} resp.~\cite[Fig.~10, p.~321]{PPP-cerfvolant}.}
Fig.~\ref{Fig:Univ_Lotus}.
\begin{defi}
A \emph{lotus $\Lambda$ relative to $(e_1, e_2)$} 
is either the segment $[e_1,e_2]$ or the union of a non-empty set of petals of the universal lotus $\Lambda(e_1,e_2)$, stable under the operation of taking the parent of the petal. The segment $[e_1, e_2]$ is called the \emph{base} of $\Lambda$ and if $\Lambda$ is of dimension $2$, the petal $\delta(e_1, e_2)$ is called its \emph{base petal} or \emph{base triangle}.  
Further, if $\Lambda$ is a lotus relative to $(e_1, e_2)$ then a \emph{sublotus} $\Lambda' \subseteq \Lambda$ is a sub-simplicial complex of $\Lambda$ that is also a lotus relative to $(e_1,e_2)$. 
\end{defi}

\begin{figure}[h!]
	\begin{center}
		\begin{tikzpicture}[scale=0.8]
			
			\draw [fill=blue!40] (1,0) -- (0,1) -- (1,1) -- cycle;
			\draw [fill=blue!20] (1,0) -- (1,1) -- (2,1) -- cycle;
			\draw [fill=blue!20](1,0) -- (2,1) -- (3,1) -- cycle;
			\draw [fill=blue!20](1,0) -- (3,1) -- (4,1) -- cycle;
			\draw [fill=blue!20](1,0) -- (4,1) -- (5,1) -- cycle;
			\draw [fill=blue!20](1,0) -- (5,1) -- (6,1) -- cycle;
			\draw [fill=blue!20](1,0) -- (6,1) -- (7,1) -- cycle;
			\draw [fill=blue!20](1,0) -- (7,1) -- (8,1) -- cycle;
			\draw [fill=blue!20](1,0) -- (8,1) -- (9,1) -- cycle;
			\draw [fill=blue!20](1,0) -- (9,1) -- (10,1) -- cycle;
			
			\draw [fill=blue!20] (1,1) -- (2,1) -- (3,2) -- cycle;
			\draw [fill=blue!20] (1,1) -- (3,2) -- (4,3) -- cycle;
			\draw [fill=blue!20] (1,1) -- (4,3) -- (5,4) -- cycle;
			\draw [fill=blue!20] (1,1) -- (5,4) -- (6,5) -- cycle;
			\draw [fill=blue!20] (1,1) -- (6,5) -- (7,6) -- cycle;
			\draw [fill=blue!20] (1,1) -- (7,6) -- (8,7) -- cycle;
			\draw [fill=blue!20] (1,1) -- (8,7) -- (9,8) -- cycle;
			\draw [fill=blue!20] (1,1) -- (9,8) -- (10,9) -- cycle;
			
			\draw [fill=blue!20] (2,1) -- (3,2) -- (5,3) -- cycle;
			\draw [fill=blue!20] (2,1) -- (5,3) -- (7,4) -- cycle;
			\draw [fill=blue!20] (2,1) -- (7,4) -- (9,5) -- cycle;
			
			\draw [fill=blue!20] (2,1) -- (3,1) -- (5,2) -- cycle;
			\draw [fill=blue!20] (2,1) -- (5,2) -- (7,3) -- cycle;
			\draw [fill=blue!20] (2,1) -- (7,3) -- (9,4) -- cycle;
			
			\draw [fill=blue!20] (3,2) -- (5,3) -- (8,5) -- cycle;
			
			\draw [fill=blue!20] (3,2) -- (4,3) -- (7,5) -- cycle;
			\draw [fill=blue!20] (3,2) -- (7,5) -- (10,7) -- cycle;
			
			\draw [fill=blue!20] (4,3) -- (5,4) -- (9,7) -- cycle;
			
			\draw [fill=blue!20] (3,1) -- (5,2) -- (8,3) -- cycle;
			\draw [fill=blue!20] (3,1) -- (4,1) -- (7,2) -- cycle;
			\draw [fill=blue!20] (3,1) -- (7,2) -- (10,3) -- cycle;
			
			\draw [fill=blue!20] (4,1) -- (5,1) -- (9,2) -- cycle;
			
			\draw [fill=blue!20] (0,1) -- (1,1) -- (1,2) -- cycle;
			\draw [fill=blue!20] (0,1) -- (1,2) -- (1,3) -- cycle;
			\draw [fill=blue!20] (0,1) -- (1,3) -- (1,4) -- cycle;
			\draw [fill=blue!20] (0,1) -- (1,4) -- (1,5) -- cycle;
			\draw [fill=blue!20] (0,1) -- (1,5) -- (1,6) -- cycle;
			\draw [fill=blue!20] (0,1) -- (1,6) -- (1,7) -- cycle;
			\draw [fill=blue!20] (0,1) -- (1,7) -- (1,8) -- cycle;
			\draw [fill=blue!20] (0,1) -- (1,8) -- (1,9) -- cycle;
			\draw [fill=blue!20] (0,1) -- (1,9) -- (1,10) -- cycle;
			
			\draw [fill=blue!20] (1,1) -- (1,2) -- (2,3) -- cycle;
			\draw [fill=blue!20] (1,1) -- (2,3) -- (3,4) -- cycle;
			\draw [fill=blue!20] (1,1) -- (3,4) -- (4,5) -- cycle;
			\draw [fill=blue!20] (1,1) -- (4,5) -- (5,6) -- cycle;
			\draw [fill=blue!20] (1,1) -- (5,6) -- (6,7) -- cycle;
			\draw [fill=blue!20] (1,1) -- (6,7) -- (7,8) -- cycle;
			\draw [fill=blue!20] (1,1) -- (7,8) -- (8,9) -- cycle;
			\draw [fill=blue!20] (1,1) -- (8,9) -- (9,10) -- cycle;
			
			\draw [fill=blue!20] (1,2) -- (2,3) -- (3,5) -- cycle;
			\draw [fill=blue!20] (1,2) -- (3,5) -- (4,7) -- cycle;
			\draw [fill=blue!20] (1,2) -- (4,7) -- (5,9) -- cycle;
			
			\draw [fill=blue!20] (1,2) -- (1,3) -- (2,5) -- cycle;
			\draw [fill=blue!20] (1,2) -- (2,5) -- (3,7) -- cycle;
			\draw [fill=blue!20] (1,2) -- (3,7) -- (4,9) -- cycle;
			
			\draw [fill=blue!20] (2,3) -- (3,5) -- (5,8) -- cycle;
			
			\draw [fill=blue!20] (2,3) -- (3,4) -- (5,7) -- cycle;
			\draw [fill=blue!20] (2,3) -- (5,7) -- (7,10) -- cycle;
			
			\draw [fill=blue!20] (3,4) -- (4,5) -- (7,9) -- cycle;

			\draw [fill=blue!20] (1,3) -- (2,5) -- (3,8) -- cycle;
			\draw [fill=blue!20] (1,3) -- (1,4) -- (2,7) -- cycle;
			\draw [fill=blue!20] (1,3) -- (2,7) -- (3,10) -- cycle;
			
			\draw [fill=blue!20] (1,4) -- (1,5) -- (2,9) -- cycle;
			
			\draw [dotted] (0,0) grid (10.5,10.5);
			\draw [->] (0,0) -- (10.7,0);
			\draw [->] (0,0) -- (0,10.7);
			\node [below] at (1,0) {$ e_1 $}; 
			\node [left] at (0,1) {$ e_2 $}; 
			\node at (-0.15,-0.15) {$ 0 $};

		\end{tikzpicture}
	\end{center}
	\caption{First parts of the universal lotus $\Lambda(e_1,e_2)$ of $ N = \ZZ^2 $ relative to $(e_1, e_2)$.}  
	\label{Fig:Univ_Lotus} 
\end{figure}
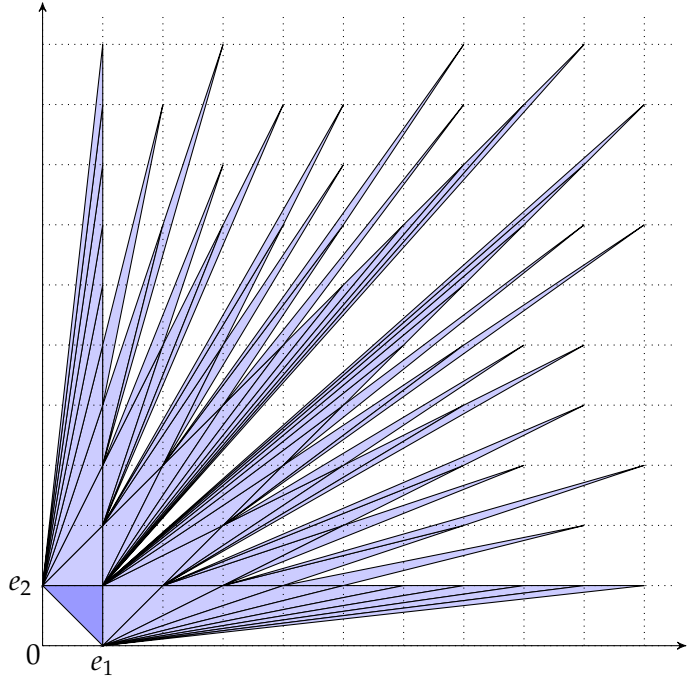

\begin{defi} \label{Def:lotus}
Let  $N$ be a lattice of rank $2$ with basis $(e_1,e_2)$. 
\begin{enumerate}
\item For $\lambda \in \QQ_{> 0}$ its \emph{lotus} $\Lambda(\lambda)$ is the union of petals of the universal lotus $\Lambda(e_1,e_2)$ which intersect the ray 
of slope $\lambda$. For $\lambda=0$ or $\lambda=\infty$, we set $\Lambda(\lambda):=[e_1,e_2]$.
\item Let $\mc{E}$ be a finite subset of $\QQ_{\geq 0} \cup \{ \infty \} $.
 The \emph{lotus} of $ \mc{E} $, denoted by $\Lambda(\mc{E})$, is the union $\bigcup_{\lambda \in \mc{E}} \Lambda(\lambda)$. If $\Sigma=\Sigma(\mc{E})$ is a Newton fan, its lotus is defined as $\Lambda(\Sigma):=\Lambda(\mc{E})$. Sometimes, a lotus of a Newton fan is called a \emph{Newton lotus}.
\end{enumerate}
\end{defi}

In the present article, we will only consider Newton lotuses, for more general lotuses with several petals see \cite{Lotus}.

From the lotus one can read off the resolution graph of a plane curve singularity (see Section~\ref{Sub:LotusGraph}) as its lateral boundary. For this we need some more notation: 

\begin{defi}
Let $\Lambda$ be a Newton lotus. If $\Lambda=[e_1, e_2]$, then set $\partial_+\Lambda:=[e_1, e_2]$, and if $\Lambda \neq [e_1, e_2]$, then let $\partial_+\Lambda$ be the compact and connected polygonal line that is the complement of the open segment $(e_1,e_2)$ in the boundary of the lotus $\Lambda$. The polygonal line $\partial_+ \Lambda \subseteq \Lambda$ is called the \emph{lateral boundary} of the lotus $\Lambda$. \\
If $\lambda \in [0 , \infty]$, then one can associate to it a unique point on $\partial_+\Lambda$ with slope $\lambda$, denoted by $p_\Lambda(\lambda)$. For a Newton lotus $\Lambda=\Lambda(\mc{E})$, where  $\mc{E} \subseteq \QQ_{\geq 0} \cup \{ \infty \}$ is some finite set, for any $\lambda \in \mc{E}$ the point $p_{\Lambda(\mc{E})}(\lambda)$ is called the \emph{marked point of $\lambda$} in $\Lambda(\mc{E})$. The lotus $\Lambda(\mc{E})$ together with its marked points $p_{\Lambda(\mc{E})}(\lambda)$ for all $\lambda \in \mc{E}$ is called a \emph{marked lotus}. \\
We call a vertex different from $e_1$ and $e_2$ a \emph{pinching point} of $\Lambda$ if it belongs to only one petal of it. Moreover, if $\Lambda \neq [e_1, e_2]$, then the lattice point connected to $e_1$ (resp.~$e_2$) in the lateral boundary $\partial_+\Lambda$ is called the \emph{first interior point} (resp.~\emph{last interior point}) of the lateral boundary. The vertices $e_1$ and $e_2$ of $\Lambda$ are called \emph{basic} vertices, and lattice points contained in the lateral boundary of $\Lambda$ are called \emph{lateral vertices}. Note that the lateral vertices of $\Lambda$ are given as the vertices $\Lambda \cap N$, which are not basic.
\end{defi}

Marked points make it possible to distinguish lotuses:

\begin{example}
	\label{Ex:marked_unmarked}
Consider the Newton lotuses $\Lambda_i=\Lambda(\mc{E}_i)$ with $\mc{E}_1=\{ \frac{3}{2}\}$ and $\mc{E}_2=\{ \frac{3}{2}, \frac{2}{1}, \frac{1}{1} \}$. As unmarked lotuses, we have $\Lambda_1=\Lambda_2$, but $\Lambda_1$ has only one marked point (the pinching point $(2,3)$), whereas $\Lambda_2$ has three marked points.
In Fig.~\ref{Fig:lotus_marked}, we illustrate the difference by drawing the marked points.

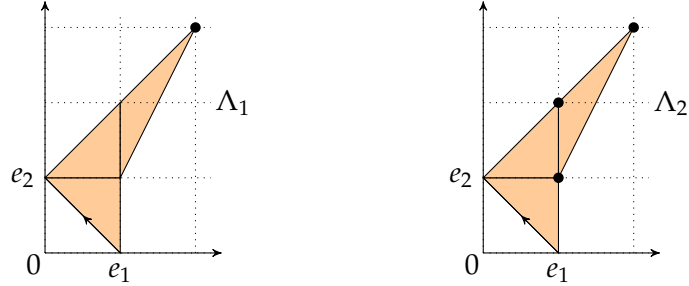
\begin{figure}[h!]
	\begin{center}
		\begin{tikzpicture}[scale=1]
			
			\filldraw [fill=orange!40,draw=black] (1,0) -- (0,1) -- (1,1) -- cycle;
			\filldraw [fill=orange!40,draw=black] (0,1) -- (1,1) -- (1,2) -- cycle;
			\filldraw [fill=orange!40,draw=black] (1,1) -- (1,2) -- (2,3) -- cycle;
			
			\draw [->] (1,0)--(0.5, 0.5);
			\draw [-] (0.5, 0.5)--(0,1);
			
			\draw [dotted] (0,0) grid (2.2,3.2);
			\draw [->] (0,0) -- (2.35,0);
			\draw [->] (0,0) -- (0,3.35);

			\node [below] at (1,0) {$ e_1 $}; 
			\node [left] at (0,1) {$ e_2 $}; 
			\node at (-0.15,-0.15) {$ 0 $};

			\node at (2.5,2) {$ \Lambda_1 $};
			
			\draw (2,3) node[fill=black,circle,inner sep=0.049cm] {} circle (0.05cm);

		\end{tikzpicture}
		\ \hspace{2cm} \ 
		\begin{tikzpicture}[scale=1]
			
			\filldraw [fill=orange!40,draw=black] (1,0) -- (0,1) -- (1,1) -- cycle;
			\filldraw [fill=orange!40,draw=black] (0,1) -- (1,1) -- (1,2) -- cycle;
			\filldraw [fill=orange!40,draw=black] (1,1) -- (1,2) -- (2,3) -- cycle;
			
			\draw [->] (1,0)--(0.5, 0.5);
			\draw [-] (0.5, 0.5)--(0,1);
			
			\draw [dotted] (0,0) grid (2.2,3.2);
			\draw [->] (0,0) -- (2.35,0);
			\draw [->] (0,0) -- (0,3.35);

			\node [below] at (1,0) {$ e_1 $}; 
			\node [left] at (0,1) {$ e_2 $}; 
			\node at (-0.15,-0.15) {$ 0 $};
			
			\draw (2,3) node[fill=black,circle,inner sep=0.049cm] {} circle (0.05cm);	        
			
			\draw (1,1) node[fill=black,circle,inner sep=0.049cm] {} circle (0.05cm);

			\draw (1,2) node[fill=black,circle,inner sep=0.049cm] {} circle (0.05cm);

			\node at (2.5,2) {$ \Lambda_2 $};

		\end{tikzpicture}
	\end{center}
	\caption{Two lotuses that only differ by their marked points, which are marked as bullet points.}  
	\label{Fig:lotus_marked} 
	\end{figure}
\end{example}

Thus we arrive at the notion of the \emph{Newton lotus} of $f$: 
Let $(S,s)$ be a smooth surface germ. Choose local coordinates $( x,y ) $ at $s$ (so they define a cross $(L,L')$) and let $f \in \mc{O}_{S,s} \cong \CC\{x,y\}$ be defining a curve germ $(C,s)$. Consider the lattices $N=N_{L,L'}$ with basis $(e_1, e_2)$ and $M=N^\vee$ and the cone $\sigma_0=\langle e_1 , e_2 \rangle_{\RR_{\geq 0}}$ in $N_{\RR}$. Then the exponents of the monomials of $f$ define the Newton polyhedron   $\mc{N}(f) \subseteq \sigma_0^\vee \subseteq M_{\RR}$. The orthogonal rays $w_K$ to the compact edges $K$ of $\mc{N}(f)$ subdivide the cone $\sigma_0$ and yield the Newton fan $\Sigma(f)$. 
These rays give us the set of slopes 
$$\mc{E}=\{ \lambda \in \QQ_+: \lambda \text{ is the slope of some } w_K \} \cup \{ 0, \infty \}\ . $$
 Then we  associate the Newton lotus $\Lambda(f):=\Lambda(\mc{E})$ to $f$. In Thm.~\ref{Thm:Lotus-dual} we will see that for Newton non-degenerate $f$ this lotus encodes the dual resolution graph of $C$.
 
 \begin{example}
 Let $f= x^3-y^2$, then $f$ defines the cusp $C$. We have $\Lambda(f)=\Lambda( \frac{3}{2})$. For $g=x^6+x^4y+xy^3+y^4$ we get $\Lambda(g)=\Lambda( \frac{3}{2}, \frac{2}{1}, \frac{1}{1})$. Computing the minimal resolution of singularities of both curves, one sees that the dual resolution graph in both cases is an $A_3$-diagram with the same self-intersection numbers. However, $V(g)$ is locally reducible, so we can distinguish it by drawing two additional arrowheads corresponding to the marked points $ p_{\Lambda(g)} (\frac{1}{1}) $ and $ p_{\Lambda(g)} (\frac{2}{1}) $. 
 \end{example}

\begin{Bem}
	\label{Rk:charts}
	One may interpret a lotus as a sequence of blowups of points. 
		The base $ [e_1, e_2] $ represents the initial situation, 
		the segment $ [e_1, e_1+e_2] $ corresponds to the $ Y $-chart, while $ [e_2, e_1+ e_2] $ is the $ X $-chart. 
		This can be iterated so that any newly added triangle in a lotus can be interpreted as a point blowup. 
		Using the orientation on the edges the charts can be assigned to the new edges.  
		\\
		Indeed, it may be observed that the lotus $ \Lambda_1 $ in Fig.~\ref{Fig:lotus_marked}
		which corresponds to the cusp defined by the vanishing locus of $ x^3 - y^2 $ 
		describes the embedded resolution
		of the curve given by $ xy (x^3 - y^2) = 0 $ (cf.~Example~\ref{Ex:Resol_cusp}).
		Notice that the edges of the lateral boundary of $ \Lambda_1 $ corresponds to the final charts of the desingularization.	
\end{Bem}

\subsection{Lotus of a toric resolution of a Newton non-degenerate curve and the dual resolution graph}
\label{Sub:LotusGraph}

In this section we briefly recall toric resolutions of curves by subdivision of fans and how to obtain the lotus and the dual resolution graph.

Recall that a cone $\sigma$ in a lattice $N$ is \emph{regular} if it can be generated by a subset of a basis of $N$. Consequently, a fan $\Sigma$ is \emph{regular} if all its cones are regular. Here note that we set $\langle \varnothing \rangle_{\RR_{\geq 0} }:=\{ 0 \}$, so that $\{ 0 \}$ is a regular cone. 

For $2$-dimensional lattices $N$, there is a minimal regular subdivision for any fan $\Sigma$ in $N$, that is, any other regular subdivision refines it, see \cite[Prop.~1.19]{OdaConvexBodies}. %Lotus, prop 3.7
Thus for a $2$-dimensional fan $\Sigma$ in $N$, define the \emph{regularization }$\Sigma^{\reg}$ 
as the minimal regular subdivision of $\Sigma$.

The regularization  $\sigma^\reg$ of a $2$-dimensional strictly convex cone $\sigma$ in a lattice $N$ of rank $2$ is obtained by looking at the compact faces of the boundary of the convex hull of $(\sigma \cap N)\backslash \{ 0 \}$ (see \cite[Prop.~1.19]{OdaConvexBodies}). Then use the rays given by the primitive integral vectors on this boundary to subdivide the cone. 
For a fan $\Sigma$, its regularization is given as the union of the regularizations of its cones. 

\begin{example} \label{Ex:cuspfan} 
	Let $\Sigma(\frac{3}{2})$ be the Newton fan of Examples~\ref{Ex:3durch2} and~\ref{Ex:cusp}. 
	Its regularization is given by introducing the additional rays through $p(\frac{1}{1})$ and $p(\frac{2}{1})$. 
	Then $\Sigma^\reg(\frac{3}{2})$ consists of four $2$-dimensional cones as depicted in Fig.~\ref{Fig:regfancusp}

	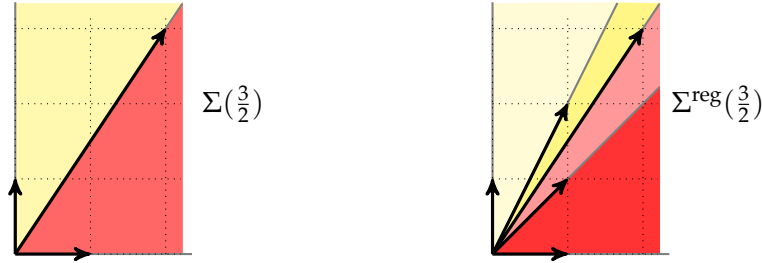
\begin{figure}[h!]
		\begin{center}
			
			\begin{tikzpicture}[scale=1]

				\filldraw [fill=yellow!40,draw=yellow!40] (0,0) -- (2.22,3.33) -- (0,3.33) -- cycle;
				
				\filldraw [fill=red!60,draw=red!60] (0,0) -- (2.22,3.33) -- (2.22,0) -- cycle;

				\draw [-, black!50, thick] (0,0) -- (2.35,0);
				\draw [->, very thick] (0,0) -- (1,0);
				\draw [-, black!50, thick] (0,0) -- (0,3.35);
				\draw [->, very thick] (0,0) -- (0,1);
				\draw [-, black!50, thick] (0,0) -- (2.22,3.33);
				\draw [->, very thick] (0,0) -- (2,3);

				\draw [dotted] (0,0) grid (2.2,3.2);

				\node at (2.9,2) {$ \Sigma (\frac{3}{2}) $};

			\end{tikzpicture}
			\ \hspace{2.5cm} \
			\begin{tikzpicture}[scale=1]

				\filldraw [fill=yellow!20,draw=yellow!20] (0,0) -- (1.66,3.34) -- (0,3.33) -- cycle;
				
				\filldraw [fill=yellow!60,draw=yellow!60] (0,0) -- (2.22,3.33)  -- (1.66,3.34)  -- cycle;
				
				\filldraw [fill=red!40,draw=red!40] (0,0) -- (2.22,3.33) --  (2.22,2.22) -- cycle;
				
				\filldraw [fill=red!80,draw=red!80] (0,0) --  (2.22,2.22) -- (2.22,0) -- cycle;
				
				\draw [-, black!50, thick] (0,0) -- (2.35,0);
				\draw [->, very thick] (0,0) -- (1,0);
				\draw [-, black!50, thick] (0,0) -- (0,3.35);
				\draw [->, very thick] (0,0) -- (0,1);
				\draw [-, black!50, thick] (0,0) -- (2.22,3.33);
				\draw [->, very thick] (0,0) -- (2,3);

				\draw [-, black!50, thick] (0,0) -- (2.22,2.22);
				\draw [->, very thick] (0,0) -- (1,1);
				
				\draw [-, black!50, thick] (0,0) -- (1.66,3.34);
				\draw [->, very thick] (0,0) -- (1,2);
				
				\draw [dotted] (0,0) grid (2.2,3.2);

				\node at (3,2) {$ \Sigma^\reg (\frac{3}{2}) $};
				
			\end{tikzpicture}	
		\end{center}
		\caption{The fan $ \Sigma(\frac{3}{2}) $ (left) and its regularization $ \Sigma^\reg (\frac{3}{2}) $ (right).}  
		\label{Fig:regfancusp} 
	\end{figure} 
\end{example}

To any cone $\sigma$ (resp.~fan $\Sigma$) one can associate the \emph{(affine) toric variety} $X_\sigma=\Spec(\CC[\sigma^\vee \cap M])$ (resp. the \emph{(projective) toric variety} $X_\Sigma$), see e.g.~\cite[Section 1.3.2]{Lotus} for details. 
One can further define toric morphisms and modifications, we refer to \cite[Section 3.3]{Lotus}, 
and one can show that a subdivision of the cone $\sigma_0$ given by a fan $\Sigma$ yields a equivariant birational morphism $\psi^{\Sigma}_{\sigma_0} \colon X_{\Sigma} \xrightarrow{} X_{\sigma_0} $. 
The preimage of $0 \in X_{\sigma_0}$ is the exceptional divisor of $\psi^\Sigma_{\sigma_0}$. 
Furthermore, one obtains the minimal resolution of a toric surface via (cf. \cite[Prop.~3.28]{Lotus}): For a non-regular cone $\sigma$ in the rank $2$ lattice $N$, the toric modification $\psi^{\sigma^\reg}_{\sigma} \colon X_{\sigma^\reg} \xrightarrow{} X_\sigma$ is the minimal resolution of the affine toric surface $X_\sigma$. Consequently, for any fan $\Sigma$ in $N$, the toric modification $\psi^{\Sigma^\reg}_{\Sigma} \colon X_{\Sigma^\reg} \xrightarrow{} X_\Sigma$ is the minimal resolution of singularities of $X_\Sigma$. 

Coming back to curves, one can obtain a resolution of a Newton non-degenerate curve $C$ (cf.~Def.~\ref{Def:nnd}) on a complex surface $S$ by a toric modification. 
However, to be precise, one has to work with \emph{toroidal varieties and modifications in the toroidal category} (see \cite[Section 3.4]{Lotus}). The objects in this category are defined as follows: a \emph{toroidal variety} is a pair $(S, \partial S)$, where $S$ is a normal complex variety and $\partial S$ is a reduced divisor on $S$, such that the germ of $(S, \partial S)$ at any point $s \in S$ is locally analytically isomorphic to  $(X_\sigma, \partial X_\sigma)$, the germ of an affine toric variety $X_\sigma$ and its boundary $\partial X_\sigma$. A morphism $\psi\colon (S_2, \partial S_2) \xrightarrow{} (S_1, \partial S_1)$ between toroidal varieties is a complex analytic morphism $\psi: S_2 \xrightarrow{} S_1$ such that $\psi^{-1}(\partial S_1) \subseteq \partial S_2$. Such a morphism is a toroidal modification if the underlying morphism $\psi$ is a modification as defined in \cite[Def.~2.31]{Lotus}.

Let  $(S,s)$ be the germ of a smooth complex surface. We choose a cross on $(L, L')$ on $(S,s)$ (giving us the coordinates $(x,y)$) and define the lattices $N_{L,L'}$, $M_{L,L'}$ and cone $\sigma_0$ as before. Any subdivision $\Sigma$ of $\sigma_0$ yields an analytic modification $\psi_{L,L'}^\Sigma \colon S_\Sigma \xrightarrow{} S$ of $S$. If we set $\partial S:=L + L'$ and $\partial S_\Sigma:=\psi^{-1}(L + L')$, then $\psi_{L,L'}^\Sigma \colon (S_\Sigma, \partial S_\Sigma) \xrightarrow{} (S, L + L')$ is a toroidal modification, called the \emph{modification of $S$ associated with $\Sigma$ relative to the cross $(L, L')$}.

For a curve $ C $ on $ (S,s) $,
the \emph{Newton modification of $S$ defined by $C$ relative to the cross $(L,L')$} is defined as $\psi_{L,L'}^C: (S_{\Sigma_{L,L'}(C)}, \partial S_{\Sigma_{L,L'}(C)}) \xrightarrow{} (S, L + L')$, the modification of $S$ associated with $\Sigma_{L,L'}(C)$ relative to the cross $(L,L')$, where $\psi_{L,L'}^C:=\psi_{L,L'}^{\Sigma_{L,L'}(C)}$,
for $ \Sigma_{L,L'}(C) := \Sigma (f) $ the Newton fan of $ C $ relative to the cross $ (L,L') $ and $ f \in \mathcal{O}_{S,s} $ is a local equation for $ C  $ at $ s $.
One also denotes the strict transform of $C$ under $\psi_{L,L'}^C$ by $C_{L,L'}$.\\

A toroidal modification $\pi \colon (S_{\Sigma}, \partial S_{\Sigma}) \xrightarrow{} (S, L+L')$  
is called a \emph{toroidal pseudo-resolution} if 
\begin{enumerate}
\item the boundary $\partial S_{\Sigma}$ of $S_{\Sigma}$ contains the reduction of the total transform $\pi^*(C)$ of $C$;
\item the strict transform of $C$ under $\pi$ does not contain singular points of $S_\Sigma$.
\end{enumerate}
If moreover $S_\Sigma$ is smooth, then $\pi$ is called a \emph{toroidal embedded resolution}. 

In \cite[Algorithm 4.22]{Lotus} an algorithm for a toroidal pseudo-resolution is given. It is also explained how to get an embedded toroidal resolution from a toroidal pseudo-resolution. We mostly care about the special case of Newton non-degenerate curves, so we will first introduce this notion and then say more about the algorithm.

\begin{defi} \label{Def:nnd}
Let $(L,L')$ be a cross in $S$ and let $C=V(f)$ be a curve in $S$. Here $f \in \CC\{x,y\}$ with coordinates $(x,y)$ given by the cross. Then $f$ is called \emph{Newton non-degenerate} if all the restrictions $f_K$ of $f$ to the compact edges $K$ of the Newton polyhedron $\mc{N}_{L,L'}(f)$ define smooth curves on the torus $(\CC^*)^2_{x,y}$. Here we write $(\CC^*)^2_{x,y}$ for the torus $(\CC^*)^2$ to reflect the choice of cross $(x,y)$.
\end{defi}

Note that $f$ is Newton non-degenerate exactly when $\Sigma_{L,L'}(f)$ already yields a toroidal pseudo-resolution of $C$,  
see \cite[Prop.~4.20]{Lotus}. Further, the regularization $\Sigma^{\reg}_{L,L'}(f)$ yields the minimal embedded resolution in this case, see \cite[Prop.~4.29]{Lotus}. 

\begin{example}
We continue with $f=x^3-y^2$, that is $C=V(f)$. Choosing the cross $L=V(x), L'=V(y)$ we have $\Sigma:= \Sigma_{L,L'}(C)=\langle e_1, 2e_1+3e_2, e_2 \rangle_{\RR_{\geq 0}}$. We calculate the strict transform $C_{L,L'}$ under $\pi\colon (S_\Sigma, \partial S_\Sigma) \xrightarrow{} (S, L+L')$: the Newton fan consists of two $2$-dimensional cones, that correspond to the charts $k[\frac{x^3}{y^2},y]$ and $k[x, \frac{y^2}{x^3}]$. In the first chart the total transform of $f$ is (with new coordinates $X=\frac{x^3}{y^2}$ and $Y=y$) $f'=XY^2-Y^2=Y^2(X-1)$. 
The strict transform is smooth and has normal crossings with the exceptional divisor. Similarly, in the other chart, the total transform is given as $f'=X'^3-X'^3Y'=X'^3(1-Y')$ (with coordinates $X'=x, Y'=\frac{y^2}{x^3}$) , so again the strict transform is smooth and has normal crossings with the exceptional divisor. Thus $\pi$ is a toroidal pseudo-resolution, which is not an embedded resolution, since $\Sigma$ is not regular. \\
An embedded toroidal resolution is obtained by the regularization $ \Sigma^\reg $ of  Example~\ref{Ex:cuspfan}. 
Since it is a subdivision of $ \Sigma $, the strict transforms are all smooth and have normal crossings with the exceptional divisor.
\end{example}

Let $f \in \CC\{x,y\}$ be Newton non-degenerate and denote by $\Lambda(f)$ its Newton lotus. By \cite[Thm.~5.29]{Lotus} the dual resolution graph of the minimal resolution of $V(f)$ can be read off the lateral boundary $\partial_+\Lambda(f)$:

\begin{Thm}[Lotus -- dual resolution graph correspondence]
	\label{Thm:Lotus-dual}
Let $C=V(f)$ be the germ of a Newton non-degenerate reduced curve singularity in the smooth complex surface $(S,s)$. 
Let $\pi\colon (S_{\Sigma}, \partial S_{\Sigma}) \xrightarrow{} (S, L+L')$ be a toroidal pseudo-resolution as above and let $ \pi^\reg \colon (S^\reg_{\Sigma}, \partial S^\reg_{\Sigma}) \xrightarrow{} (S, L+L')$ be the corresponding minimal embedded resolution of $ C $ (obtained by taking the coarsest refinement of the Newton fan $\Sigma(f)$) and let $\Lambda(f)$ be the corresponding Newton lotus. Then
\begin{enumerate}
\item The basic vertices $e_1$ and $e_2$ of $\Lambda(f)$ represent the branches $L$ and $L'$. The lateral vertices of $\Lambda(f)$ correspond to the irreducible components $E_k$ of the exceptional divisor $(\pi^{\reg})^{-1}(s)=\bigcup_{k=1}^mE_k$.
\item The lotus $\Lambda(f)$ corresponds to a triangulated $(m+2)$-gon $P$.
\item The lateral boundary $\partial_+\Lambda(f)$ is the dual graph of the boundary divisor $\partial S_{\Sigma}^{\reg}$. 
 The self-intersection number $E_k^2$ is given by the opposite of the number of triangles in $P$ incident to the vertex $E_k$.
\end{enumerate}
\end{Thm}

\begin{proof}
See \cite[Thm.~5.29]{Lotus}.
\end{proof}

\begin{example}[Lotus associated to a continued fraction $\lb b_1, \ldots, b_r \rb$] \label{Ex:lotus-continuedfraction}
Here we explicitly give the lotuses for continued fractions: they are characterized by having precisely one pinching point. Equivalently, the triangulated polygon $P$ corresponding to such a lotus has precisely two ears, where one of them is at the vertex $(1,0)$ of the base petal (this comes from our convention $\frac{n}{q} \geq 1$). 
Let $\frac{n}{q}=\lb b_1, \ldots, b_r \rb$ and $\frac{n}{n-q}=\lb b_1', \ldots, b_s' \rb$. The continued fractions are related by the triangulated polygon of Fig.~\ref{Fig:Kidoh-duality}. The quiddity sequence of the polygon corresponding to $\Lambda(\frac{n}{q})$ is 
$( b_1, \ldots, b_r, 1, b_s', b_{s-1}', \ldots, b_1',1 )$ and the dual resolution graph of the curve $C=V(x^n-y^q)$ is 
shown in Fig.~\ref{fig:GammaofContFrac}.
  
    \begin{figure}[h!]
    \begin{center}
\begin{tikzpicture}[scale=0.8]
	
\draw [->, color=black, thick] (6,0)--(6,1);
\draw [-, color=black, thick](0,0) -- (2.4,0);
\draw [-, color=black, dotted](2.4,0) -- (3.6,0);
\draw [-, color=black, thick](3.6,0) -- (10.4,0);
\draw [-, color=black, dotted](10.4,0) -- (11.6,0);
\draw [-, color=black, thick](11.6,0) -- (12,0);

\node [below] at (0,0) {{$-b_2$}};
\node [below] at (2,0) {{$-b_3$}};
\node [below] at (4,0) {{$-b_r$}};
\node [below] at (6,0) {$-1$};
\node [below] at (8,0) {{$-b_s'$}};
\node [below] at (10,0) {{$-b_{s-1}'$}};
\node [below] at (12,0) {{$-b_1'$}};

\node[draw,circle, inner sep=1.pt,fill=black] at (0,0){};
\node[draw,circle, inner sep=1.pt,fill=black] at (2,0){};
\node[draw,circle, inner sep=1.pt,fill=black] at (4,0){};
\node[draw,circle, inner sep=1.pt,fill=black] at (6,0){};
\node[draw,circle, inner sep=1.pt,fill=black] at (8,0){};
\node[draw,circle, inner sep=1.pt,fill=black] at (10,0){};
\node[draw,circle, inner sep=1.pt,fill=black] at (12,0){};

\node [above] at (6,1) {$V(x^{n} - y^q)'$};

\end{tikzpicture}
\end{center}
 \caption{The dual resolution graph of $C=V(x^{n}-y^q)$.}
\label{fig:GammaofContFrac}
\end{figure}
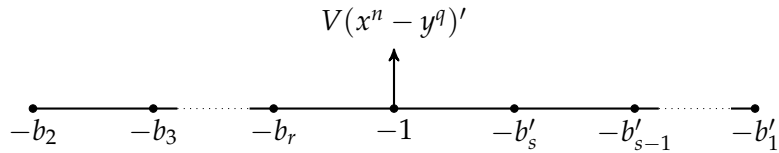
For our running example $\frac{n}{q}=\frac{11}{8}=\lb 2,2,3,2 \rb$ there is a schematic picture of the corresponding lotuses $\Lambda(\frac{n}{q})$ and $\Lambda(\frac{n}{n-q})$ in Fig.~\ref{Fig:11_8}.
Observe that $\Lambda(\frac{11}{8}) $ is one way to embed the triangulated polygon of Fig.~\ref{Fig:running} into the universal lotus $ \Lambda(e_1, e_2) $.

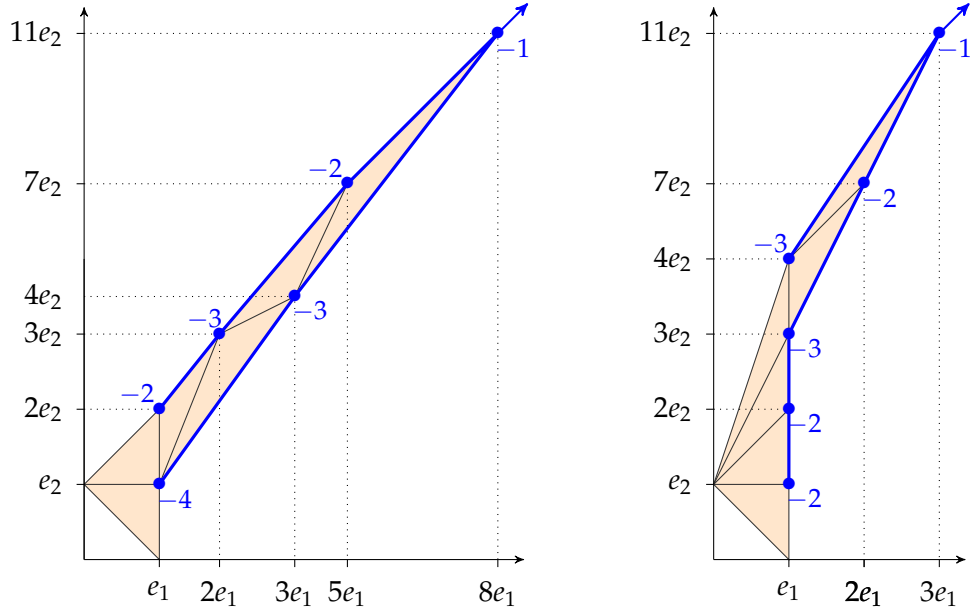
\begin{figure}[h!]
	\begin{center}
		\begin{tikzpicture}[scale=1]
			\filldraw [fill=orange!20,draw=black!80] 
			(1,0) --
			(1,1) --
			(2.8,3.5) -- %eigentlich (3,4) 
			(5.5,7) -- %eigentlich (8,11) 
			(3.5,5) -- %eigentlich (5,7)
			(1.8,3) -- %eigentlich (2,3) 
			(1,2) -- 
			(0,1) --
			cycle;
			
			\draw[black!80] (0,1) -- (1,1);
			\draw[black!80] (1,1) -- (1,2);
			\draw[black!80] (1,1) -- (1.8,3);
			\draw[black!80] (2.8,3.5) -- (1.8,3);
			\draw[black!80] (2.8,3.5) -- (3.5,5);

			\draw [->] (0,0) -- (5.85,0);
			\draw [->] (0,0) -- (0,7.35);
			\draw (0,0) -- (0,4);

			\draw  (0,7) -- (-0.1,7);
			\node [left] at (-0.15, 7) {$ 11 e_2 $};
			\draw[dotted] (0,7) -- (5.5,7) -- (5.5,0);
			\draw  (5.5,0) -- (5.5,-0.1);
			\node [below] at (5.5,-0.15) {$ 8e_1 $};

			\draw  (0,5) -- (-0.1,5);
			\node [left] at (-0.15, 5) {$ 7 e_2 $};
			\draw[dotted] (0,5) -- (3.5,5) -- (3.5,0);
			\draw  (3.5,0) -- (3.5,-0.1);
			\node [below] at (3.5,-0.15) {$ 5e_1 $};
			
			\node [left] at (-0.15, 3.5) {$ 4 e_2 $};
			\draw[dotted] (0,3.5) -- (2.8,3.5) -- (2.8,0);
			\draw  (2.8,0) -- (2.8,-0.1); 
			\node [below] at (2.8,-0.15) {$ 3e_1 $};
			
			\draw  (0,3) -- (-0.1,3);
			\node [left] at (-0.15, 3) {$ 3 e_2 $};
			\draw[dotted] (0,3) -- (1.8,3) -- (1.8,0);
			\draw  (1.8,0) -- (1.8,-0.1);
			\node [below] at (1.8,-0.15) {$ 2e_1 $};

			\draw  (0,2) -- (-0.1,2);
			\node [left] at (-0.15, 2) {$ 2 e_2 $};
			\draw[dotted] (0,2) -- (1,2);
			
			\draw  (0,1) -- (-0.1,1);
			\node [left] at (-0.15, 1) {$ e_2 $};

			\draw  (1,0) -- (1,-0.1);
			\node [below] at (1,-0.15) {$ e_1 $};

			% dual graph
			\draw [very thick, blue] (1,1) --
			(1,1) --
			(2.8,3.5) -- %eigentlich (3,4) 
			(5.5,7) -- %eigentlich (8,11) 
			(3.5,5) -- %eigentlich (5,7)
			(1.8,3) -- %eigentlich (2,3) 
			(1,2) ;

			\node at (1,1) {\color{blue} $ \bullet $};
			\node at (1.2,0.8) {\color{blue} \small $ - 4 $};
			
			\node at (2.8,3.5) {\color{blue} $ \bullet $};
			\node at (3,3.3) {\color{blue} \small $ - 3 $};
			
			\node at (5.5,7) {\color{blue} $ \bullet $};
			\node at (5.7,6.8) {\color{blue} \small $ - 1 $};
			\draw[->, thick, blue] (5.5,7) -- (5.9,7.4);
			
			\node at (3.5,5) {\color{blue} $ \bullet $};
			\node at (3.2,5.2) {\color{blue} \small $ - 2 $};
			
			\node at (1.8,3) {\color{blue} $ \bullet $};
			\node at (1.6,3.2) {\color{blue} \small $ - 3 $};
			
			\node at (1,2) {\color{blue} $ \bullet $};	
			\node at (0.7,2.2) {\color{blue} \small $ - 2 $};
			 
		\end{tikzpicture}	
		\hspace{1cm}
		\begin{tikzpicture}[scale=1]
			\filldraw [fill=orange!20,draw=black!80] 
			(1,0) --
			(1,1) --
			(1,2) -- 
			(1,3) -- 
			(2,5) -- % eigentlich (2,7) -- 
			(3,7) --% eigentlich (3,11) --
			(1,4) --
			(0,1) --
			cycle;
			
			\draw[black!80] (0,1) -- (1,1);
			\draw[black!80] (0,1) -- (1,2);
			\draw[black!80] (0,1) -- (1,3);
			\draw[black!80] (1,3) -- (1,4);
			\draw[black!80] (1,4) -- (2,5);
			
			\draw [->] (0,0) -- (3.35,0);
			\draw [->] (0,0) -- (0,7.35);

			\draw  (0,7) -- (-0.1,7);
			\node [left] at (-0.15, 7) {$ 11 e_2 $};	
			\draw[dotted] (0,7) -- (3,7) -- (3,0);			
			\draw  (3,0) -- (3,-0.1);
			\node [below] at (3,-0.15) {$ 3e_1 $};

			\draw  (0,5) -- (-0.1,5);
			\node [left] at (-0.15, 5) {$ 7 e_2 $};
			\draw[dotted] (0,5) -- (2,5) -- (2,0);
			\draw  (2,0) -- (2,-0.1);
			\node [below] at (2,-0.15) {$ 2e_1 $};

			\draw  (0,4) -- (-0.1,4);
			\node [left] at (-0.15, 4) {$ 4 e_2 $};
			\draw[dotted] (0,4) -- (1,4);

			\draw  (0,3) -- (-0.1,3);
			\node [left] at (-0.15, 3) {$ 3 e_2 $};
			\draw[dotted] (0,3) -- (1,3);
			
			\draw  (0,2) -- (-0.1,2);
			\node [left] at (-0.15, 2) {$ 2 e_2 $};
			\draw[dotted] (0,2) -- (1,2);
			
			\draw  (0,1) -- (-0.1,1);
			\node [left] at (-0.15, 1) {$ e_2 $};

			\draw  (1,0) -- (1,-0.1);
			\node [below] at (1,-0.15) {$ e_1 $};
			
			\draw  (2,0) -- (2,-0.1);
			\node [below] at (2,-0.15) {$ 2e_1 $};

			% dual graph
			\draw [very thick, blue] (1,1) --
			(1,2) -- 
			(1,3) -- 
			(2,5) -- % eigentlich (2,7) -- 
			(3,7) --% eigentlich (3,11) --
			(1,4) ;			
			
			\node at (1,1) {\color{blue} $ \bullet $};
			\node at (1.2,0.8) {\color{blue} \small $ - 2 $};
			
			\node at (1,2) {\color{blue} $ \bullet $};
			\node at (1.2,1.8) {\color{blue} \small $ - 2 $};
			
			\node at (1,3) {\color{blue} $ \bullet $};
			\node at (1.2,2.8) {\color{blue} \small $ - 3 $};
			
			\node at (2,5) {\color{blue} $ \bullet $};
			\node at (2.2,4.8) {\color{blue} \small $ - 2 $};
			
			\node at (3,7) {\color{blue} $ \bullet $};
			\node at (3.2,6.8) {\color{blue} \small $ - 1 $};
			\draw[->, thick, blue] (3,7) -- (3.4,7.4);
			
			\node at (1,4) {\color{blue} $ \bullet $};	
			\node at (0.75,4.2) {\color{blue} \small $ - 3 $};
		\end{tikzpicture}		
	\end{center}
	\caption{Schematic picture of the lotus $ \Lambda (\frac{11}{8}) $ 
		(on the left) and $ \Lambda (\frac{11}{3}) $ (on the right).
		The corresponding resolution graph is marked (in blue) as thick line with bullet points and with an additional arrow for the strict transform of $ V (x^{11} - y^8) $ resp.~$ V (x^{11} - y^3) $.}  
	\label{Fig:11_8} 
\end{figure}%   
\end{example}

\begin{Bem} \label{Rmk:Riemenschneider}
 For a continued fraction $\lambda=\frac{n}{q}=\lb b_1, \ldots, b_r \rb$ one can associate the cyclic quotient surface $\CC^2/G=X_{n,q}$, where $G$ is the group $\frac{1}{q}(1,n)$, see e.g. \cite{ReidQuotientSing} for notation. In the associated frieze $\mc{F}(\lambda)$ one can easily read off the numbers used to describe the equations (the sequences ${\bf i}$ and ${\bf j}$ in \cite[Eqn.~(2),(3)]{Riemenschneider74}) as well as the indices of the special representations (the sequence ${\bf t}$ defined in \cite[Section 2]{Riemenschneider98}): the sequences ${\bf i}$ and ${\bf j}$ comprise the diagonal containing the entry $n$ in the frieze (this is the $(0,r+1)$-entry by Cor.~\ref{Cor:n_q_in_frieze}) and the ${\bf t}$ sequence is another half diagonal ending in $n$. Further, deformations can be interpreted as deleting one of the triangles in the triangulation of the associated polygon, and using the results in Section \ref{Sec:friezes-graphs}, one sees how the quiddity, i.e., the continued fraction expansion, changes. 
It would be interesting to study higher dimensional cyclic quotient singularities using a variant of friezes.
\end{Bem}

\section{Constructing dual resolutions graphs and lotuses from CC-friezes} \label{Sec:friezes-graphs}

We have seen in the previous sections that every lotus $\Lambda(C)$ for a Newton non-degenerate curve $C \subseteq S$ corresponds to a triangulated $(m+2)$-gon, where $m$ is the number of irreducible exceptional curves in the minimal resolution of $C$.  Hence one can associate a frieze to it. In this section we show that for any CC-frieze $\F$ there exists a Newton non-degenerate curve $C$ such that its dual resolution graph (that is, the lateral boundary 
$ \partial_+ \Lambda(C) $ of the corresponding lotus) is given by the quiddity sequence of $\F$.

\begin{lemma} \label{Lem:quidditybase}
	Let $\F$ be a CC-frieze of width $w$ and suppose the first $w+1$ terms
	$a_1, \ldots, a_{w+1} $ in the quiddity sequence of $\F$ are known. Then the remaining two elements $ a_{w+2}, a_{w+3} $ in the quiddity sequence are uniquely determined. 
\end{lemma}

\begin{proof}
	Write the frieze with entries $m_{i,j}$, where the indexing follows the pattern as in \eqref{Eqn:frieze}: the indices $i,j$ are in $\ZZ$, with $i \leq j \leq w+i+3$  
	and the  boundary conditions are $m_{i,i}=0$, $m_{i,i+1}=1$, $m_{i,w+i+2}=1$, $m_{i,w+i+3}=0$. The quiddity sequence is $m_{i-1,i+1}=a_i$ for 
	$ i \in \{ 1, \ldots, w+3\}$. 
	Assuming that the $m_{i-1,i+1}$ for 
$ i \in \{ 1, \ldots, w+1\} $
	are known, one can calculate all the entries in the triangle below using the diamond rule, that is, for 
$ i \in \{ 0, \ldots , w \}$, 
	all $m_{i,j}$ for $i+2 \leq j \leq w+2$.  Thus we can calculate all entries in the diagonals $m_{0,j}$ and $m_{1,j}$, in particular, 
	the two entries $m_{0,w+1}$ and $m_{1, w+2}$. 
By \cite[(6.3), p.~306]{Coxeter}  
	we have  
	\begin{equation}
		\label{eq:Coxeter_6_3}
		m_{i,j} = m_{j, i + w + 3} = m_{i + w + 3, j + w + 3}
		\ .
	\end{equation}
	Applying the first equality of \eqref{eq:Coxeter_6_3} for $ i = k-w-2 $ and $ j = k-1 $, 
	we obtain
	\[
		m_{k - w - 2, k-1} = m_{k-1,k+1}
		\ . 
	\]
	For $ k = w + 2 $ (resp.~$ k = w+ 3 $), this provides
	$ m_{0,w+1} = m_{w+1,w+3} $ (resp.~$ m_{1,w+2} = m_{w+2,w+4} $). 
	Since $ a_k = m_{k-1,k+1} $ for all $ k $, the assertion of the lemma follows.
\end{proof}

\begin{Thm}
	\label{Thm:Frieze_Coord_Lotus}
	Let $ \cF $ be a CC-frieze of width $ w $ with entries $ m_{i,j} $ indexed as in \eqref{Eqn:frieze}. 
	Let $ P $ be the corresponding $ (w+3) $-gon with triangulation $\mc{T}$ (Thm.~\ref{Thm:CoCo}). 
	For every element $ m_{k-1,k+1} $ of the quiddity sequence of $ P $, 
	there exists a unique embedding of $ P  $ 
	as a Newton lotus of the form $  \Lambda  = \Lambda(\mc{E})$
	into the universal lotus $ \Lambda(e_1,e_2) $ of $ \ZZ^2 $ relative to the standard basis $ (e_1, e_2) $
	such that the quiddity sequence of the resulting triangulated polygon is 
		$ ( m_{k-1,k+1}, m_{k,k+2}, \ldots, m_{k+w+1,k+w+3} )$ 
		starting from the vertex $ (0,1) $: 
	\\
	Choose $ k \in \{ 1 , \ldots, w + 3 \} $  (corresponding to the element $ m_{k-1,k+1} $ in the quiddity sequence). 
	Then the vertices of the embedded polygon are determined by the two diagonals from top left to bottom right containing $ m_{k-1,k+1} $ and $ m_{k,k+2} $ respectively.
	More precisely, the vertices are
	\[
	(0,1), \ 
	(1,m_{k-1,k+1}), \ 
	(m_{k,k+2},m_{k-1,k+2}), \
	\ldots, 
	\
	(m_{k,k+\ell}, m_{k-1,k+\ell}), \
	\ldots, 
	\
	(m_{k,k+w+1},1),
	\
	(1, 0) 
	\ . 	
	\] 
\end{Thm}

\begin{proof}
	We deduce the embedding by constructing the mentioned lotus 
	$ \Lambda $.
	In the first step, we choose the vertex $ (0,1) $ to be the one corresponding to the quiddity entry $ m_{k-1,k+1} $.
	Hence, we have to have $ m_{k-1,k+1} $ many triangles incident to the vertex $ (0,1) $. 
	In order to achieve this, the triangles with vertices $ (0,1), (1,a), (1,a+1) $ have to be part of the lotus $ \Lambda $,
	for $ a \in \{ 0, \ldots, m_{k-1,k+1} - 1 \} $,
	and the triangle for $ a = m_{k-1,k+1} $ is not allowed to be contained in $ \Lambda $.  
	Hence, the vertex following clockwise to $ (0,1) $ has to be $ (1,m_{k-1,k+1}) $.  
	See also Fig.~\ref{Fig:Embedd1}. 
	\\
	\begin{figure}[h!]
		\begin{center}
			\begin{tikzpicture}[scale=1]
				
				\filldraw [fill=blue!20,draw=black!80, dotted] (0,1) -- (1,1) -- (1,2) -- cycle;
				\filldraw [fill=blue!20,draw=black!80, dotted] (0,1) -- (1,2) -- (1,3) -- cycle;
				\filldraw [fill=orange!20,draw=black!80, dotted] (1,4) -- (2,4) -- (2.5,5) -- cycle;

				\filldraw [fill=blue!40,draw=black!80] (1,0) -- (0,1) -- (1,1) -- cycle;
				\filldraw [fill=blue!40,draw=black!80] (0,1) -- (1,3) -- (1,4) -- cycle;
				\filldraw [fill=orange!40,draw=black!80] (1,3) -- (1,4) -- (2,4) -- cycle;
				\filldraw [fill=orange!40,draw=black!80]  (1,4) -- (2.5,5) -- (3,6) -- cycle;

				%\draw [dotted] (0,0) grid (3.2,1.2);
				\draw [->] (0,0) -- (3.35,0);
				\draw [->] (0,0) -- (0,6.35);

				\node [below] at (1,0) {$ (1,0) $}; 
				\node at (1,0) {\scriptsize $ \bullet $}; 
				
				\node [left] at (0,1) {$ (0,1) $};  
				\node at (0,1) {\scriptsize $ \bullet $}; 
				
				\node [right] at (1,1) {$ (1,1) $}; 
				\node at (1,1) {\scriptsize $ \bullet $};

				\node [right] at (1,2.85) {$ (1,m_{k-1,k+1}-1) $};
				\node at (1,3) {\scriptsize  $ \bullet $}; 
				
				\draw [->] (0.6,5.3) to [out=-90,in=120] (0.94,4.07);
				\node[above] at (1.2,5.2) {$ (1,m_{k-1,k+1}) $}; 
				\node at (1,4) {\scriptsize $ \bullet $};

				\node [right] at (3,6) {$ (1,m_{k-1,k+1}-1) +  ( m_{k,k+2} - 1 ) (1, m_{k-1,k+1} )$}; 
				\node at (3,6) {\scriptsize $ \bullet $};

			\end{tikzpicture}		
		\end{center}
		\caption[Symbolic picture of the beginning of the embedding procedure of the triangulated polygon.]{
			Symbolic picture of the beginning of the embedding procedure of the triangulated polygon.
			(The bullet points are not marked points, but indicate for which vertex the coordinates next to are provided.) }  
		\label{Fig:Embedd1} 
	\end{figure}%     
	Since $ m_{k,k+2} $ is the entry following $ m_{k-1,k+1} $ in the quiddity sequence, the number of triangles incident to the vertex $ (1,m_{k-1,k+1}) $ has to be equal to $ m_{k,k+2} $.  
	There is already the triangle with vertices $ (0,1), (1,m_{k-1,k+1}), (1,m_{k-1,k+1}-1) $. 
	The remaining triangles are those with vertices 
	\[ 
	(1,m_{k-1,k+1}), \
	(1,m_{k-1,k+1}-1) + (b-1) (1,m_{k-1,k+1}), \
	(1,m_{k-1,k+1}-1) + b (1,m_{k-1,k+1})  \ ,
	\] 
	for $ b \in \{ 1, \ldots, m_{k, k+2}-1 \} $ (if $ m_{k,k+2} >1 $).
	This is also visualized in Fig.~\ref{Fig:Embedd1}.
	Therefore, the next vertex of $ \Lambda $ following the vertex $ (1,m_{k-1,k+1}) $ clockwise is
	\[ 
	(1,m_{k-1,k+1}-1 ) +  ( m_{k,k+2} - 1 ) ( 1,m_{k-1,k+1}  )
	= 	
	( m_{k,k+2} ,\ m_{k-1,k+2} ) 
	\ , 	
	\]
	where we use 
	that
	$ m_{k,k+2} m_{k-1,k+1} - 1 =  m_{k-1,k+2} $ by the frieze rule 
	\eqref{eq:friezerule}.  
	
	We iterate the last step of the construction:
	Suppose that we have already constructed the vertices 
	$ v_0 := (m_{k,k+\ell},m_{k-1,k+\ell}) $ and $ v_1 := ( m_{k,k+\ell+1} , m_{k-1,k+\ell+1} ) $, 
	for some $ \ell $,
	and we want to show that the next vertex $ v_2 $ following clockwise is $ (m_{k,k+\ell+2}, m_{k-1,k+\ell+2}) $. 
	(Note that $ m_{k,k+1} = m_{k-1,k} = 1 $ and $ m_{k,k} = 0 $.)
	There have to be $ \mu := m_{k+\ell,k+\ell+2} $ many triangles incident to the vertex $ v_1 $.
	By the same argument as for $ \ell = 2 $ above, we get 
	$ v_2 = v_1  - v_0 + (\mu - 1 ) v_1  = \mu v_1 - v_0$,
	see also Fig~\ref{Fig:Embedd2}. 
	\begin{figure}[h!]
		\begin{center}
			\begin{tikzpicture}[scale=0.8]
				
				\filldraw [fill=blue!20,draw=black!80, dotted] (0,1) -- (1,1) -- (1,2) -- cycle;
				\filldraw [fill=blue!20,draw=black!80, dotted] (1,1) -- (1,2) -- (2,2) -- cycle;
				\filldraw [fill=blue!20,draw=black!80, dotted] (1,2) -- (2,2) -- (1,3) -- cycle;

				\filldraw [fill=orange!20,draw=black!80, dotted] (1,3) -- (2,3) -- (2,4) -- cycle;
				\filldraw [fill=orange!20,draw=black!80, dotted] (1,3) -- (2,4) -- (2,5) -- cycle;
				\filldraw [fill=blue!20,draw=black!80, dotted] (2,6) -- (3,6) -- (3.5,7) -- cycle;

				\filldraw [fill=blue!40,draw=black!80] (1,0) -- (1,1) -- (0,1) -- cycle;
				\filldraw [fill=orange!40,draw=black!80] (1,3) -- (2,2) -- (2,3) -- cycle;
				\filldraw [fill=orange!40,draw=black!80] (1,3) -- (2,5) -- (2,6) -- cycle;
				\filldraw [fill=blue!40,draw=black!80] (2,5) -- (2,6) -- (3,6) -- cycle;
				\filldraw [fill=blue!40,draw=black!80]  (2,6) -- (3.5,7) -- (4,8) -- cycle;
				
				\draw [->] (0,0) -- (4.35,0);
				\draw [->] (0,0) -- (0,8.35);

				\node [left] at (1,3) {$  v_0 $}; 
				\node at (1,3) {\scriptsize $ \bullet $};

				\node [right] at (2, 4.85) {$ v_1 - v_0 $};
				\node at (2,5) {\scriptsize  $ \bullet $}; 
				
				\node [left] at (2,6) {$ v_1 $}; 
				\node at (2,6) {\scriptsize $ \bullet $};

				\node [right] at (4,8) {$ v_1 - v_0  +  ( \mu  - 1 ) v_1 $}; 
				\node at (4,8) {\scriptsize $ \bullet $};

			\end{tikzpicture}		
		\end{center}
		\caption{Symbolic picture of the iteration of the embedding procedure of the triangulated polygon.}  
		\label{Fig:Embedd2} 
	\end{figure}%     
	
	Set $ a_k := m_{k-1,k+1} $ for $ k \in \ZZ $.  
	The second coordinate of $ v_2 = (v_{2,1}, v_{2,2}) $ is 
	\[
	\begin{array}{c}
		v_{2,2} 
		= 
		m_{k+\ell,k+\ell+2} m_{k-1,k+\ell+1} 	- m_{k-1,k+\ell}  
		=
		\\[5pt]
		\stackrel{\eqref{eq:frieze_entry_cont}}{=}
		a_{k+\ell+1} P_{\ell+1}(a_k, \ldots, a_{k+\ell}) - P_{\ell} (a_k, \ldots, a_{k+\ell-1})  
		\stackrel{(*)}{=} 
		P_{\ell+2}(a_k, \ldots, a_{k+\ell+1})
		=
		\\[5pt]
		\stackrel{\eqref{eq:frieze_entry_cont}}{=}
		m_{k-1,k+\ell+2}
		\ , 
	\end{array} 
	\]  
	where $ (*) $ uses the recursion $ P_n(y_1, \ldots, y_n) = y_n P_{n-1}(y_1, \ldots, y_{n-1}) - P_{n-2} (y_1, \ldots, y_{n-2} ) $, \cite[Number 547]{Muir}.
	The analogous computation leads to 
	$ v_{2,1} = m_{k,k+\ell+2} $.
	The claim follows.
\end{proof}

\begin{Bem} \label{Bem:cerfvolant}
Popescu-Pampu described how to embed a so-called membrane \cite[D\'ef.~4.3]{PPP-cerfvolant} into the universal lotus
	in \cite[Paragraph after Remarque~5.3 and Example~5.4]{PPP-cerfvolant}.
	The notion of a membrane is similar to a triangulation of a polygon, but involves two different type of triangles corresponding to petals and half-petals in the universal lotus. 
	While our embedding result (Thm.~\ref{Thm:Frieze_Coord_Lotus}) provides the precise coordinates of the vertices of the embedded lotus, 
	Popescu-Pampu gives a procedure how to embed a given membrane. 	
\end{Bem}

\begin{example}
	In Fig.~\ref{Fig:11_8}, we have seen one embedding of the triangulated polygon (Fig.~\ref{Fig:running})
	corresponding to the frieze arising from the continued fraction $\lambda=\frac{11}{8}=\lb 2,2,3,2 \rb$.
	For the reader's convenience, we recall the frieze $\F(\lambda)$ in Fig.~\ref{Fig:FriezeRunningRecalled} 
	-- this time including the zero rows at the top and bottom.
	\begin{figure}[h!]
		\small
		\begin{tikzcd}[row sep=0.1em, column sep=-0.35em]
			\col{0} && 0 && 0 && \fbox  0 && 0 && 0 && 0 && 0  && 0 && \col{0}
			\\
			&%\col{1}&&\col{1} &&
			\col{1} && 1 && \fbox 1   && \fcolorbox{black}{lightgray}{1}   && 1  && 1   && 1   && 1 &&1 && \col{1} %&& \col{1} %&& \col{1} && \col{1}
			\\
			&&%\col{1}&& \col{3}&& 
			\col{4} && 1     &&   \fcolorbox{black}{lightgray}{2}     && \fbox 2      && 3     && 2    && 1    && 3 && 4 && \col{1} %&& \col{2} && \col{2}  && \col{3}
			\\
			&&&% \col{2}&&\col{11}&&
			\col{3}&& \ 1     &&  \fbox 3     && \fbox 5      && 5     && 1    && 2     && 11 &&3 &&\col{1} %&&\col{3} &&\col{5}  &&\col{5} 
			\\
			&&&&%\col{7} &&\col{8}&&
			\col{2}&& 1     && \fbox 7     && \fbox 8     && 2      && 1     && 7     && 8 &&2 && \col{1} %&& \col{7} %&& \col{8} && \col{4}
			\\
			&&&&& %\col{5} &&\col{5}&&
			\col{1}&& 2     && \fbox{11}   && \fbox 3    && 1      && 3   && 5     && 5 &&1 && \col{2}%&& \col{11} && \col{3} && \col{1}
			\\
			&&&&&&% \col{3} &&\col{2}&&
			\col{1}&& 3     &&  \fbox 4   &&  \fbox1    && 2      && 2   && 3     &&2 &&1 && \col{3} %&& \col{4}% && \col{1} && \col{2}
			\\
			&&&&&&&% \col{1} &&\col{1}&&
			\col{1} && 1   && \fbox 1  && \fbox 1  && 1   && 1   && 1  && 1  && 1 &&\col{1} %&& \col{1} && \col{1} && \col{1}
			\\
			&&&&&&&& \col{0} && 0 && \fbox 0 && 0 && 0 && 0 && 0 && 0  && 0 && \col{0}  
		\end{tikzcd}
		\caption{The frieze obtained from $\lambda=\frac{11}{8}=\lb2,2,3,2\rb$,
		where the entries of the two diagonals determining the coordinates of $ \Lambda (\frac{11}{8}) $ are marked with boxes and the vertex $ (1, m_{k-1,k+1}) $ is additionally colored.}
		\label{Fig:FriezeRunningRecalled}
	\end{figure}
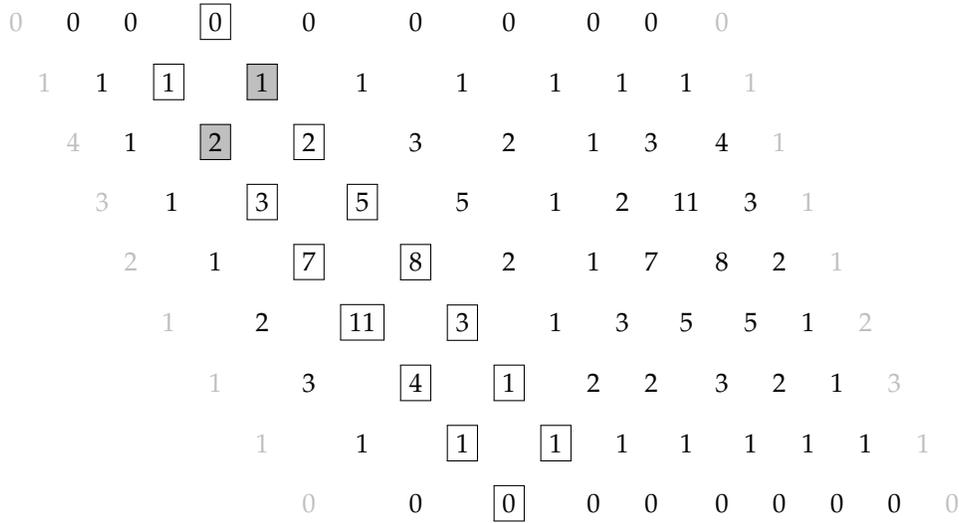
	By Thm.~\ref{Thm:Frieze_Coord_Lotus}, we can read off the coordinates of the vertices of the mentioned lotus (cf.~Fig.~\ref{Fig:11_8}):
	\[
	(0,1), \
	(1,2), \
	(2,3), \ 
	(5,7), \ 
	(8,11), \ 
	(3,4), \
	(1,1), \
	(1,0) \ .
	\]
	On the other hand, we could also choose the embedding of the polygon obtained by choosing $ m_{2,4} = 3 $ instead of $ m_{0,2} = 2 $ for the vertex $ (0,1) $. 
	Then Thm~\ref{Thm:Frieze_Coord_Lotus} provides that the coordinates of the corresponding lotus are
	\[
		(0,1), \
		(1,3), \ 
		(2,5), \
		(1,2), \
		(1,1), \ 
		(3,2), \
		(2,1), \
		(1,0) \ .
	\]
	In Fig.~\ref{Fig:11_8_anders} we depict the lotus. 
	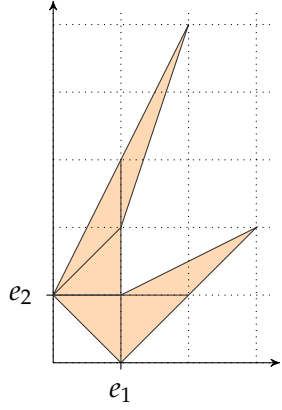
\begin{figure}[h!]
		\begin{center}
			\begin{tikzpicture}[scale=0.9]
				\filldraw [fill=orange!30,draw=black!80] 
				(0,1) --
				(1,3) --
				(2,5) -- 
				(1,2) -- 
				(1,1) -- 
				(3,2) --
				(2,1) --
				(1,0) --
				cycle;

				\draw [dotted] (0,0) grid (3.2,5.2);
				\draw [->] (0,0) -- (3.35,0);
				\draw [->] (0,0) -- (0,5.35);
				
				\draw[draw=black!80] (0,1)--(1,1) -- (2,1);
				\draw[draw=black!80] (1,0)--(1,1);
				\draw[draw=black!80] (1,1)--(0,1) -- (1,2) -- (1,3);

				\draw  (0,1) -- (-0.1,1);
				\node [left] at (-0.15, 1) {$ e_2 $};

				\draw  (1,0) -- (1,-0.1);
				\node [below] at (1,-0.15) {$ e_1 $};
			\end{tikzpicture}		
		\end{center}
		\caption{Alternative embedding of the triangulated polygon of Fig.~\ref{Fig:running} into the universal lotus $ \Lambda(e_1, e_2) $.}  
		\label{Fig:11_8_anders} 
	\end{figure}%     	
\end{example}

\begin{corollary}
	\label{Cor:Frieze_gibt_Kurve}
Given a frieze $ \cF $ with quiddity sequence $\{ a_i\}_{i=1}^m$, there exists a plane curve $ C = V (f) $ such that its dual resolution graph $ \Gamma(f) $ is of type $A_{m-2}$ and corresponds to the quiddity sequence of $ \cF $, i.e., $\Gamma(f)$ looks as in Fig.~\ref{fig:GammaofFrieze}: 
    \begin{figure}[h!]
    \begin{center}
\begin{tikzpicture}[scale=0.8]
	
\draw [-, color=black, thick](0,0) -- (2.4,0);
\draw [-, color=black, dotted](2.4,0) -- (3.6,0);
\draw [-, color=black, thick](3.6,0) -- (8.4,0);
\draw [-, color=black, dotted](8.4,0) -- (9.6,0);
\draw [-, color=black, thick](9.6,0) -- (12,0);

\node [below] at (0,0) {$-a_2$};
\node [below] at (2,0) {$-a_3$};
\node [below] at (4,0) {$-a_{k-1}$};
\node [below] at (6,0) {$-a_{k}$};
\node [below] at (8,0) {$-a_{k+1}$};
\node [below] at (10,0) {$-a_{m-2}$};
\node [below] at (12,0) {$-a_{m-1}$};

\node[draw,circle, inner sep=1.pt,fill=black] at (0,0){};
\node[draw,circle, inner sep=1.pt,fill=black] at (2,0){};
\node[draw,circle, inner sep=1.pt,fill=black] at (4,0){};
\node[draw,circle, inner sep=1.pt,fill=black] at (6,0){};
\node[draw,circle, inner sep=1.pt,fill=black] at (8,0){};
\node[draw,circle, inner sep=1.pt,fill=black] at (10,0){};
\node[draw,circle, inner sep=1.pt,fill=black] at (12,0){};

\end{tikzpicture}
\end{center}
 \caption{The dual resolution graph of $C=V(f)$.}
\label{fig:GammaofFrieze}
\end{figure}
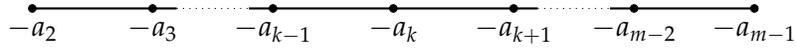
\end{corollary}

\begin{proof}
We can associate a triangulated polygon $P(\mc{F})$ with $m$ vertices $v_1, \ldots, v_m$  to the frieze $\mc{F}$ such that $a_i$ is the number of triangles incident to $v_i$ by Thm.~\ref{Thm:CoCo}. 
Then Thm.~\ref{Thm:Frieze_Coord_Lotus} embeds $P$ uniquely as a Newton lotus $\Lambda(\mc{E})$ for some nonempty finite set $\mc{E} \in \QQ_+$ into the universal lotus $\Lambda(e_1, e_2)$ such that the starting vertex $(0,1)$ corresponds to $v_1$ (i.e., the the base is $e_1=v_m$ and $e_2=v_1$). Now the curve $C=V(f)$ can be reconstructed from $\mc{E}=\{\lambda_1, \ldots, \lambda_\rho \}$. Explicitly, if $\lambda_i=\frac{d_i}{c_i}$ with $\gcd(d_i, c_i)=1$ for 
$ i \in \{ 1, \ldots, \rho \}$, then we may take
\[ 
	f = \prod_{i = 1}^\rho ( x^{d_i} - y^{c_i} ) \ . 
\]
\end{proof}

\begin{Bem}
	Since the lotus $ \Lambda (\mc{E}) $ constructed in the proof of Cor.~\ref{Cor:Frieze_gibt_Kurve} is unmarked, the finite set $ \mc{E} $, and thus the curve $ C $, is not unique,
	cf.~Example~\ref{Ex:marked_unmarked}.
	Nonetheless, there is a unique $ \mc{E}=\{\lambda_1, \ldots, \lambda_\rho \}$ if we impose $ \rho $ to be minimal.
	The number of irreducible components of the corresponding curve is then minimal.
\end{Bem}

As a consequence of Thm.~\ref{Thm:Lotus-dual} and Thm.~\ref{Thm:Frieze_Coord_Lotus} we further obtain the following result.

\begin{corollary} \label{Cor:graph-boundary}
	Let $\Gamma(f)$ be the dual resolution graph of a Newton non-degenerate plane curve $C=V(f) \subseteq S$ of type $A_{m-2}$ such that the vertex $i$ has weight $-a_{i+1}$, 
	$ i \in \{ 1, \ldots, m-2 \} $,
	for some $a_i \in \ZZ_{>0}$. Then $\Gamma(f)$ corresponds to the lateral boundary $\D_+ \Lambda(\mc{E})$ for some Newton lotus $\Lambda(\mc{E})$. Further, one can associate a frieze $\mc{F}(\Lambda(\mc{E}))$ of width $m-3$ and quiddity sequence $\{a_i\}_{i=1}^m$ with $a_1$ and $a_m$ determined by Lemma \ref{Lem:quidditybase}.
\end{corollary}

The previous results allow us to enumerate dual resolution graphs of type $A_n$. Therefore recall that a graph $\Gamma=(V,E)$ is of type $A_n$ if $\Gamma$ is simply laced, has no loops, $V=\{1, \ldots, n\}$, and one can find a labeling of the vertices such that $E=\{(i,i+1)$ for all $ i \in \{ 1, \ldots, n-1 \} \}$. 
Further recall that two weighted graphs $\Gamma= ( V, E, W ) $ and $\Gamma'= ( V',E', W')$ are isomorphic if there exists an edge and weight preserving bijection between $V$ and $V'$.

For the following result we consider a dual resolution graph $\Gamma(C)$ of a plane curve $C \subseteq S$ without the arrowhead vertices, that is, we ignore the components of the strict transform of $C$.
\begin{corollary} \label{Cor:numberresgraphs}There exist
	$$\left \lceil  \frac{C_n  }{2} \right \rceil=\left \lceil  \frac{\frac{1}{n+1} {2n \choose n} }{2} \right \rceil $$
	pairwise different dual resolution graphs of plane curves $C \subseteq S$ of type $A_n$,
	where $ C_n $ is the $ n $-th Catalan number.
\end{corollary}

\begin{proof}
This is a consequence of Cor.~\ref{Cor:graph-boundary}, the fact that there are $ C_n $ pairwise different triangulations of an $ (n+2)$-gon
and since the only nontrivial isomorphism of a weighted graph of type $ A_n $ is the reflection about the middle, that is, the map sending vertex $i$ to vertex $n-i+1$.
\end{proof}

\section{Geometric interpretation of frieze entries}

We have seen that some entries of the quiddity sequence of a frieze $\mc{F}$ correspond to the negative of the self intersection numbers of the exceptional divisors in the minimal resolution of a curve associated to $\mc{F}$, see Thm.~\ref{Thm:Lotus-dual} and Cor.~\ref{Cor:graph-boundary}. Further, the coordinates of the lotus can be seen as neighboring diagonals in the frieze, cf.~Thm.~\ref{Thm:Frieze_Coord_Lotus}.   
In this section we will interpret some of the entries of the frieze $\mc{F}(\Lambda(f))$ in terms of partial resolutions of the curve $C=V(f)$. 

We first describe how the diagonals of the lotus $\Lambda(f)$ correspond to the entries in the frieze, using Pl\"ucker coordinates. We view the lotus $\Lambda(f)$ as a polygon according to Def.~\ref{def:polygon} and so its diagonals are defined. This allows one to associate a quiver to $\mc{F}(\Lambda(f))$ which yields a cluster category $\mc{C}$ of type $A$, as explained in Section \ref{Sub:categoryA}. 
In order to define reduction of friezes we look at the  Auslander--Reiten-quiver (=AR-quiver) of $\mc{C}$ and consider reduction of $\mc{C}$ with respect to a rigid module $M \in \mc{C}$.  The main result of this section (Thm.~\ref{Thm:Partial_reduction}) states that the reduction of $\mc{C}$ (resp. the corresponding frieze) will yield a partial resolution of the corresponding curve singularity, and hence a sublotus of $\Lambda(f)$. The entries in the reduction of $\mc{F}$ are already contained in the original frieze $\mc{F}$.

\subsection{Friezes and cluster categories of type $A$} \label{Sub:categoryA}

Consider a frieze of width $w=m-3$ corresponding to a triangulated $m$-gon $P$ with triangulation $\mc{T}=\{[i,j] \}$, where $[i,j]$ are $m-3$ noncrossing diagonals with $1 \leq i < j \leq m$ and $|i-j| \geq 2$. The condition $|i-j| \geq 2$  is to rule out the boundary edges. 
Then one can label a frieze as follows (cf.~with Coxeter's labeling \eqref{Eqn:frieze})
\begin{equation} \label{Eqn:plueckers}
	\adjustbox{scale=0.75,center}{
		\begin{tikzcd}[row sep=0.35em, column sep=-0.35em]
			\ldots &p_{11} && p_{22} && p_{33} && p_{44} && \ldots & & \ldots& & p_{m,m}&  & p_{11}  && \ldots\\
			\ldots   && p_{12} && p_{23} && p_{34} && \ldots && \ldots &&p_{m-1,m} & & p_{1,m} \\
			&\ldots && p_{13} && p_{24} && p_{35} && \ldots& &  p_{m-2,m} && p_{1,m-1} && \ldots \\
			&& && p_{14} && p_{25} && \ldots &&p_{m-3,m} && p_{1,m-2}  && \ldots\\
			%    &\ldots && (-2,0) && (-1,1) && (0,2) && (1,3) && p_{m}& & \ldots \\
			& \ldots && \ldots && \ldots && \ldots && \ldots && \ldots \\
			\ldots && \ldots && p_{m-2,m} && p_{1,m-1} && p_{2,m}  && p_{13} && \ldots \\
			&  && p_{m-2,m-1} && p_{m-1,m} && p_{1,m} && p_{12} && \ldots \\
			&& \ldots&& p_{m-1,m-1} && p_{m, m} && p_{11} && p_{22} && \ldots &
		\end{tikzcd}
	}
\end{equation}

\begin{Bem} Note here: the $p_{ij}$ are treated as variables, namely, as the Pl\"ucker coordinates on the homogeneous coordinate ring of the Grassmannian $Gr(2,m)$, that is, $\CC[p_{ij}: 1 \leq i < j \leq m]/I_P$, where $I_P$ is the ideal generated by the Pl\"ucker relations. This ring has a cluster structure by \cite[Thm.~3]{Scott06}, where the clusters are formed by maximal sets of compatible Pl\"ucker coordinates (note that loc.~cit.~ shows the cluster structure of the coordinate ring for general $Gr(k,m)$, $k \in \{2, \ldots, \lfloor \frac{m}{2} \rfloor\}$, the result for $Gr(2,m)$ was also previously shown in \cite[Prop.~12.6]{FZ2003II}). We use the usual conventions in the context of cluster algebras: $p_{ii}=0$ and 
	we will consider the indices modulo $m$. The generating relations of $I_P$ are then the Ptolemy relations on the $m$-gon $P$ and they also give us the frieze relations (see e.g.~\cite[1.3]{BaurIsfahan}). 
\end{Bem}

Setting $p_{ii}=0$ and $p_{i,i+1}=1$ and $p_{i,i+2}=a_{i+1}$ yields the CC-frieze with quiddity $\{a_{i}\}_{i=1}^m$. 
Then $p_{ij}$ corresponds to the diagonal $[i,j]$ in the polygon $P$. The same frieze can also be obtained by putting $p_{ij}=1$ if and only if $[i,j] \in \mc{T}$. One can see that these $1$s uniquely determine the corresponding CC-frieze, this is explained in detail in \cite[Section 3]{BFGST-survey}.

Further, from the description of the frieze \eqref{Eqn:plueckers} one immediately gets a \emph{fundamental domain} for the entries of the frieze, that is, a region, such that any entry is contained in it and the frieze is just a repetition of the fundamental domain: 

\begin{lemma}
	Let $\mc{F}$ be a CC-frieze of width $m-3$ corresponding to a triangulation $\mc{T}$ of an $m$-gon $P$. Then a fundamental domain is given by the Pl\"ucker coordinates $p_{ij}$, $1 \leq i < j \leq m$, and there are ${m \choose 2}$ entries in it. Note that the Pl\"ucker coordinates $p_{i,i+1}=1$ correspond to the boundary edges $[i,i+1]$ and the other values of the $p_{ij}$ are determined by the triangulation $\mc{T}$ of $P$.
\end{lemma}

\begin{proof}
	This follows from the description of the entries of a frieze above and  \eqref{Eqn:plueckers}. The fundamental domain can be pictured as follows:
	\begin{equation} \label{Eqn:fundamentaldomain}
		\adjustbox{scale=0.9,center}{
			\begin{tikzcd}[row sep=0.35em, column sep=-0.35em]
				%  &0 && 0 && 0 && 0 && \ldots & & \ldots& & &  &   &&\\
				&& p_{1m} && p_{12} && p_{23} && \ldots && \ldots &&p_{m-2,m-1} & &  \\
				& && p_{2m} && p_{13} && p_{24} && \ldots& &  p_{m-3,m-1} &&  &&  \\
				&& && p_{3m} && p_{14} && \ldots &&p_{m-3,m-1} &&   && \\
				%    &\ldots && (-2,0) && (-1,1) && (0,2) && (1,3) && p_{m}& & \ldots \\
				&  &&  && \ldots && \ldots && \ldots && \\
				&&  &&  && p_{m-2,m} && p_{1,m-1}  &&  &&  \\
				&  &&  &&  && p_{m-1,m} &&  &&\\
				% &&&&  && 0 && 0 &&  &&  &
			\end{tikzcd}
		}
	\end{equation}
\end{proof}

	\begin{Bem}
		With this notation, the quiddity sequence is $a_1=p_{2,m}$, $a_2=p_{13}$, \ldots, $a_{m-1}=p_{m-2,m}$, $a_{m}=p_{1,m -1}$. For a lotus $\Lambda$ consider the associated polygon $P$ with vertices $v_1, \ldots, v_m$ so that vertex $1$ is at the point $(0,1)$, and label clockwise, that is, the base petal are the two vertices $v_1$ and $v_m$ and their quiddities are $a_1$ and $a_m$. Further note that the weights of the dual resolution graph are then $-a_2, \ldots, -a_{m-1}$.
	\end{Bem}

From a triangulation $\mc{T}$ of an $m$-gon one can also construct a \emph{quiver} $Q_\mc{T}$, which in turn will allow us to use representation theory, in particular cluster categories. Since we do not need the quiver explicitly here, we refer to \cite[Sections 3.1.3, 3.4.1]{Schifflerbook} for the construction of  $Q=Q_{\mc{T}}$. Then the cluster category $\mc{C}_Q$ is defined as $\mc{C}_Q:=D^b(\mmod kQ)/\tau^{-1}[1]$, where $\tau$ is the Auslander--Reiten-translation functor and $[1]$ is the shift functor on the triangulated category $D^b(\mmod kQ)$, see \cite{BMRRT06,CCS06}. Strictly speaking, we have to consider the generalized cluster category $\mc{C}_{(Q,W)}$ of a quiver $Q$ with potential $W$, introduced by Amiot \cite{Amiot}, see \cite[Section 2]{BFGST18} for details. 
For ease of notation, we will write $\mc{C}_Q$. \\
Consider the AR-quiver of $\mc{C}_Q$: this is a stable translation quiver and we can define the Hom- and Ext-hammocks, following the notation in \cite[Section 3.1.4]{Schifflerbook}. Let $M \in \mc{C}$ be indecomposable (this means that $M$ corresponds to a vertex in the AR-quiver). Denote by $\ms{R}_{\rightarrow}(M)$ be the maximal slanted rectangle in the AR-quiver of $\mc{C}$ with leftmost point $M$. Then $\ms{R}_{\rightarrow}(M)$ is the \emph{forward Hom--hammock} of $M$, that is, an indecomposable module $N \in \ms{R}_{\rightarrow}(M)$ if and only if $\Hom(M,N) \neq 0$. In particular, $\dim(\Hom(M,N))=1$ for any such $N$. Similarly, define the \emph{backward Hom--hammock} $\ms{R}_{\leftarrow}(M)$ as the maximal slanted rectangle in the AR-quiver of $\mc{C}$ with rightmost point $M$. This is the set of $N \in \mc{C}$ such that $\Hom(N,M) \neq 0$, in particular $\dim (\Hom(N,M))=1$ for any such $N$. See Fig.~\ref{Pic:Hom-hammocks} for an illustration of the two hammocks. \\
\begin{figure}[h!] 
	\begin{center}
		\adjustbox{scale=0.16,center}{
			\includegraphics{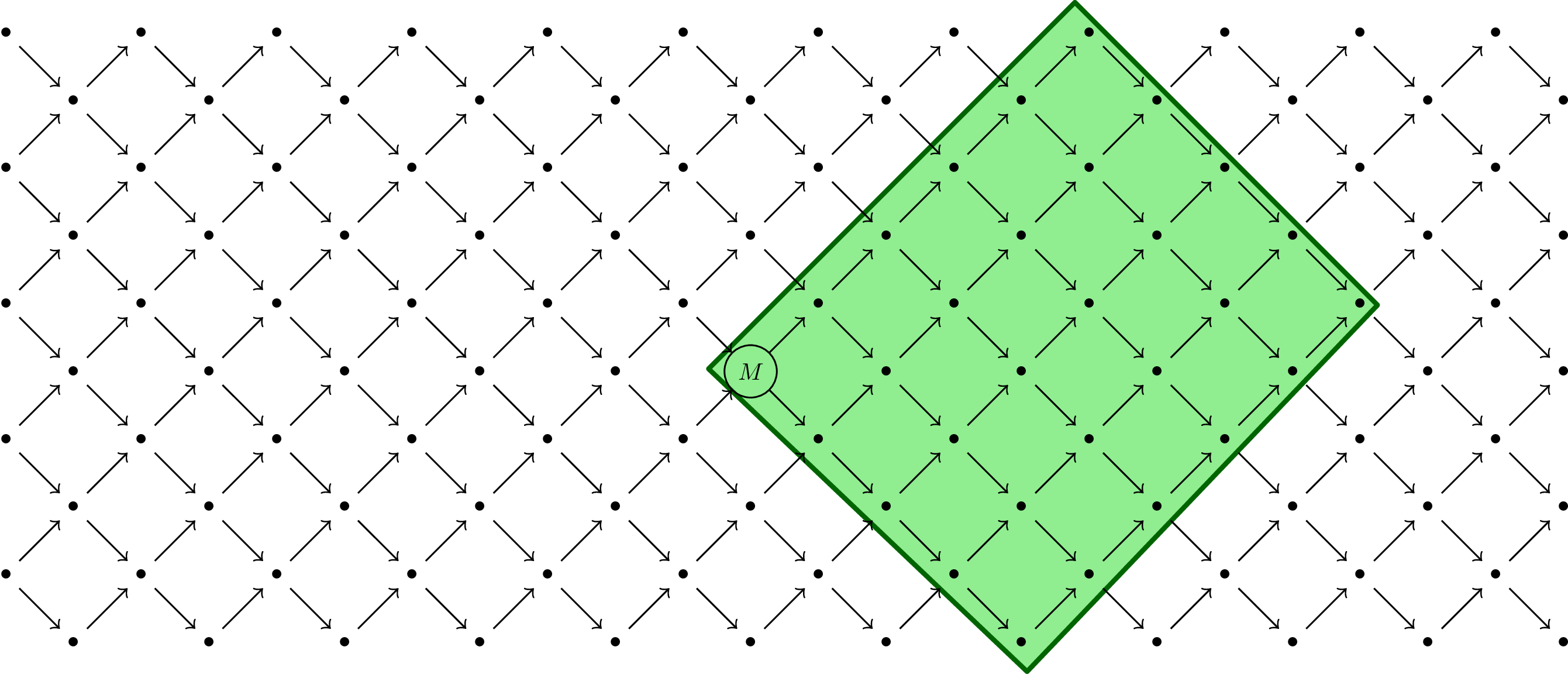} 
			\hspace{4cm}
			\includegraphics{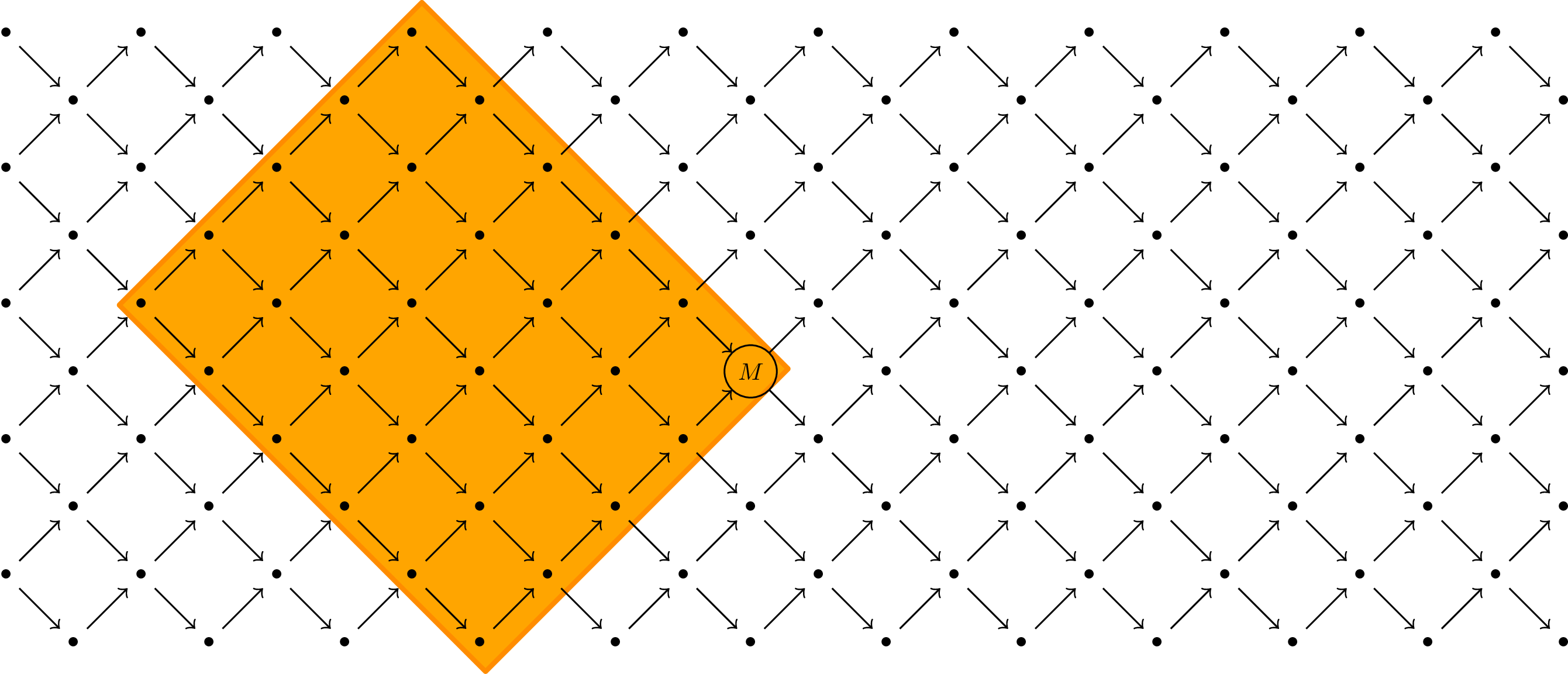} 
		}
	\end{center}
	\caption{Forward Hom-hammock $\ms{R}_{\rightarrow}(M)$ (left) and backward Hom-hammock $\ms{R}_{\leftarrow}(M)$ (right)   of the module $M$ in the AR-quiver. \label{Pic:Hom-hammocks}}
\end{figure}

Furthermore, the \emph{backward Ext--hammock} of $M$ is $\ms{R}_{\leftarrow}(\tau M)$, the set of $N \in \mc{C}$ such that $\Ext^1(M,N)=\Hom(M,N[1]) \neq 0$. Similarly, the set of $N \in \mc{C}$ such that $\Ext^1(N,M) \neq 0$ is described by the \emph{forward Ext--hammock} of $M$, denoted by $\ms{R}_{\rightarrow}(\tau^{-1}M)$, see Fig.~\ref{Pic:Ext-hammocks}. For proofs of these results see e.g. \cite[Sections 2.3, 2.4]{BFGST18}.

\begin{figure}[h!] 
	\begin{center}
		\adjustbox{scale=0.8,center}{
                         \includegraphics{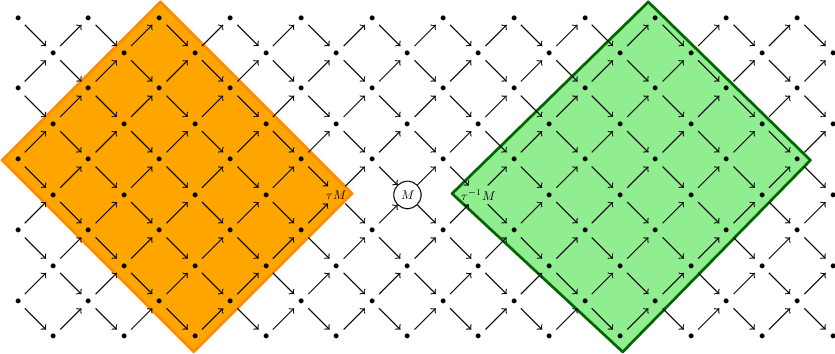}
		}
	\end{center}
	\caption{Forward Ext-hammock $\ms{R}_{\rightarrow}(\tau^{-1}M)$ (right) and backward Ext-hammock $\ms{R}_{\leftarrow}(\tau M)$ (left) of the module $M$ in the AR-quiver. \label{Pic:Ext-hammocks}
	}
\end{figure}

Using the labeling of \eqref{Eqn:plueckers}, the indecomposables in the AR-quiver will be denoted by $M_{ij}$ (corresponding to the entry $p_{ij}$). In the following, we will use the notation $M_{ij}$ and $p_{ij}$ interchangeably. Note that $\Ext^1(M_{ij},M_{kl}) \neq 0$ for indecomposable $M_{ij},M_{kl} \in \mc{C}$ if and only if the corresponding diagonals $[i,j]$, $[k,l]$ in the $m$-gon $P$ cross.  

To sum up the correspondence between triangulated polygons and cluster categories of type $A$: from a polygon $P$ with $m$ vertices and with triangulation $\mc{T}$ one obtains a quiver $Q=Q_\mc{T}$, which is mutation equivalent to a Dynkin quiver of type $A_{m-3}$. From the quiver one then constructs the cluster category $\mc{C}_Q$. The indecomposable objects in $\mc{C}_Q$ are in bijection with the diagonals in $P$ and the diagonals in $\mc{T}$ form a \emph{cluster tilting object} in $\mc{C}_Q$. 

For the corresponding frieze $\mc{F}(\mc{T})$ this means that also $\mc{C}_Q$ defines the same frieze (so $\mc{F}(\mc{T})=\mc{F}(Q)$). We have already discussed that the entries ``1'' of $\mc{F}(\mc{T})$ are precisely the ones corresponding to the diagonals in $\mc{T}$. One can get this frieze $\mc{F}(Q)$ also by looking at the AR-quiver of $\mc{C}_Q$ and evaluating the cluster character to $1$ at the indecomposable direct summands of the cluster tilting object. For details we refer to \cite{BFGST18}.

\subsection{Reduction of cluster categories and friezes} \label{Sub:reduction}

Reduction of a cluster category was first studied by Iyama and Yoshino \cite{IY08}. It was used for reduction of friezes in \cite{BFGST2} and more generally for reduction of Frobenius extriangulated categories in \cite{FMP}. Here we only briefly describe the operation on a frieze and the associated polygon, for proofs of the statements see loc.~cit. 

Let $\mc{C}=\mc{C}_Q$ be the cluster category of a Dynkin quiver $Q$. For any rigid indecomposable object $M \in \mc{C}$, that is, any $M$ such that $\Ext^1_{\mc{C}}(M,M)=0$, define 
$$M^\perp:=\{ X \in \mc{C}: \Hom_\mc{C}(M,X[1])=0\}=\{ X \in \mc{C}: \Hom_\mc{C}(X, M[1])=0 \} \ , $$
and the factor category
$$\mc{C}(M):=M^\perp/ \langle \add M \rangle \ . $$
We call $\mc{C}(M)$ the \emph{reduction (of \mc{C}) with respect to $M$}.
By \cite[Thm.~4.2 and 4.7]{IY08} the category $\mc{C}(M)$ is also triangulated and $2$-Calabi--Yau and there is a cluster character. Now going back to the corresponding frieze $\mc{F}$ of $\mc{C}$, in \cite[Prop.~5.3]{BFGST2} it was shown that if one reduces along a rigid module $M$ for which the corresponding frieze entry is equal to $1$, then  one obtains a so-called mesh frieze for $\mc{C}(M)$, the \emph{reduction of $\mc{F}$ with respect to $M$}.
This is explained in detail in \cite[Section 8]{FMP} (in the context of Frobenius extriangulated categories, which also applies here).

Concretely, for a frieze $\mc{F}$ coming from a triangulation $\mc{T}$ of an $m$-gon $P$, this means that cutting $P$ along a diagonal $[i,j]$ of the triangulation $\mc{T}$ (corresponding to the indecomposable module $M_{ij}$ in the AR-quiver of $\mc{C}$), one obtains two smaller polygons $P'$ and $P''$ and hence two smaller friezes $\mc{F}'$ and $\mc{F}''$:

\begin{lemma}
	Let $\mc{F}$ be a frieze coming from a triangulation $\mc{T}$ of an $m$-gon $P$ and assume that $[i,j] \in \mc{T}$ for some $ i, j \in \{ 1, \ldots, m \} $ with $ i < j $ and $|i-j|>1$. Denote by $\mc{C}$ the corresponding cluster category, and by $M_{ij}$ the module corresponding to $[i,j]$.   
	Then one obtains by reduction with respect to $M_{ij}$ the factor category $\mc{C}(M_{ij})$ and the reduction of the frieze with respect to $M_{ij}$. This mesh frieze yields  two friezes $\mc{F}'$ and $\mc{F}''$ from $\mc{F}$, as pictured in Fig.~\ref{Fig:reductionFrieze}.
\end{lemma}

\begin{proof}
	The categorical proof of this fact is in \cite[Thm.~8.7 and Example 8.8]{FMP}. Note that loc.~cit.~ is in terms of Grassmannian cluster categories, which are a Frobenius version of the categories $\mc{C}$ we consider here. The Pl\"ucker frieze of \cite[Example 8.8]{FMP} is our frieze $\mc{F}$ with entries $p_{ij}$ as in \eqref{Eqn:plueckers}. See below,  Fig.~\ref{Fig:reductionFrieze}, for a schematic picture of the reduction at the entry $p_{ij}$ and the two smaller friezes $\mc{F}'$ and $\mc{F}''$. In particular, note that the entry $p_{ij}$ becomes part of the boundary of both of these friezes (the boundary rows consist of entries $1$ only). Since the diagonal $[i,j] \in \mc{T}$, we have $p_{ij}=1$ and indeed all the entries in $\mc{F}'$ and $\mc{F}''$ satisfy the frieze relations. 
	
	\begin{figure}[h!] 
		\begin{center}
			\adjustbox{scale=0.3,center}{
				\includegraphics{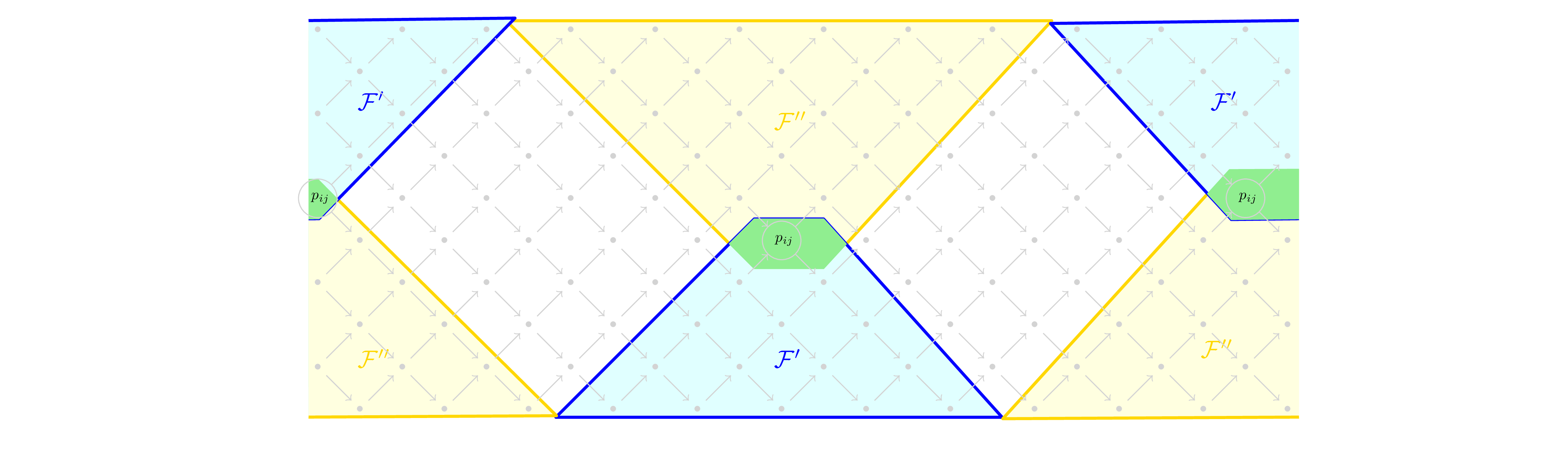} 
			}
		\end{center}
		\caption{Picture with the friezes $\mc{F}'$ and $\mc{F}''$. \label{Fig:reductionFrieze}}
	\end{figure}
\end{proof}

In our context, on the singularity side, we have fixed a cross $(L,L')$ (i.e., the boundary diagonal $[1,m]$ of a lotus $\Lambda(f)$). Considering $\Lambda(f)$ as a triangulated polygon $P$, this boundary diagonal will be contained in either $P'$ or $P''$. 
We then define $P({[i,j]})$ the \emph{reduction of $P$ with respect to $[i,j]$} as 
$$ P({[i,j]}) := \begin{cases} P' &  \text{ if } [1,m] \in P' \ , \\
	P'' & \text{ if } [1,m] \in P'' \ .
\end{cases} $$
Then we say that the frieze from $P({[i,j]})$, denoted by $\mc{F}({[i,j]})$, is the  \emph{reduction of $\mc{F}$ with respect to $[i,j]$}. We can explicitly describe the quiddity of $\mc{F}({[i,j]})$: in Lemma \ref{Lem:Friezered}  we start with the classical case of cutting off an ear that was already considered by Conway and Coxeter (and is part of their induction argument for proving Thm.~\ref{Thm:CoCo}, cf.~Remark~\ref{Rmk:cuttingear}) and then deal with the general case in Prop.~\ref{Prop:generalreduction}.

\begin{lemma} \label{Lem:Friezered}
	Let $\mc{F}$ be a frieze coming from a triangulation $\mc{T}$ of an $m$-gon $P$ with quiddity sequence $\{a_i\}_{i=1}^{m}$. Assume that the diagonal $[i-1,i+1] \in \mc{T}$, where $i-1,i+1$ are the representatives modulo $m$ in $\{1, \ldots, m\}$. Then the reduction $\mc{F}({[i-1,i+1]})$ is 
	\begin{enumerate}[(1)]
		\item the trivial frieze if $i=1$ or $i=m$ (note: in those cases, the considered diagonals are $[2, m]$, $[1,m-1]$, respectively),
		\item the frieze with quiddity
		$$ (a_1,  \ldots, a_{i-2}, p_{i-2,i+1}, p_{i-1, i+2}, a_{i+2}, \ldots, a_{m}) $$
		if $ i \in \{ 2, \ldots, m-1 \}$. 
	\end{enumerate}
\end{lemma}

\begin{proof}
	Note that reducing along a diagonal $[i-1,i+1]$ means that we  divide the polygon $P$ into an $(m-1)$-gon $P'$ and a triangle $P''$. In the frieze, this means that we choose the entry $p_{i-1,i+1}$ in the first row. Then we can reduce categorically: $p_{i-1,i+1}$ corresponds to the module $M_{i-1,i+1}$ and the AR-quiver of $\mc{C}(M_{i-1,i+1})$ is given by the  AR-quiver of $\mc{C}$ with the indecomposables in the two regions $\ms{R}_{\rightarrow}(\tau^{-1}M_{i-1,i+1})=\ms{R}_{\rightarrow}(M_{i+1,i+2})$ and $\ms{R}_{\leftarrow}(\tau M_{i-1,i+1})=\ms{R}_{\leftarrow}(M_{i-2,i})$ deleted. This means that we delete the two diagonals $(p_{i,i+2}, \ldots, p_{i-2,i})$ and $(p_{i-2,i}, \ldots,  p_{i,i+2})$ in the frieze.  Note that in fact the two diagonals contain the same elements. For a picture, see \eqref{Fig:1strowred} (with $m=9$ and $i=5$ and the two diagonals and $p_{46}$ marked).
	
	\begin{equation}\label{Fig:1strowred}
		\adjustbox{scale=0.25,center}{
			\includegraphics{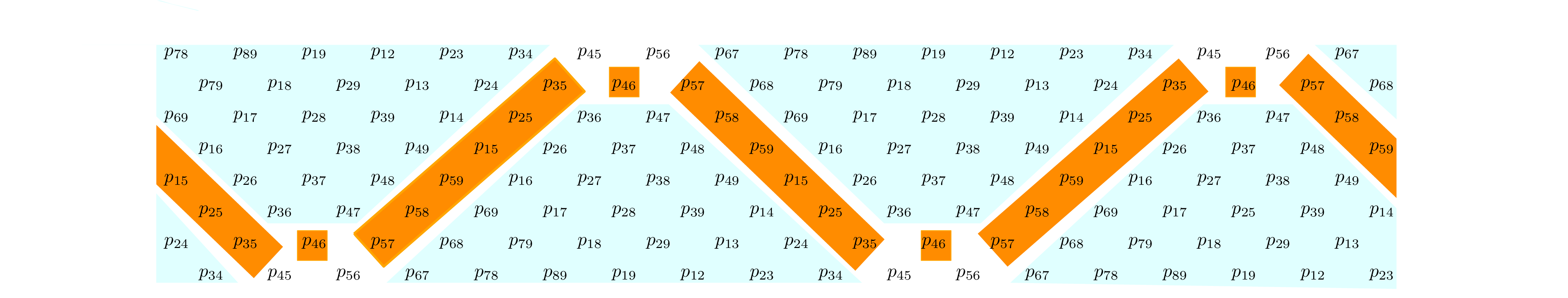}
		}
	\end{equation}

	Clearly, $i=1$ or $i=m$ if and only if  $p_{1m}$ is contained in the trivial frieze $\mc{F}''$ corresponding to the triangle $P''$. Now assume that $ i \in \{ 2, \ldots, m-1 \}$. The quiddity sequence of $\mc{F}$ is given as 
	\[ 
	\begin{array}{c}
		(a_1, a_2, \ldots, a_{i-2}, a_{i-1}, a_{i}, a_{i+1}, a_{i+2}, \ldots, a_{m-1}, a_{m}) = \\[5pt]
		= (p_{2,m}, p_{13}, \ldots, p_{i-3,i-1}, p_{i-2,i}, p_{i-1,i+1}, p_{i,i+2}, p_{i+1,i+3}, \ldots, p_{m-2,m}, p_{1,m-1}) \ . 
	\end{array}
	\]
	Then the new quiddity can be read off similarly as in \eqref{Fig:1strowred} as
	\begin{equation} \label{Eq:reducedquiddity} (p_{m,2}, p_{13}, \ldots, p_{i-3,i-1}, p_{i-2,i+1}, p_{i-1, i+2}, p_{i+1,i+3}, \ldots, p_{m-1,1}) \ .
	\end{equation}
	
\end{proof}

\begin{Bem} \label{Rmk:cuttingear}The quiddity for the reduced frieze \eqref{Eq:reducedquiddity} is precisely the formula derived in \cite[Question 23]{CoCo1} (also cf. \cite[Lemma 9]{HenrySoizic}): In Eq. (4) of \cite{HenrySoizic} it is shown that for a frieze $\mc{F}$ of width $m-3$ with quiddity $\{ a_i \}_{i=1}^{m}$ with some $a_i=1$, then one obtains a frieze $\mc{F}'$ of width $m-4$ and with quiddity
\begin{equation}\label{Eq:CCreduction} 
		\begin{array}{l}
			(a_1, \ldots, a_{i-1}, a_i=1, a_{i+1}, a_{i+2}, \ldots, a_{m})  \rightarrow  (a_1, \ldots, a_{i-1}-1, a_{i+1}+1, a_{i+2}, \ldots, a_{m}) \ .
		\end{array}
	\end{equation}
	
	In our setting,  $[i-1,i+1] \in \mc{T}$ means that $a_i=p_{i-1,i+1}=1$. Reducing the frieze with Lemma \ref{Lem:Friezered} shows that the new quiddity sequence is given as \eqref{Eq:reducedquiddity}, that is, $a_{i-1}$ is replaced by $p_{i-2,i+1}$, $a_i$ is deleted, and $a_{i+1}$ is replaced by $p_{i-1, i+2}$. The diamond rule for $\mc{F}$ shows that $a_{i-1}\cdot 1 - 1 \cdot p_{i-2,i+1} =1$ and likewise, $1\cdot a_{i+1}-1 \cdot p_{i-1,i+2}=1$, which proves \eqref{Eq:CCreduction}.
\end{Bem} 

\begin{proposition}[General reduction] \label{Prop:generalreduction}  Let $\mc{F}$ be a frieze coming from a triangulation $\mc{T}$ of an $m$-gon $P$ with quiddity sequence $\{a_i\}_{i=1}^{m}$. Assume that the diagonal $[i,j] \in \mc{T}$, 
	for some $ i, j \in \{ 1, \ldots, m \} $ with $ i < j $ and $|i-j|\geq 2$. Then the reduction $\mc{F}({[i,j]})$ is 
	\begin{enumerate}[(1)]
	\item  if $1 \leq i < j \leq m-1$: the frieze of width $m+i-j-2$, with quiddity sequence 
		$$(a_1, \ldots, a_{i-1},p_{i-1,j},p_{i,j+1},a_{j+1}, \ldots, a_m) \ ,$$ OR
		\item if $j=m$: the frieze of width $i-2$  with quiddity sequence 
		$$(a_1, \ldots, a_{i-1}, p_{i-1,m}, p_{1i}). \ $$
		
	\end{enumerate}
\end{proposition}

\begin{Bem} \label{Rmk:parts-of-frieze}
If we consider $p_{ij}$ in the fundamental domain chosen in \eqref{Eqn:fundamentaldomain} and $\mc{F}'$ and $\mc{F}''$ as in Fig.~\ref{Fig:upperlower}, then we will call $\mc{F}''$ the \emph{upper part} (which is completely contained in the fixed fundamental region) and $\mc{F}'$ the \emph{lower part} of the frieze $\mc{F}$.\\
	\begin{figure}[h!] 
		\begin{center}
			\adjustbox{scale=1,center}{
				\includegraphics{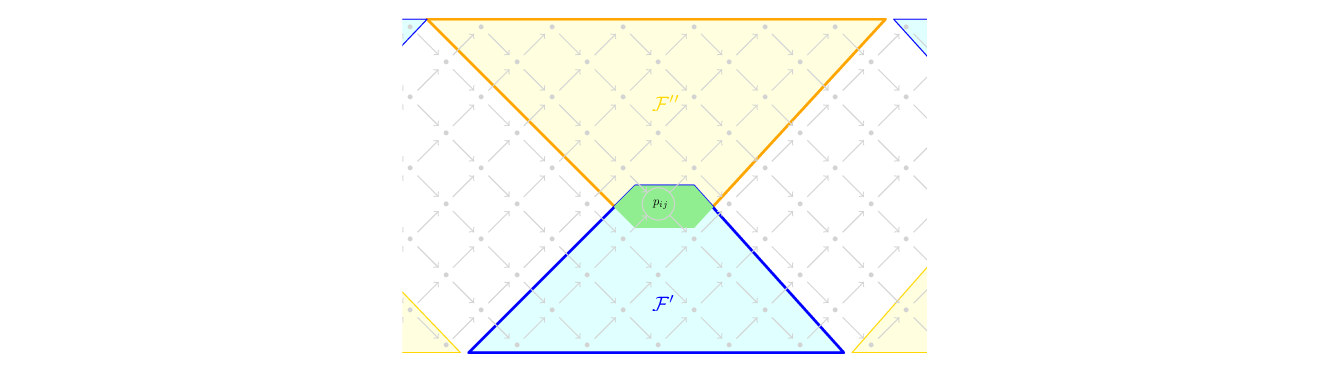} 
			}
		\end{center}
		\caption{Picture with the upper part $\mc{F}''$ and lower part $\mc{F}'$ of the frieze. \label{Fig:upperlower}}
	\end{figure}

\end{Bem}

\begin{proof} The proof is similar to the proof of Lemma \ref{Lem:Friezered}: for $[i,j] \in \mc{T}$ we reduce categorically to $\mc{C}(M_{ij})$. The AR-quiver of this category is also given as the AR-quiver of $\mc{C}$ with the rectangular regions $\ms{R}_\rightarrow(M_{ij})=\ms{R}_\leftarrow(\tau M_{ij})$ deleted. Consider the fundamental domain as in \eqref{Eqn:fundamentaldomain}.
One sees that $p_{1m}$ is in the lower part of the frieze
 if and only if $j < m$. Similarly, the new quiddity sequence can be read off the frieze (see Fig.~\ref{Fig:generalij}, for $m=9$ and $[i,j]=[4,9]$).
	\begin{figure}[h!] 
		\adjustbox{scale=0.72,center}{
			\includegraphics{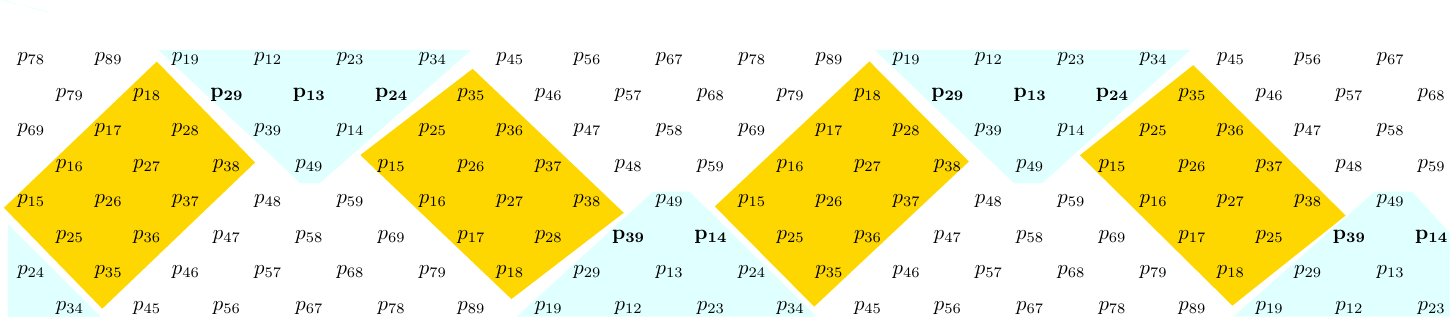}
		}
		\caption{Frieze reduced at $p_{49}$, the quiddity of $\mc{F}'$ is given by the highlighted {\bf bold} numbers.} \label{Fig:generalij}
	\end{figure}

	In the case $j=m$, $p_{1m}$ is in the upper part of the frieze and a similar argument as for the other case shows how the quiddity sequence changes.
\end{proof}

\begin{Qu}
 Is there a more elementary way to show that the reduction of $\cF$ satisfies the frieze relations and in particular, that $\cF([i,j])$ is a frieze?
\end{Qu}

\subsection{Partial resolutions and reduction of friezes} In this section the reduction of a frieze $\cF(\Lambda(f))$ as in the last Section \ref{Sub:reduction} will be related to partial lotuses of the corresponding curve singularity $C=V(f)$.

\begin{defi}
	\phantomsection
	\label{Def:partial_resol}
	\begin{enumerate}[(1)]
		\item Let $\pi=\varphi_{\ell} \circ \cdots \circ \varphi_1 \colon \widetilde S \xrightarrow{} S$ be a sequence of blowups corresponding to a subdivision of the fan $\sigma_0$ by rays $w_1, \ldots, w_\ell$, i.e., we choose a cross $(L,L')$ in $S$ and then each blowup is the blowup in the origin of one of the charts from the previous blowups. 
		The \emph{lotus associated to $\pi$} is defined as $\Lambda(\pi):=\Lambda(\mc{E})$, where $\mc{E}$ is the set of the slopes $\lambda_i$ of the $w_i$. We call the frieze $\cF(\pi):=\cF(\Lambda(\pi))$ the \emph{frieze of $\pi$.}
		\item	Let $\pi \colon \widetilde S \xrightarrow{} S$ be the minimal resolution of the Newton non-degenerate curve $(C,s)\subseteq (S,s)$ given by $f \in \CC\{x,y\}$, and let $\Lambda(f)$ be the associated lotus. Note that $\pi=\varphi_{\ell} \circ \cdots \circ \varphi_1$ can be written as a sequence of blowups in points and thus $\Lambda(f)=\Lambda(\pi)$. Then we call a sequence of blowups $\pi'=\varphi'_{\ell'} \circ \cdots \circ \varphi'_{1}\colon \widetilde{S'} \xrightarrow{} S$ a \emph{partial resolution of $C$} if $\Lambda(\pi')\subset \Lambda(f)$ is a sublotus.
	\end{enumerate}
\end{defi}

Here partial resolution means in particular that the strict transform $C'$ of $C$ under $\pi'$ is partially resolved in the sense that the Puiseux characteristics $(m; \beta_1, \ldots, \beta_g)$ of $C'$ is strictly smaller than that of $C$, for notation and precise statement see \cite[Thm.~3.5.5]{Wall}. 

Note that a sublotus of $\Lambda(f)$ must always contain the base petal, that is, at least the first blowup must be $\varphi_1$.

The main result of this section is to connect reduction of friezes to partial resolutions: 
\begin{Thm}
	\label{Thm:Partial_reduction}
	Consider the curve $C=V(f)$, where $f$ is assumed to be Newton non-degenerate, and its minimal resolution $\pi$.  
	Then the frieze $\cF(\pi')$ of any partial resolution $\pi'$ of $\pi$ is obtained as a reduction of the frieze of $\pi$. In particular, if the dual resolution graph $\Gamma(f)$ is of type $A_{m-2}$ with self-intersection numbers $\{-a_i\}_{i=2}^{m-1}$, then  the dual graph of the exceptional curves appearing in $\pi'$ is of type $A_k$ for some $k \leq m-2$ and the self-intersection numbers $\{-b_j\}_{j=1}^{k}$ correspond to negatives of entries in the frieze of $\pi$.
\end{Thm}

\begin{proof}
	Since $C=V(f)$ is Newton non-degenerate, its minimal resolution $\pi$ consists of Steps 1 to $3^{reg}$ of the resolution algorithm of \cite[Algorithm 1.4.22 and Prop.~1.4.29]{Lotus} and hence the dual resolution graph corresponds to the boundary of the Newton lotus $\Lambda(f)$, see Thm.~\ref{Thm:Lotus-dual}.  The same theorem shows that $\Lambda(f)=\Lambda(\pi)$ corresponds to a triangulated $m$-gon $P$. The quiddity sequence of the frieze of $\pi$ is then $\{a_i\}_{i=1}^{m}$, where the $a_i$'s for $ i \in \{ 2, \ldots, m-1 \} $ are precisely the negatives of the self-intersection numbers appearing in $\Gamma(f)$ and $a_1$ and $a_m$ are uniquely determined by $P$, see Lemma \ref{Lem:quidditybase}. \\
	By definition of a partial resolution $\pi'$ (Def.~\ref{Def:partial_resol}~(2)), the lotus of $\pi'$ corresponds to one of the two smaller polygons $P'$ and $P''$ obtained by  cutting $P$ along one of the diagonals in its triangulation. Without loss of generality, assume that the lotus of $\pi'$ corresponds to $P'$. Using Prop.~\ref{Prop:generalreduction} one can compute the quiddity sequence of the frieze of $\pi'$ explicitly: if $[i,j]$ is an diagonal in the triangulation of $P$ (corresponding to the chart of a blowup), then the elements in the quiddity are entries in the frieze of $\pi$ given by one of the two formulas in the proposition.
\end{proof}

This result allows us to see that the entries $p_{ij}$ of the frieze that appear in the quiddity sequences of the reduction of a frieze also correspond to (negatives) of the self-intersection numbers of exceptional curves. In particular, the frieze of $\pi$ contains all the information about any partial resolution $\pi'$. 

\begin{example} \label{Ex:entries-frieze} In our running example $C=V(f)$ with $f=x^{11}-y^8$ there are 6 partial resolutions of $C$, their graphs are given below in Fig. \ref{Fig:partresgraphs}.
	\begin{figure}[h!]
		\begin{tabular}{lll}
		\begin{minipage}{7cm}
			\begin{tikzpicture}[scale=0.6]
				
				\draw [-, color=black, thick](0,0) -- (10,0);
				
				\node [below] at (0,0) {\rot{$-2$}};
				\node [below] at (2,0) {\rot{$-3$}};
				\node [below] at (4,0) {\rot{$-2$}};
				\node [below] at (6,0) {\rot{$-1$}};
				\node [below] at (8,0) {\rot{$-3$}};
				\node [below] at (10,0) {\rot{$-4$}};

				\node[draw,circle, inner sep=1.pt,fill=black] at (0,0){};
				\node[draw,circle, inner sep=1.pt,fill=black] at (2,0){};
				\node[draw,circle, inner sep=1.pt,fill=black] at (4,0){};
				\node[draw,circle, inner sep=1.pt,fill=black] at (6,0){};
				\node[draw,circle, inner sep=1.pt,fill=black] at (8,0){};
				\node[draw,circle, inner sep=1.pt,fill=black] at (10,0){};
				
			\end{tikzpicture}
		\end{minipage}
		& \hspace{0.2cm} &
		\begin{minipage}{6cm}
			\begin{tikzpicture}[scale=0.6]
				\draw [-, color=black, thick](0,0) -- (8,0);
				
				\node [below] at (0,0) {\rot{$-2$}};
				\node [below] at (2,0) {\rot{$-3$}};
				\node [below] at (4,0) {{\color{blue}{$-1$}}};
				\node [below] at (6,0) {{\color{blue}{$-2$}}};
				\node [below] at (8,0) {\rot{$-4$}};

				\node[draw,circle, inner sep=1.pt,fill=black] at (0,0){};
				\node[draw,circle, inner sep=1.pt,fill=black] at (2,0){};
				\node[draw,circle, inner sep=1.pt,fill=black] at (4,0){};
				\node[draw,circle, inner sep=1.pt,fill=black] at (6,0){};
				\node[draw,circle, inner sep=1.pt,fill=black] at (8,0){};
				
			\end{tikzpicture}
		\end{minipage}
		\\[10pt]
		\begin{minipage}{6cm}
			\begin{tikzpicture}[scale=0.7]
				\draw [-, color=black, thick](0,0) -- (6,0);
				
				\node [below] at (0,0) {\rot{$-2$}};
				\node [below] at (2,0) {{\color{green}$-2$}};
				\node [below] at (4,0) {{\color{green}$-1$}};
				\node [below] at (6,0) {\rot{$-4$}};

				\node[draw,circle, inner sep=1.pt,fill=black] at (0,0){};
				\node[draw,circle, inner sep=1.pt,fill=black] at (2,0){};
				\node[draw,circle, inner sep=1.pt,fill=black] at (4,0){};
				\node[draw,circle, inner sep=1.pt,fill=black] at (6,0){};
				
			\end{tikzpicture}
		\end{minipage}
		& &
		\begin{minipage}{6cm}
			\begin{tikzpicture}[scale=0.7]
				\draw [-, color=black, thick](0,0) -- (4,0);
				
				\node [below] at (0,0) {\rot{$-2$}};
				\node [below] at (2,0) {{\color{orange}$-1$}};
				\node [below] at (4,0) {{\color{orange}$-3$}};

				\node[draw,circle, inner sep=1.pt,fill=black] at (0,0){};
				\node[draw,circle, inner sep=1.pt,fill=black] at (2,0){};
				\node[draw,circle, inner sep=1.pt,fill=black] at (4,0){};
			\end{tikzpicture}
		\end{minipage}
		\\[10pt]
		\begin{minipage}{6cm}
			\begin{tikzpicture}[scale=0.8]
				\draw [-, color=black, thick](0,0) -- (2,0);
				
				\node [below] at (0,0) {{\color{violet}$-1$}};
				\node [below] at (2,0) {{\color{violet}$-2$}};

				\node[draw,circle, inner sep=1.pt,fill=black] at (0,0){};
				\node[draw,circle, inner sep=1.pt,fill=black] at (2,0){};

			\end{tikzpicture}
		\end{minipage}
	&  &
		\begin{minipage}{6cm}
			\begin{tikzpicture}[scale=0.8]
				
				\node [below] at (0,0) {$-1$};
				
				\node[draw,circle, inner sep=1.pt,fill=black] at (0,0){};
				
			\end{tikzpicture}
		\end{minipage}
	\end{tabular}

		\caption{Dual graphs of the partial resolutions of $V(x^{11}-y^8)$.} \label{Fig:partresgraphs}
	\end{figure}
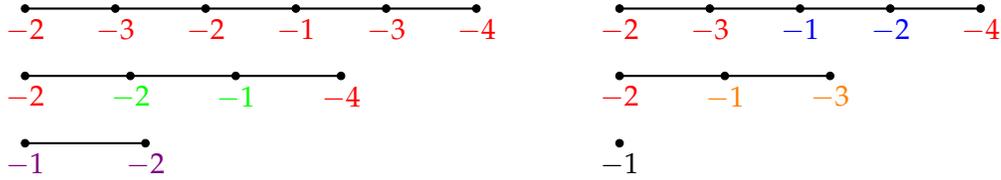

	In Fig.~\ref{Fig:partialResRunning} the corresponding entries of the frieze are highlighted.
	
	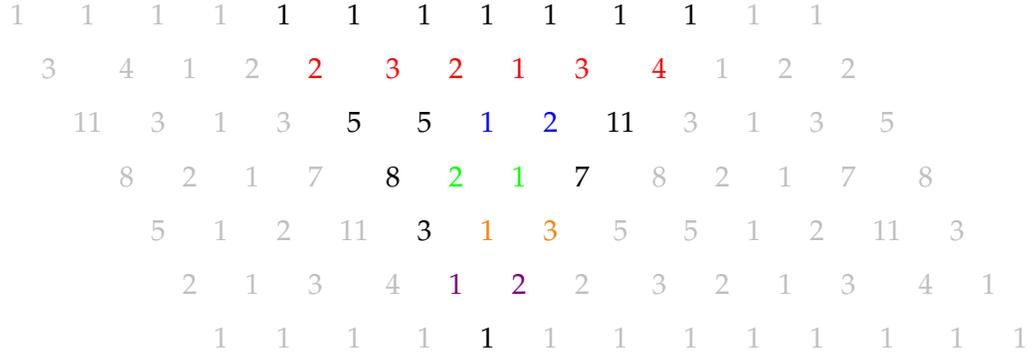
\begin{figure}[h!]
	\begin{tikzcd}[row sep=0.35em, column sep=-0.35em]
		\col{1}  && \col{1}  && \col{1}  && \col{1}   && 1   && 1  && 1   && 1   && 1 &&1 && 1  && \col{1}  && \col{1}  \\
		& \col{3} && \col{4}  && \col{1}     && \col{2}     && \rot{2}     && \rot{3}     && \rot{2}    && \rot{1}    && \rot{3} && \rot{4} && \col{1}  && \col{2}  && \col{2}     \\
		&& \col{11} && \col{3} && \col{1}     && \col{3}     && 5     && 5     && {\color{blue}{1}}    && {\color{blue}{2}}     && 11 &&\col{3} && \col{1}  &&\col{3}  &&\col{5}    \\
		&  && \col{8} && \col{2} && \col{1}     && \col{7}     && 8     && {\color{green}2 }      && {\color{green} 1}     && 7     && \col{8} &&\col{2} && \col{1} && \col{7}  && \col{8}   \\
		 &&  &&\col{5} &&\col{1} && \col{2}     && \col{11}   && 3     && {\color{orange}1}      && {\color{orange}3}   && \col{5}     && \col{5} &&\col{1} && \col{2} && \col{11}  && \col{3}   \\
		  &&&  && \col{2} && \col{1} && \col{3}     && \col{4}   && {\color{violet} 1}     && {\color{violet} 2}      && \col{2}   && \col{3}     &&\col{2} &&\col{1} && \col{3} && \col{4}  && \col{1}   \\
		 &&&&  && \col{1} && \col{1}  && \col{1}   && \col{1}  && 1   && \col{1}   && \col{1}   && \col{1}  && \col{1}  && \col{1} &&\col{1}  && \col{1}  && \col{1}  
	\end{tikzcd}

		\caption{Entries in the frieze corresponding to partial resolutions of $ V( x^{11}-y^8 ) $.}
		\label{Fig:partialResRunning}
	\end{figure}
	
\end{example}

\begin{Bem}
	With the interpretation of dual resolution graphs as boundary edges of triangulated polygons, one can now also count the number of partial resolutions of a curve $C$ as the number of triangulated sub-polygons of the lotus of $C$. Explicitly, if $C$ is irreducible, that is, $\Lambda(C)$ is determined by a continued fraction, then if $\Gamma(C)$ has $n$ nodes, there are $n$ partial resolutions of $C$ (see Example \ref{Ex:entries-frieze}). \\
	For dual resolution graphs $\Gamma(C)$ containing more than one $-1$ decoration, an explicit formula is more complicated to write down and depends on the configuration of the blowups. 
\end{Bem}

\begin{Bem} In Fig.~\ref{Fig:partialResRunning} one can see that not all entries in the fundamental domain of a frieze $\mc{F}(f)$ have an interpretation in terms of negatives of the self-intersection numbers appearing in the dual graph of a partial resolution of $f$. It would be very interesting to find such an interpretation for the remaining entries.
\end{Bem}

\section{Mutation of polygons in terms of lotuses} \label{Sec:mutation}

Given a triangulation $\mc{T}$ of a polygon $ P $, there is the notion of {\em mutation of an inner diagonal}: 
Let $[a,b]$ be an inner diagonal with vertices $ a, b $ of $ P $
and 
let $ c, d $ be the vertices of $ P $ completing $[a,b]$ to a quadrilateral within the triangulation.
The mutation of $[a,b]$ is obtained by flipping the diagonal, i.e.,
we remove $[a,b]$ from the triangulation and introduce the new diagonal $[c,d]$,
 as illustrated in Fig.~\ref{Fig:mutation_diag}. 
 	\begin{figure}[h!]
	\begin{center}
		\begin{tikzpicture}[scale=0.6]
		\node[below] at (0,0) {$ a $};
		\node[right] at (2,1) {$ c $};
		\node[above] at (1,3) {$ b $};
		\node[left] at (-1,2) {$ d $};
		
		\draw (0,0) -- (2,1) -- (1,3) -- (-1,2) -- (0,0);	
		\draw (0,0) -- (1,3);
		\end{tikzpicture}
		\hspace{.4cm}
		\raisebox{1.6cm}{
		\begin{tikzpicture}[scale=1]
			\draw[->] (0,0) -- (2,0);
			\node[above] at (1,0) {\small mutation};
		\end{tikzpicture}
		}
		\hspace{.4cm}
		\begin{tikzpicture}[scale=0.6]
		\node[below] at (0,0) {$ a $};
		\node[right] at (2,1) {$ c $};
		\node[above] at (1,3) {$ b $};
		\node[left] at (-1,2) {$ d $};
		
		\draw (0,0) -- (2,1) -- (1,3) -- (-1,2) -- (0,0);	
		\draw (2,1) -- (-1,2);
		\end{tikzpicture}
	\end{center}
	\caption{Mutation of a diagonal in the corresponding quadrilateral.}  
	\label{Fig:mutation_diag} 
\end{figure}
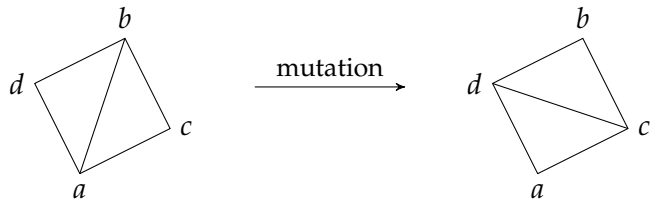%     

Clearly, this process provides a new triangulation $ \mu_{[a,b]} (\mc{T}) $ of $ P $ and by mutating at the new diagonal $[c,d]$, we regain the original triangulation of $ P $,
i.e., $ \mu_{[c,d]} (\mu_{[a,b]} (\mc{T})) = \mc{T} $.

Let us point out that this notion of mutation is connected to the mutation of quivers resp.~cluster algebras,
see \cite{CCS06} or \cite[\S~2.2]{FWZ2016}. 

Let us have a look how the mutation of inner diagonals for a triangulated pentagon affect the corresponding lotus.  

\begin{example}
	In Fig.~\ref{Fig:5gon}, we fix the notation for the vertices of the pentagon.
	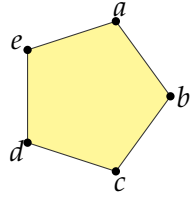
\begin{figure}[h!]
		\begin{center}	
			\begin{tikzpicture}[scale=0.7,cap=round,>=latex]
				
				%decorations
				\node  at (0:1.75) {$b$};
				\node at (72:1.75) {$a$};
				\node at (144:1.75) {$e$};
				\node at (216:1.75) {$d$};
				\node at (288:1.75) {$c$};
				
				\coordinate  (p1) at (0:1.5) {}; 
				\coordinate  (p2) at (72:1.5) {};
				\coordinate (p3) at (144:1.5) {};
				\coordinate (p4) at (216:1.5) {};
				\coordinate (p5) at (288:1.5) {};

				%the triangles                        
				
				\filldraw [fill=yellow!50,draw=black!80] (p1)--(p2)--(p3)--(p4)--(p5) -- cycle;

				%marked points		
				\draw (p1) node[fill=black,circle,inner sep=0.039cm] {} circle (0.01cm);	        
				\draw (p2) node[fill=black,circle,inner sep=0.039cm] {} circle (0.01cm);
				\draw (p3) node[fill=black,circle,inner sep=0.039cm] {} circle (0.01cm);
				\draw (p4) node[fill=black,circle,inner sep=0.039cm] {} circle (0.01cm);
				\draw (p5) node[fill=black,circle,inner sep=0.039cm] {} circle (0.01cm);

			\end{tikzpicture}
		\end{center}
		\caption{Notation for the vertices of the pentagon.}  
		\label{Fig:5gon} 
	\end{figure}     
	
	There are five different ways to triangulate the pentagon. 
	In Fig.~\ref{Fig:5-gon_mutation}, we picture all five lotuses and 
	discuss the transformation of the lotus under the process of mutating the inner edges.
	We also provide the equation of the curve (with minimal number of irreducible factors) corresponding to the drawn lotus.  

	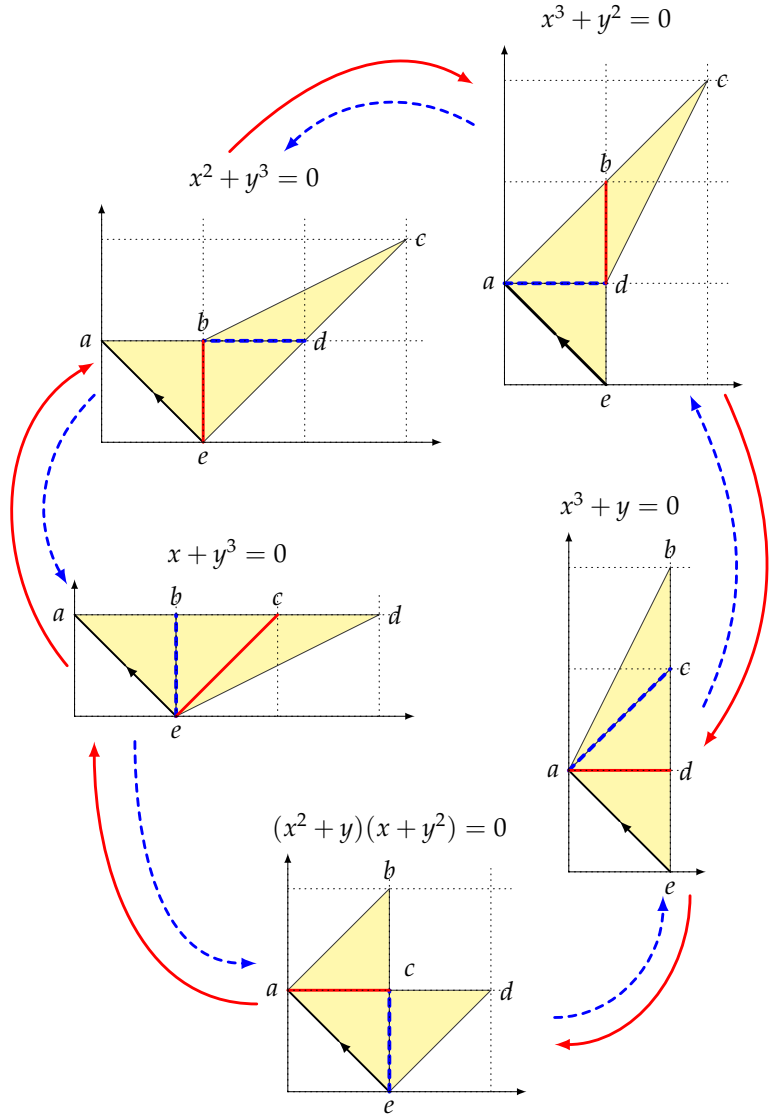
\begin{figure}[h!]
		\begin{center} 
			\scalebox{0.9}{
			\begin{tikzpicture}[scale=0.8,cap=round,>=latex]
				\coordinate  (p1) at (-3,-1) {}; 
				\coordinate  (p2) at (1.25,4) {};
				\coordinate (p3) at (1,13) {};
				\coordinate (p4) at (-5.5,11) {};
				\coordinate (p5) at (-6,5) {}; 
				
				% von p1 nach p2
				\draw [->][very thick, color=blue, dashed] (0,-2) to [out=0,in=-90] (2,0.25); 
				\draw [->][very thick, color=red] (2.5,0.25) to [out=-90,in=0] (0,-2.5); 
				
				% von p2 nach p3
				\draw [->][very thick, color=blue, dashed] (2.75,3.75) to [out=65,in=-65] (2.5,9.5); 
				\draw [->][very thick, color=red] (3.15,9.5) to [out=-65,in=55] (2.75,3);

				% von p3 nach p4
				\draw [->][very thick, color=blue, dashed] (-1.5,14.5) to [out=150,in=45] (-5,14); 
				\draw [->][very thick, color=red] (-6,14) to [out=45,in=150] (-1.5,15.25);

				% von p4 nach p5
				\draw [->][very thick, color=blue, dashed] (-8.5,9.5) to [out=-135,in=120] (-9,6);
				\draw [->][very thick, color=red] (-9,4.5) to [out=130,in=-150] (-8.5,10.1);

				% von p5 nach p1
				\draw [->][very thick, color=blue, dashed] (-7.75,3.1) to [out=-90,in=180] (-5.5,-1); 
				\draw [->][very thick, color=red] (-5.5,-1.75) to [out=180,in=-90] (-8.5,3.1);

				\node at (p1) {
					\begin{minipage}{9cm}
						\begin{center}
							\begin{tikzpicture}[scale=1.5]
								
								\filldraw [fill=yellow!40,draw=black!80] (1,0) -- (0,1) -- (1,1) -- cycle;
								\filldraw [fill=yellow!40,draw=black!80] (1,0) -- (1,1) -- (2,1) -- cycle;
								\filldraw [fill=yellow!40,draw=black!80] (0,1) -- (1,1) -- (1,2) -- cycle;
								
								\draw [->, thick] (1,0)--(0.5, 0.5);
								\draw [-, thick] (0.5, 0.5)--(0,1);

								\draw [-, ultra thick, blue, dashed] (1, 0)--(1,1);
								\draw [-, very thick, red] (0,1)--(1,1);

								\draw [dotted] (0,0) grid (2.2,2.2);
								\draw [->] (0,0) -- (2.35,0);
								\draw [->] (0,0) -- (0,2.35);

								\node [below] at (1,0) {$ e $}; 
								\node [left] at (0,1) {$ a $}; 
								\node [above] at (1,2) {$ b $}; 
								\node at (1.2,1.2) {$ c $}; 
								\node [right] at (2,1) {$ d $};

								\node at (1,2.6) {$ (x^2 + y) (x + y^2) = 0 $}; 
								
							\end{tikzpicture}
						\end{center}
					\end{minipage}
				};

				\node at (p2) {
					\begin{minipage}{9cm}
						\begin{center}
							\begin{tikzpicture}[scale=1.5]
								
								\filldraw [fill=yellow!40,draw=black!80] (1,0) -- (0,1) -- (1,1) -- cycle;
								\filldraw [fill=yellow!40,draw=black!80] (0,1) -- (1,1) -- (1,2) -- cycle;
								\filldraw [fill=yellow!40,draw=black!80] (0,1) -- (1,2) -- (1,3) -- cycle;
								
								\draw [->, thick] (1,0)--(0.5, 0.5);
								\draw [-, thick] (0.5, 0.5)--(0,1);

								\draw [-, ultra thick, blue, dashed] (0,1)--(1,2);
								\draw [-, very thick, red] (0,1)--(1,1);
								
								\draw [dotted] (0,0) grid (1.2,3.2);
								\draw [->] (0,0) -- (1.35,0);
								\draw [->] (0,0) -- (0,3.35);

								\node [below] at (1,0) {$ e $}; 
								\node [left] at (0,1) {$ a $}; 
								\node [above] at (1,3) {$ b $}; 
								\node [right] at (1,2) {$ c $}; 
								\node [right] at (1,1) {$ d $};

								\node at (0.5,3.6) {$ x^3 + y = 0 $}; 
							\end{tikzpicture}
						\end{center}
					\end{minipage}
				};

				\node at (p3) {
					\begin{minipage}{9cm}
						\begin{center}
							\begin{tikzpicture}[scale=1.5]
								
								\filldraw [fill=yellow!40,draw=black!80] (1,0) -- (0,1) -- (1,1) -- cycle;
								\filldraw [fill=yellow!40,draw=black!80] (0,1) -- (1,1) -- (1,2) -- cycle;
								\filldraw [fill=yellow!40,draw=black!80] (1,1) -- (1,2) -- (2,3) -- cycle;
								
								\draw [->, very thick] (1,0)--(0.5, 0.5);
								\draw [-, very thick] (0.5, 0.5)--(0,1);
								
								\draw [-, ultra thick, blue, dashed] (0, 1)--(1,1);
								\draw [-, very thick, red] (1, 1)--(1,2);
								
								\draw [dotted] (0,0) grid (2.2,3.2);
								\draw [->] (0,0) -- (2.35,0);
								\draw [->] (0,0) -- (0,3.35);

								\node [below] at (1,0) {$ e $}; 
								\node [left] at (0,1) {$ a $}; 
								\node [above] at (1,2) {$ b $}; 
								\node [right] at (2,3) {$ c $}; 
								\node [right] at (1,1) {$ d $};

								\node at (1,3.6) {$ x^3 + y^2 = 0 $}; 
								
							\end{tikzpicture}
						\end{center}
					\end{minipage}
				};

				\node at (p4) {
					\begin{minipage}{9cm}
						\begin{center}
							\begin{tikzpicture}[scale=1.5]
								
								\filldraw [fill=yellow!40,draw=black!80] (1,0) -- (0,1) -- (1,1) -- cycle;
								\filldraw [fill=yellow!40,draw=black!80] (1,0) -- (1,1) -- (2,1) -- cycle;
								\filldraw [fill=yellow!40,draw=black!80] (1,1) -- (2,1) -- (3,2) -- cycle;
								
								\draw [->, thick] (1,0)--(0.5, 0.5);
								\draw [-, thick] (0.5, 0.5)--(0,1);

								\draw [-, ultra thick, blue, dashed] (2, 1)--(1,1);
								\draw [-, very thick, red] (1, 0)--(1,1);
								
								\draw [dotted] (0,0) grid (3.2,2.2);
								\draw [->] (0,0) -- (3.35,0);
								\draw [->] (0,0) -- (0,2.35);

								\node [below] at (1,0) {$ e $}; 
								\node [left] at (0,1) {$ a $}; 
								\node [above] at (1,1) {$ b $}; 
								\node [right] at (3,2) {$ c $}; 
								\node [right] at (2,1) {$ d $};

								\node at (1.5,2.6) {$ x^2 + y^3 = 0 $}; 
								
							\end{tikzpicture}
						\end{center}
					\end{minipage}
				};

				\node at (p5) {
					\begin{minipage}{9cm}
						\begin{center}
							\begin{tikzpicture}[scale=1.5]
								
								\filldraw [fill=yellow!40,draw=black!80] (1,0) -- (0,1) -- (1,1) -- cycle;
								\filldraw [fill=yellow!40,draw=black!80] (1,0) -- (1,1) -- (2,1) -- cycle;
								\filldraw [fill=yellow!40,draw=black!80] (1,0) -- (2,1) -- (3,1) -- cycle;
								
								\draw [->, thick] (1,0)--(0.5, 0.5);
								\draw [-, thick] (0.5, 0.5)--(0,1);

								\draw [-, ultra thick, blue, dashed] (1, 0)--(1,1);
								\draw [-, very thick, red] (1, 0)--(2,1);
								
								\draw [dotted] (0,0) grid (3.2,1.2);
								\draw [->] (0,0) -- (3.35,0);
								\draw [->] (0,0) -- (0,1.35);

								\node [below] at (1,0) {$ e $}; 
								\node [left] at (0,1) {$ a $}; 
								\node [above] at (1,1) {$ b $}; 
								\node [above] at (2,1) {$ c $}; 
								\node [right] at (3,1) {$ d $};

								\node at (1.5,1.6) {$ x + y^3 = 0 $}; 
								
							\end{tikzpicture}
							
						\end{center}
					\end{minipage}
				};
			\end{tikzpicture}
		}
		\end{center}
		\caption{Mutations of the pentagon illustrated for the lotus with base petal determined by the edge $ [ a, e ] $.
			Mutating the (blue) dashed inner edge of a lotus leads to the next counterclockwise neighbor in the picture.
			Mutating the other (red) inner edge brings us to the clockwise neighbor.}  
		\label{Fig:5-gon_mutation} 
	\end{figure}     
\end{example}

In order to provide the general interpretation of the mutation of inner diagonals for lotuses, 
we have to introduce a way to encode the data of the lotus appropriately. 

\begin{Not}
	\label{Not:Not_mut} 
	Let $ P $ be a triangulated polygon
	and
	let $ \Lambda = \Lambda (P) $ be an embedding of $ P $ into the universal lotus with respect to a chosen basis $ (e_1, e_2) $.
	Fix an inner diagonal $ [ a, b ] $ of $ P $ and let $ c, d $ be the other vertices of $ P $ completing the inner diagonal to the quadrilateral of the triangulation whose diagonal is $  [a, b ] $.
	We denote by $ \alpha, \beta, \gamma, \delta $ the remaining part of the lotus $ \Lambda $, which is connected to the edge $ [ a, c ] $, $ [b,c] $, $ [b,d] $, $ [a,d] $ respectively.
	\\
	We choose the names of the vertices such that $ \alpha $ contains the base petal $ \delta(e_1,e_2) $, 
	as shown in Fig.~\ref{Fig:Not_mutation}. 
	In particular, this implies that $ b $ is the third vertex of the petal $ \delta(a,c) $
	and  that $ d $ is the third vertex of the petal $ \delta(a,b) $.  
	Furthermore, this leads to a natural distinction into two types:

	\begin{enumerate}[$(i)$]
		\item
		when the clockwise ordering of the vertices of the petal $ \delta(a,c) $ is $ (a,c,b) $,
		and
		
		\item
		when the clockwise ordering of the vertices of the petal $ \delta(a,c) $ is $ (a,b,c) $. 
	\end{enumerate}
	
	Notice that it is possible that $ \beta, \gamma $ or $ \delta $ are empty.

	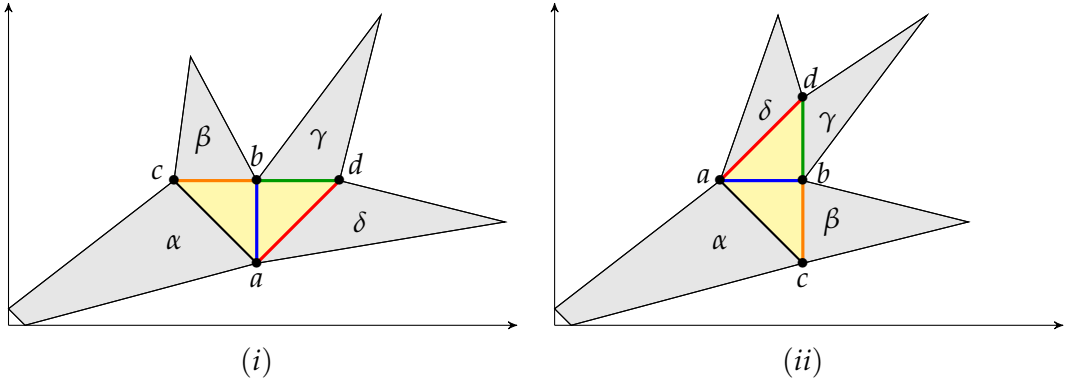
\begin{figure}[h!]
	\begin{center}	
		\begin{tikzpicture}[scale=1.1]
			% NAME DES BILDES
			\node at (1, -1.2) {$(i)$};
			
			%%%%% FILLING IF NEEDED		
			\draw [fill=gray!20] (1,0) -- (-1.8,-0.75) -- (-2,-0.55) -- (0,1) -- (0.2,2.5) -- (1,1) -- (2.5,3) -- (2,1)  -- (4,0.5) -- (1,0);
			\draw [fill=yellow!40] (1,0) -- (0,1) -- (1,1) -- (2,1) -- (1,0);

			%%%% UNCHANGED PART OF LOTUS
			\draw (1,0) -- (-1.8,-0.75) -- (-2,-0.55) -- (0,1);
			\node at (0,0.25) {$\alpha$};
			
			\draw (0,1) -- (0.2,2.5) -- (1,1);
			\node at (0.35,1.5) {$\beta$};
			
			\draw (1,1) -- (2.5,3) -- (2,1);
			\node at (1.75,1.5) {$\gamma$};
			
			\draw (1,0) -- (4,0.5) -- (2,1);
			\node at (2.25,0.5) {$\delta$};

			%%%%% RELEVANT EDGES
			\draw [very thick, blue] (1,0) -- (1,1);
			\draw [very thick, red] (1,0) -- (2,1);
			\draw [very thick, green!60!black] (1,1) -- (2,1);
			\draw [very thick, orange] (1,1) -- (0,1);
			\draw[thick] (1,0) -- (0,1);

			%%%% VERTICES
			\node at (1,0) {\footnotesize $ \bullet $};
			\node at (1,-0.2) {$ a $};
			 
			\node at (2,1) {\footnotesize $ \bullet $};
			\node at (2.2,1.2) {$ d $};
			 
			\node at (1,1) {\footnotesize $ \bullet $}; 
			\node at (1,1.3) {$ b $}; 
			
			\node at (0,1) {\footnotesize $ \bullet $}; 
			\node at (-0.2,1.1) {$ c $};

			%%%% COORDINATE AXES 
			\draw[->] (-2,-0.75)--(-2,3.15);
			\draw[->] (-2,-0.75)--(4.15,-0.75);
		\end{tikzpicture}
		\ \ \  \ 
		\begin{tikzpicture}[scale=1.1]
			% NAME DES BILDES
			\node at (1, -1.2) {$(ii)$};

			%%%%% FILLING IF NEEDED		
			\draw [fill=gray!20] (1,0) -- (-1.8,-0.75) -- (-2,-0.55) -- (0,1) --  (0.7,3) -- (1,2)  -- (2.5,3) -- (1,1) -- (3,0.5) -- (1,0);
			\draw [fill=yellow!40] (1,0) -- (1,1) -- (1,2) -- (0,1) -- (1,0);

			%%%% UNCHANGED PART OF LOTUS
			\draw (1,0) -- (-1.8,-0.75) -- (-2,-0.55) -- (0,1);
			\node at (0,0.25) {$\alpha$};
			
			\draw (0,1) -- (0.7,3) -- (1,2);
			\node at (0.55,1.85) {$\delta$};
			
			\draw (1,1) -- (2.5,3) -- (1,2);
			\node at (1.3,1.7) {$\gamma$};
			
			\draw (1,0) -- (3,0.5) -- (1,1);
			\node at (1.35,0.5) {$\beta$};

			%%%%% RELEVANT EDGES
			\draw [very thick, blue] (0,1) -- (1,1);
			\draw [very thick, orange] (1,0) -- (1,1);
			\draw [very thick, green!60!black] (1,1) -- (1,2);
			\draw [very thick, red] (1,2) -- (0,1);
			\draw[thick] (1,0) -- (0,1);

			%%%% VERTICES
			\node at (1,0) {\footnotesize $ \bullet $};
			\node at (1,-0.2) {$ c $};
			
			\node at (1,1) {\footnotesize $ \bullet $};
			\node at (1.25,1.1) {$ b $};
			
			\node at (1,2) {\footnotesize $ \bullet $}; 
			\node at (1.1,2.25) {$ d $}; 
			
			\node at (0,1) {\footnotesize $ \bullet $}; 
			\node at (-0.2,1.05) {$ a $};

			%%%% COORDINATE AXES 
			\draw[->] (-2,-0.75)--(-2,3.15);
			\draw[->] (-2,-0.75)--(4.15,-0.75);
		\end{tikzpicture}
	\end{center}
	\caption{Illustration of the notation introduced in Notation~\ref{Not:Not_mut}.}  
	\label{Fig:Not_mutation} 
\end{figure}     
	
\end{Not}

\begin{Thm}
	\label{Thm:Mutation}
	Let $\mc{T}$ be a triangulation of a polygon $ P $ and
	let $ \Lambda = \Lambda (P)$ be an embedding into the universal lotus with respect to a chosen basis $ (e_1, e_2) $.
	Fix an inner diagonal $ [a, b] $.
	Using Notation \ref{Not:Not_mut}, the mutation $ \mu_{[a,b]} (\mc{T}) $ of the diagonal $  [a, b]  $ leads to the following transformation of the lotus $ \Lambda \mapsto \mu_{[a,b]} (\Lambda ) $, where the lotus $\mu_{[a,b]} (\Lambda )$ is defined by:
	\begin{enumerate}[(1)]
		\item 
		If $ \Lambda $ is of type $ (i) $,
		then $  \mu_{[a,b]} (\Lambda ) $ is the lotus of type $ (ii) $ for which the clockwise ordering of the vertices of the quadrilateral with vertices $ \{ a,b,c,d\} $ is the same as in $ \Lambda $ and for which the vertices $ a $ and $ c $ are at the same position as in $ \Lambda $.
		
		\item 
		If $ \Lambda $ is of type $ (ii) $,
		then $  \mu_{[a,b]} (\Lambda ) $ is the lotus of type $ (i) $ for which the analogous rules as in (1) are applied. 
	\end{enumerate} 

	Notice that in both cases $ \alpha, \beta, \gamma, \delta $ are unchanged and are suitably fitted into the universal lotus. 
	In Fig.~\ref{Fig:Lot_mutation}, we provide schematic pictures of the ```mutated lotus" $ \mu_{[a,b]} (\Lambda ) $.  
\end{Thm}

	\begin{figure}[h!]
	\begin{center}	
		\begin{tikzpicture}[scale=1.1]
			% NAME DES BILDES
			\node at (1, -1.2) {$\mu_{[a,b]} (\Lambda ) $ if $ \Lambda $ is of type $ (i)$};

			%%%%% FILLING IF NEEDED		
			\draw [fill=gray!20] (1,0) -- (-1.8,-0.75) -- (-2,-0.55) -- (0,1) --  (0.7,3) -- (1,2)  -- (2.5,3) -- (1,1) -- (3,0.5) -- (1,0);
			\draw [fill=yellow!40] (1,0) -- (1,1) -- (1,2) -- (0,1) -- (1,0);

			%%%% UNCHANGED PART OF LOTUS
			\draw (1,0) -- (-1.8,-0.75) -- (-2,-0.55) -- (0,1);
			\node at (0,0.25) {$\alpha$};
			
			\draw (0,1) -- (0.7,3) -- (1,2);
			\node at (0.55,1.85) {$\beta$};
			
			\draw (1,1) -- (2.5,3) -- (1,2);
			\node at (1.3,1.7) {$\gamma$};
			
			\draw (1,0) -- (3,0.5) -- (1,1);
			\node at (1.35,0.5) {$\delta$};

			%%%%% RELEVANT EDGES
			\draw [very thick, purple] (0,1) -- (1,1);
			\draw [very thick, red] (1,0) -- (1,1);
			\draw [very thick, green!60!black] (1,1) -- (1,2);
			\draw [very thick, orange] (1,2) -- (0,1);
			\draw[thick] (1,0) -- (0,1);

			%%%% VERTICES
			\node at (1,0) {\footnotesize $ \bullet $};
			\node at (1,-0.2) {$ a $};
			
			\node at (1,1) {\footnotesize $ \bullet $};
			\node at (1.25,1.1) {$ d $};
			
			\node at (1,2) {\footnotesize $ \bullet $}; 
			\node at (1.1,2.25) {$ b $}; 
			
			\node at (0,1) {\footnotesize $ \bullet $}; 
			\node at (-0.2,1.05) {$ c $};

			%%%% COORDINATE AXES 
			\draw[->] (-2,-0.75)--(-2,3.15);
			\draw[->] (-2,-0.75)--(4.15,-0.75);
		\end{tikzpicture}
			\ \ \  \ 
	\begin{tikzpicture}[scale=1.1]
		% NAME DES BILDES
		\node at (1, -1.2) {$\mu_{[a,b]} (\Lambda ) $ if $ \Lambda $ is of type $ (ii)$};
		
		%%%%% FILLING IF NEEDED		
		\draw [fill=gray!20] (1,0) -- (-1.8,-0.75) -- (-2,-0.55) -- (0,1) -- (0.2,2.5) -- (1,1) -- (2.5,3) -- (2,1)  -- (4,0.5) -- (1,0);
		\draw [fill=yellow!40] (1,0) -- (0,1) -- (1,1) -- (2,1) -- (1,0);

		%%%% UNCHANGED PART OF LOTUS
		\draw (1,0) -- (-1.8,-0.75) -- (-2,-0.55) -- (0,1);
		\node at (0,0.25) {$\alpha$};
		
		\draw (0,1) -- (0.2,2.5) -- (1,1);
		\node at (0.35,1.5) {$\delta$};
		
		\draw (1,1) -- (2.5,3) -- (2,1);
		\node at (1.75,1.5) {$\gamma$};
		
		\draw (1,0) -- (4,0.5) -- (2,1);
		\node at (2.25,0.5) {$\beta$};

		%%%%% RELEVANT EDGES
		\draw [very thick, purple] (1,0) -- (1,1);
		\draw [very thick, orange] (1,0) -- (2,1);
		\draw [very thick, green!60!black] (1,1) -- (2,1);
		\draw [very thick, red] (1,1) -- (0,1);
		\draw[thick] (1,0) -- (0,1);

		%%%% VERTICES
		\node at (1,0) {\footnotesize $ \bullet $};
		\node at (1,-0.2) {$ c $};
		
		\node at (2,1) {\footnotesize $ \bullet $};
		\node at (2.2,1.2) {$ b $};
		
		\node at (1,1) {\footnotesize $ \bullet $}; 
		\node at (1,1.3) {$ d $}; 
		
		\node at (0,1) {\footnotesize $ \bullet $}; 
		\node at (-0.2,1.1) {$ a $};

		%%%% COORDINATE AXES 
		\draw[->] (-2,-0.75)--(-2,3.15);
		\draw[->] (-2,-0.75)--(4.15,-0.75);
	\end{tikzpicture}
	\end{center}
	\caption{Symbolic picture of the lotus  $ \mu_{[a,b]} (\Lambda ) $.}  
	\label{Fig:Lot_mutation} 
\end{figure}
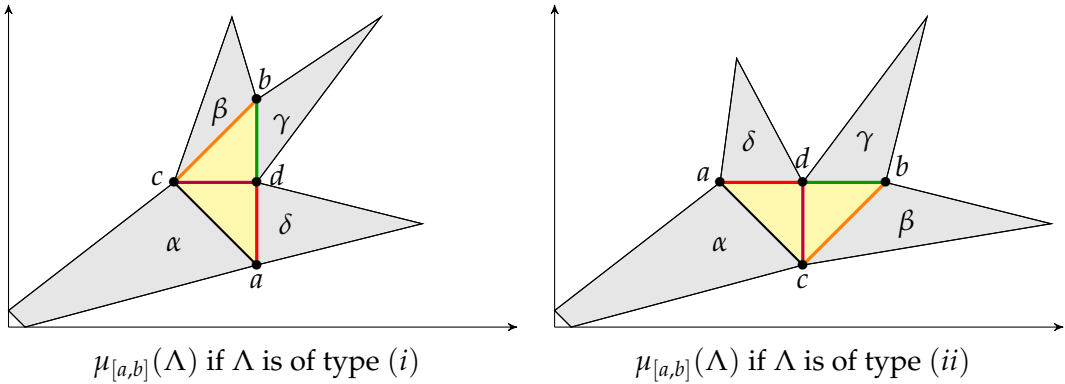     

\begin{proof}
	By the rule of the mutation for an inner diagonal (cf.~Fig.~\ref{Fig:mutation_diag}), 
	$ [ a, b ] $ is replaced by the new inner diagonal $ [c, d ] $. 
	Hence, the petals $ \delta(a,c) $ with third vertex $ b $ and $ \delta(a,b) $ with third vertex $ d $ are replaced by the 
	petal $ \delta(a,c) $ with third vertex $ d $ and $ \delta (c,d) $ with third vertex $ b $ respectively. 
	Clearly, the ordering of the vertices of the quadrilateral determined by $ \{ a, b ,c , d\} $ remains unchanged. 
	The other parts of the triangulation are untouched. 
	Therefore, the claim follows. 	
\end{proof} 

As an immediate consequence, we obtain the following formulas for the number of triangles incident to a fixed vertex of $ \mu_{[a,b] } (\Lambda) $,
which recalls the well-known formula for the change of the quiddity sequence along the mutation of the inner diagonal. 

\begin{cor}
	Set $ \Lambda' :=  \mu_{[a,b] } (\Lambda) $. 
	If we denote by $ \Delta(\Lambda, v) $ the number of triangles incident to a given vertex $ v $ of $ \Lambda $, then we have
	\[  
		\Delta(\Lambda', a) = \Delta(\Lambda, a) - 1 \ , \ \ \ \ 
		\Delta(\Lambda', b) = \Delta(\Lambda, b) - 1 \ , 
	\]
	\[ 
		\Delta(\Lambda', c) = \Delta(\Lambda, c) + 1 \ , \ \ \ \
		\Delta(\Lambda', d) = \Delta(\Lambda, d) + 1 \ ,
	\]	
	and
	$ \Delta(\Lambda', v) =  \Delta(\Lambda, v) $
	if $ v \notin \{ a, b, c, d\} $.  
\end{cor}

\begin{Bem}
	Clearly, the plane curves corresponding to $ \Lambda $ and to $ \mu_{[a,b] } (\Lambda) $ can be rather different.
	Nonetheless, the common parts $ \alpha, \beta, \gamma, \delta $ imply that their respective desingularizations are related.
	In order to see this, recall Remark \ref{Rk:charts}.  
	For example, the common $ \alpha $-part provides that the beginning of the resolution process coincide until we reach the chart corresponding to the edge $ [a, c] $.  
\end{Bem}

\section{Further questions}

We end by discussing further directions that could be considered as continuation. 

Up to now, we always considered Newton non-degenerate curves and their corresponding lotuses. 
There are plane curves that are not Newton non-degenerate and it is possible to define a lotus associated to them  by gluing disjoint unions of lotuses of the form as considered in the present article, e.g., see 
\cite[Def.~5.26]{Lotus}. 
Hence, it is reasonable to ask

\begin{Qu}
	How do the results of the present article generalize for (plane) curves that are not Newton non-degenerate? In particular, can we associate ``higher'' friezes to them?
\end{Qu}

First considerations suggest that re-embeddings resulting from which the curve becomes Newton non-degenerate in some higher-dimensional ambient space provide a suitable perspective. 
In this case, 
a natural candidate to substitute triangulations of polygons by triangles would be their analog using $ d $-simplices instead of triangles,
where $ d $ is the dimension of the eventual ambient space.  
One motivation for the re-embeddings is Teissier's perspective on the problem of local uniformization via overweight deformations, see \cite{Teissier, Teissier2, Teissier3}.
See also \cite{GoldinTeissier,AnaPedroHussein}, where the resolution of curves with one toric morphism is discussed,
and \cite{Tevelev}, where it is shown that whenever a desingularization exists then it can be realized as a single toric morphism.  
\\
Let us explain the reasons for this approach on an example.

\begin{example}
	Consider the plane curve $ C \subset \mathbb{A}_\CC^2 $ defined by the vanishing locus of 
	\[ 
		f = ( y^2  - x^3 )^5 - x^{14} y  \in \CC[x,y] \ .
	\] 
	The Newton polyhedron has a single compact edge with vertices $ (15,0) $ and $ (0,10) $.
	Since the restriction of $ f $ to this edge is $ ( x^2  - y^3 )^5 $,
	which is not smooth on the torus $ (\CC^*)_{x,y}^2 $ as it is not reduced,
	the curve is not Newton non-degenerate. 
	%\\
	We re-embed by introducing a new variable $ z $ fulfilling the relation
	$ z = y^2 - x^3 $, i.e.,
	we apply the isomorphism
	$ %\[
		\CC [x,y] \cong \CC[x,y,z]/ \langle z - y^2 + x^3 \rangle  	
	$ %\]
	and determine the image of $ C $ considered as curve in $ \mathbb{A}_\CC^3 $. 
	By substitution, the latter can be described by the relations
	\[
		\begin{array}{lcl}
			y^2 - x^3 & = & z \ ,
			\\[3pt]
			z^5 - x^{14} y & = &  0 \ .  
		\end{array} 
	\]
	By choosing the weights $ W(x) = 10, W(y) = 15, W(z) = 31 $, 
	it can be seen that this is an overweight deformation of the toric variety $ \mathcal{X} := V (y^2 - x^3,  z^5 - x^{14} y  ) $
	(for the definition of an overweight deformation, we refer to \cite[Def.~3.1]{Teissier2}).
	In particular, there is a close connection between the resolution of $ \mathcal{X} $ and that of $ C $.
	\\
	Since $ \mathcal{X} \subset \mathbb{A}_\CC^3 $ lives in a three-dimensional ambient space, 
	the underlying lattice $ N $ has rank $ 3 $ and
	the toric resolution of $ \mathcal{X} $ is a subdivision of the cone $ \sigma_0 := \RR_{\geq 0}^3 $.    
	Hence, it is reasonable to work with tetrahedra instead of triangles for the associated lotus.
	Furthermore, in the re-embedded situation, 
	the blowup of the closed point has three charts and not two,
	which also supports the approach via tetrahedra (cf.~Remark \ref{Rk:charts}).
\end{example}

In Def.~\ref{Def:lotus_d}, we propose a definition for the variant of the universal lotus in the setting of Newton non-degenerate curve singularities living in higher-dimensional ambient spaces. 
	Notice that this is a special case reflecting that the desingularization is obtained by blowing up closed points.
	In \cite[Section~9]{PPP-cerfvolant}, Popescu-Pampu introduced the definition of an arbitrary dimensional universal lotus. 
	The latter is far more general, but also more technical than Def.~\ref{Def:lotus_d}.

\begin{defi}
	\label{Def:lotus_d} 
	Let $ N $ be a lattice of rank $ d $ with a chosen basis $ ( e_1, \ldots, e_d ) $.
	We define the {\em base petal} as the convex and compact $ d $-simplex $ \delta (e_1, \ldots, e_d) \subseteq N_\RR $
	with vertices $ e_1, \ldots, e_d, e_1 + \cdots + e_d $.
	The points $ e_1, \ldots, e_d $ are called the {\em basic vertices} of the petal. 
	\\
	We construct more petals as follows: 
	Let $ v_1, \ldots, v_d $ be the vertices of any facet of the petal $ \delta (e_1, \ldots, e_d) $ for which $ e_1 + \cdots + e_d $ is one of the vertices. 
	Since $ v_1, \ldots, v_d $ is a basis for $ N $, the petal $ \delta(v_1, \ldots, v_d) $ is the one constructed from this basis.
	In other words, it is the $ d $-simplex with vertices $ v_1, \ldots, v_d, v_1 + \cdots + v_d $.
	\\
	By continued iteration of this process, we obtain an infinite simplicial complex in $ \sigma_0 := \langle e_1, \ldots, e_d \rangle_{\RR_{\geq 0}} \cong \RR_{\geq 0}^d $,
		which we call the {\em universal barycentric lotus $ \Lambda(e_1, \ldots, e_d) $} of $ N $ relative to the basis $ (e_1, \ldots, e_d) $. 
\end{defi}

Since this leads us to higher-dimensional lotuses, we could also pass from curves to singularities of dimension $ \geq 2 $. 
The notion of Newton non-degenerate singularities is not limited to curves. 
More generally, quasi-ordinary singularities are a natural class to be taken into account.
Given $ X \subseteq \mathbb{A}_\CC^n $ of dimension $ d $ it is called {\em quasi-ordinary} if there exists a projection $ X \to \mathbb{A}^d_\CC $ such that its discriminant locus is a simple normal crossing divisor. 
In~\cite{Bernd-Hussein}, the irreducible hypersurface case is considered. 
More precisely, there is a construction of re-embeddings using weighted polyhedra determined so that a connection as overweight deformation of an irreducible toric variety is deduced. 

Notice that, in view of Remark \ref{Rk:charts}, there are more possible centers than just closed points in higher dimension. 
This needs to be encountered, when investigating the higher dimensional case,
cf.~\cite[Section~9]{PPP-cerfvolant}.

\begin{Qu}
	How do the results of the present article generalize to higher dimensions?
\end{Qu}

\def\cprime{$'$}

\end{document}